\documentclass[11pt, reqno]{amsart}

\usepackage{etex}
\usepackage{enumitem}
\usepackage{setspace}
\usepackage{amsmath}

\usepackage{tabu}
\usepackage{tikz}
\usepackage{tikz-cd}
\usepackage[margin=1.2in]{geometry}
\usepackage[matrix,arrow,curve,frame]{xy}
\usepackage{hyperref}
\usepackage{amsthm}
\usepackage{textcomp}
\usepackage{amsfonts}
\usepackage{amssymb}
\usepackage{wasysym}
\usepackage{mathrsfs}
\usepackage{mathtools}
\usetikzlibrary{decorations.pathreplacing}
\usetikzlibrary{decorations.markings}
\usetikzlibrary{calc}

\definecolor{my-linkcolor}{rgb}{0.75,0,0}
\definecolor{my-citecolor}{rgb}{0.1,0.57,0}
\definecolor{my-urlcolor}{rgb}{0,0,0.75}
\hypersetup{
	colorlinks, 
	linkcolor={my-linkcolor},
	citecolor={my-citecolor}, 
	urlcolor={my-urlcolor}
}

\title[A general mirror equivalence theorem]{A general mirror equivalence theorem for coset vertex operator algebras}
 \author{Robert McRae}
\date{}
\address{Yau Mathematical Sciences Center, Tsinghua University, Beijing 100084, China}
\curraddr{}
\email{rhmcrae@tsinghua.edu.cn}


\numberwithin{equation}{section}

\theoremstyle{definition}\newtheorem{rema}{Remark}[section]
\theoremstyle{plain}\newtheorem{propo}[rema]{Proposition}
\newtheorem{theo}[rema]{Theorem}

\theoremstyle{definition}\newtheorem{defi}[rema]{Definition}
\theoremstyle{plain}\newtheorem{lemma}[rema]{Lemma}
\newtheorem{corol}[rema]{Corollary}
\theoremstyle{definition}
\theoremstyle{definition}
\theoremstyle{definition}
    
\theoremstyle{definition}\newtheorem{assum}[rema]{Assumption}

\newcommand{\cY}{\mathcal{Y}}
\newcommand{\cU}{\mathcal{U}}
\newcommand{\cV}{\mathcal{V}}
\newcommand{\cA}{\mathcal{A}}
\newcommand{\cR}{\mathcal{R}}
\newcommand{\cM}{\mathcal{M}}
\newcommand{\cS}{\mathcal{S}}
\newcommand{\cF}{\mathcal{F}}

\newcommand{\cO}{\mathcal{O}}
\newcommand{\cC}{\mathcal{C}}
\newcommand{\cG}{\mathcal{G}}
\newcommand{\cW}{\mathcal{W}}
\newcommand{\cL}{\mathcal{L}}

\newcommand{\til}{\widetilde}

\newcommand{\CC}{\mathbb{C}}
\newcommand{\ZZ}{\mathbb{Z}}
\newcommand{\NN}{\mathbb{N}}

\newcommand{\RR}{\mathbb{R}}

\newcommand{\Id}{\mathrm{Id}}

\newcommand{\tensU}{\boxtimes_U}
\newcommand{\tensV}{\boxtimes_V}
\newcommand{\tensC}{\boxtimes_{\mathcal{C}}}
\newcommand{\tensA}{\boxtimes_A}
\newcommand{\tens}{\boxtimes}
\newcommand{\vac}{\mathbf{1}}
\newcommand{\ind}{\mathrm{Ind}}

\DeclareMathOperator{\im}{Im}

\DeclareMathOperator{\rep}{Rep}

\let\ker\relax
\let\hom\relax
\DeclareMathOperator{\ker}{Ker}
\DeclareMathOperator{\hom}{Hom}
\DeclareMathOperator{\Endo}{End}

\newcommand{\even}{{\bar{0}}}
\newcommand{\odd}{{\bar{1}}}

\newcommand{\repA}{\rep A}

\begin{document}

\begin{abstract}
 We prove a general mirror duality theorem for a subalgebra $U$ of a simple conformal vertex algebra $A$ and its commutant $V=\mathrm{Com}_A(U)$. Specifically, we assume that $A\cong\bigoplus_{i\in I} U_i\otimes V_i$ as a $U\otimes V$-module, where the $U$-modules $U_i$ are simple and distinct and are objects of a semisimple braided ribbon category of $U$-modules, and the $V$-modules $V_i$ are semisimple and contained in a (not necessarily rigid) braided tensor category of $V$-modules. We also assume $U=\mathrm{Com}_A(V)$. Under these conditions, we construct a braid-reversed tensor equivalence $\tau: \mathcal{U}_A\rightarrow\mathcal{V}_A$, where $\mathcal{U}_A$ is the semisimple category of $U$-modules with simple objects $U_i$, $i\in I$, and $\mathcal{V}_A$ is the category of $V$-modules whose objects are finite direct sums of the $V_i$. In particular, the $V$-modules $V_i$ are simple and distinct, and $\mathcal{V}_A$ is a rigid tensor category. As an application, we find a rigid semisimple tensor subcategory of modules for the Virasoro algebra at central charge $13+6p+6p^{-1}$, $p\in\mathbb{Z}_{\geq 2}$, which is braided tensor equivalent to an abelian $3$-cocycle twist of the category of finite-dimensional $\mathfrak{sl}_2$-modules. Consequently, the Virasoro vertex operator algebra at central charge $13+6p+6p^{-1}$ is the $PSL_2(\mathbb{C})$-fixed-point subalgebra of a simple conformal vertex algebra $\mathcal{W}(-p)$, analogous to the realization of the Virasoro vertex operator algebra at central charge $13-6p-6p^{-1}$ as the $PSL_2(\CC)$-fixed-point subalgebra of the triplet algebra $\mathcal{W}(p)$.
\end{abstract}

\maketitle

\tableofcontents

\section{Introduction}

The philosophy behind this paper is that commuting actions of two algebraic structures on a third algebraic structure lead to interesting relations among the representation theories of the three algebraic structures under consideration. A classical example is Schur-Weyl duality relating representations of the general linear and symmetric groups. In this work, all three algebraic structures will be vertex operator algebras, which are an approach to the rigorous mathematical study of two-dimensional chiral conformal quantum field theories.

Our specific setting is a vertex operator algebra $A$ (or more generally, a conformal vertex algebra) which contains two vertex subalgebras $U$ and $V$ whose vertex operator actions on $A$ commute with each other and whose conformal vectors add up to the conformal vector of $A$, that is, $A$ contains a vertex operator subalgebra isomorphic to $U\otimes V$. In the language of conformal field theory, $V$ will be the coset of $U$ in $A$ and $U$ will be the 
coset of $V$ in $A$. Our main result is that under fairly general conditions, (sub)categories of $U$- and $V$-modules enjoy a kind of ``mirror duality,'' that is, there is a braid-reversed tensor equivalence relating the $U$-modules that occur in $A$ with the $V$-modules that occur in $A$. Such a result has been proved earlier under stronger assumptions \cite{Lin, CKM2}. Here, although we do assume that the representation theory of $U$ is quite nice, we make fairly minimal assumptions on $V$, mainly just that $V$ has a suitable representation category that admits the vertex algebraic braided tensor category structure of \cite{HLZ1}-\cite{HLZ8}. In full detail, our main theorem is as follows (see Theorem \ref{thm:mirror_equiv} in the main text):

\begin{theo}\label{thm:main_thm}
 Let $U$ and $V$ be simple self-contragredient vertex operator algebras such that:
 \begin{itemize}
  \item There is an injective conformal vertex algebra homomorphism $\iota_A: U\otimes V\rightarrow A$, where $A$ is a simple conformal vertex algebra.
  \item As a $U\otimes V$-module,
  \begin{equation*}
   A\cong\bigoplus_{i\in I} U_i\otimes V_i
  \end{equation*}
where the $U_i$ are distinct simple $U$-modules and the $V_i$ are semisimple $V$-modules. Let $0\in I$ denote the index such that $U_0=U$.

\item The vertex subalgebras $U$ and $V$ form a dual pair in $A$ in the sense that $V_0=V$ and
\begin{equation*}
 \dim \hom_V(V,V_i)=\delta_{i,0}
\end{equation*}
for all $i\in I$.

\item The $U$-modules $U_i$ for $i\in I$ are objects of a locally-finite semisimple braided ribbon tensor category $\cU$ of $U$-modules that is closed under contragredients.

\item The $V$-modules $V_i$ for $i\in I$ are objects of a braided tensor category $\cV$ of grading-restricted generalized $V$-modules which is closed under submodules, quotients, and contragredients.
 \end{itemize}
If $I$ is infinite, let $\cC$ be the category of $U\otimes V$-modules whose objects are finite direct sums of modules $M\otimes W$ for $M$ an object of $\cU$ and $W$ an object of $\cV$, and assume in addition that:
 \begin{itemize}
  \item Every $V$-module in $\cV$ is finitely generated.
  \item For every intertwining operator $\cY$ of type $\binom{X_3}{M_1\otimes W_1\,M_2\otimes W_2}$, where $M_1$ and $M_2$ are $U$-modules in $\cU$, $W_1$ and $W_2$ are $V$-modules in $\cV$, and $X_3$ is a generalized $U\otimes V$-module which is the union of submodules in $\cC$, the image $\im\cY\subseteq X_3$ is an object of $\cC$.
 \end{itemize}
 Let $\cU_A$ (respectively, $\cV_A$) denote the category of $U$-modules (respectively, $V$-modules) whose objects are finite direct sums of the $U_i$ (respectively, of the $V_i$) for $i\in I$. Then:
 \begin{enumerate}
  \item The category $\cU_A$ of $U$-modules is a tensor subcategory of $\cU$.
  
  \item The category $\cV_A$ of $V$-modules is a semisimple ribbon tensor subcategory of $\cV$ with distinct simple objects $\lbrace V_i\rbrace_{i\in I}$. In particular, $\cV_A$ is rigid.
  
  \item There is a braid-reversed tensor equivalence $\tau: \cU_A\rightarrow\cV_A$ such that $\tau(U_i)\cong V_i'$ for $i\in I$, where $V_i'$ is the contragredient of $V_i$.
 \end{enumerate}
\end{theo}

See Theorem \ref{thm:superalg_mirror_equiv} for the generalization to the case that $A$ is a conformal vertex superalgebra, in which case $\cU_A$ and $\cV_A$ are only braid-reversed equivalent up to certain sign factors in the braiding isomorphisms. In case the index set $I$ is infinite in Theorem \ref{thm:main_thm}, the intertwining operator condition in terms of the category $\cC$ of $U\otimes V$-modules can be replaced by sufficient conditions depending only on $\cU$ and $\cV$ individually (see Proposition \ref{prop:IndC_intw_op_cond}).

The asymmetry between the assumptions on $\cU$ and $\cV$ in Theorem \ref{thm:main_thm} sets this result apart from earlier work, especially part (2) of \cite[Main Theorem 3]{CKM2}. While we assume here that the category $\cU$ of $U$-modules has many of the nice properties of the finite-dimensional representation theory of a compact group, especially semisimplicity and rigidity (existence of duals in a strong sense), we emphatically do not assume rigidity for $\cV$. Although rigidity is highly useful in the study of tensor categories, it is notoriously difficult to prove for vertex algebraic tensor categories: so far, almost the only general result is Huang's rigidity theorem \cite{Hu-rig-mod} for strongly rational vertex operator algebras, so that for non-rational examples, rigidity usually must be proved on a case-by-case basis.

Consequently, we envision Theorem \ref{thm:main_thm} as a tool for transferring nice properties, especially rigidity, from the representation category of a well-understood vertex operator algebra $U$ to that of a new vertex operator algebra $V$. Such results have been achieved in the easier partially classical setting of commuting actions of a compact group $G$ and its fixed-point vertex operator subalgebra $V^G$ on a simple vertex operator algebra $V$. In this case, the analogue of Theorem \ref{thm:main_thm}, proved in \cite{McR}, states that the category of finite-dimensional $G$-modules is tensor equivalent to the semisimple subcategory of $V^G$-modules whose simple objects occur as $V^G$-submodules of $V$. In \cite{McR, MY}, this theorem was used to transfer rigidity from finite-dimensional $SU(2)$-modules to many irreducible modules for Virasoro vertex operator algebras at central charges $13-6p-6p^{-1}$, $p\in\ZZ_+$.

We expect that it will be possible to establish rigidity for many more vertex operator algebras using cosets and Theorem \ref{thm:main_thm}. In fact, since the first version of this paper was completed, Creutzig, Genra, and Linshaw used Theorem \ref{thm:main_thm} to prove rigidity for the grading-restricted module category of the simple affine vertex operator algebra $L_k(\mathfrak{sp}_{2n})$ at principal and coprincipal admissible levels $k$ \cite[Corollary 4.2(1)]{CGL}. They proved this by realizing $L_k(\mathfrak{sp}_{2n})$ as the coset of a rational $W$-algebra of type $C$ inside an admissible-level affine vertex operator superalgebra of type $\mathfrak{osp}_{1\vert 2n}$. Previously,  tensor categories for affine vertex operator algebras at admissible levels were known to be rigid only in simply-laced types \cite{CHY, Cr1} and a few cases in type $B$ \cite{CKoL}.

In this paper, we apply Theorem \ref{thm:main_thm} to the Virasoro vertex operator algebras $V_{13+6p+6p^{-1}}$ at central charges $13+6p+6p^{-1}$, $p\in\ZZ_{\geq 2}$. It was shown in \cite{CJORY} that the category $\cO_{-p}$ of $C_1$-cofinite grading-restricted generalized $V_{13+6p+6p^{-1}}$-modules admits the braided tensor category structure of \cite{HLZ1}-\cite{HLZ8}. So far, the detailed structure of $\cO_{-p}$ is unknown; in particular, it is not known if $\cO_{-p}$ is rigid. However, it is shown in \cite{Ar} that there is a simple conformal vertex algebra extension of $V_{13-6p-6p^{-1}}\otimes V_{13+6p+6p^{-1}}$, called the chiral universal centralizer of $SL_2$ at level $k=-2+\frac{1}{p}$, that satisfies the conditions of Theorem \ref{thm:main_thm}. Moreover, $V_{13-6p-6p^{-1}}$ admits a suitable rigid tensor category of modules \cite{MY}. Thus applying Theorem \ref{thm:main_thm}, we obtain (see Theorem \ref{thm:main_thm_app} and Corollary \ref{cor:O_-p_L_fus_rules} in the main text):
\begin{theo}\label{thm:intro_Vir_app}
For $p\in\ZZ_{\geq 2}$, let $\cO_{\pm p}^L$ denote the semisimple categories of $V_{13\mp 6p \mp 6p^{-1}}$-modules with simple objects $\cL_{r,1}^{\pm p}$, $r\in\ZZ_+$, where $\cL_{r,1}^{\pm p}$ has lowest conformal weight $\pm\frac{r^2-1}{4}p-\frac{r-1}{2}$.
\begin{enumerate}
\item The category $\cO_{-p}^L$ is a rigid tensor subcategory of $\cO_{-p}$, and there is a braid-reversed tensor equivalence $\tau:\cO_p^L\rightarrow\cO_{-p}^L$ such that $\tau(\cL_{r,1}^p)\cong\cL_{r,1}^{-p}$ for all $r\in\ZZ_+$.

\item The following tensor product formula holds in $\cO_{-p}$: for $r,r'\in\ZZ_+$,
\begin{equation*}
 \cL_{r,1}^{-p}\tens\cL_{r',1}^{-p}\cong\bigoplus_{\substack{k=\vert r-r'\vert+1\\ k+r+r'\equiv 1\,(\mathrm{mod}\,2)}}^{r+r'-1} \cL_{k,1}^{-p}.
\end{equation*}
\end{enumerate}
\end{theo}

The $p=1$ version of this theorem has already been proved in \cite{MY2}.  For $p\geq 2$, not all simple objects of $\cO_{-p}$ appear in $\cO_{-p}^L$, so Theorem \ref{thm:intro_Vir_app} is just the beginning of a detailed understanding of the structure of $\cO_{-p}$, which will be completed in future work. The $\mathfrak{sl}_2$-type fusion rules in Theorem \ref{thm:intro_Vir_app}(2) follow from the fact (proved in \cite[Theorem 4.3]{MY}) that $\cO_p^L$ is braided tensor equivalent to an abelian $3$-cocycle twist of the category $\rep\mathfrak{sl}_2$ of finite-dimensional $\mathfrak{sl}_2$-modules. Thus by Theorem \ref{thm:intro_Vir_app}(1), $\cO_{-p}^L$ is also an abelian $3$-cocycle twist of $\rep\mathfrak{sl}_2$ (with the inverse braiding as for $\cO_p^L$).

As an application of Theorem \ref{thm:intro_Vir_app}, we construct new simple conformal vertex algebras at central charge $13+6p+6p^{-1}$: by applying the braid-reversed tensor equivalence $\cO_p^L\rightarrow\cO_{-p}^L$ to the triplet $W$-algebra $\cW(p)$ \cite{AdM} and the singlet $W$-algebra $\cM(p)$ \cite{Ad}, which are simple vertex operator algebra extensions of $V_{13-6p-6p^{-1}}$, we get corresponding conformal vertex algebra extensions of $V_{13+6p+6p^{-1}}$ (see Theorem \ref{thm:W(-p)} and Proposition \ref{prop:M(-p)} in the main text):
\begin{theo}\label{thm:intro_W(-p)_M(-p)}
Let $p\in\ZZ_{\geq 2}$, and let $V(n)$ be the $(n+1)$-dimensional simple $\mathfrak{sl}_2$-module.
\begin{enumerate}
\item There is a unique simple conformal vertex algebra $\cW(-p)$ of central charge $13+6p+6p^{-1}$ such that as a $V_{13+6p+6p^{-1}}$-module,
\begin{equation}\label{eqn:intro_W(-p)_decomp}
\cW(-p)\cong\bigoplus_{n=0}^\infty V(2n)\otimes\cL_{2n+1,1}^{-p}.
\end{equation}
Moreover, $PSL_2(\CC)$ acts on $\cW(-p)$ by conformal vertex algebra automorphisms such that \eqref{eqn:intro_W(-p)_decomp} also gives the decomposition of $\cW(-p)$ as a $PSL_2(\CC)$-module.

\item  There is a simple conformal vertex algebra $\cM(-p)$ of central charge $13+6p+6p^{-1}$ such that as a $V_{13+6p+6p^{-1}}$-module,
\begin{equation*}
\cM(-p)\cong\bigoplus_{n=0}^\infty \cL_{2n+1,1}^{-p}.
\end{equation*}
\end{enumerate}
\end{theo}

Part (1) of this theorem shows in particular that $V_{13+6p+6p^{-1}}$ is the $PSL_2(\CC)$-fixed-point subalgebra of a simple conformal vertex algebra. Again, the $p=1$ version was proved in \cite{MY2}, where we also showed that $\cW(-1)$ has a tensor category of modules that is braid-reversed tensor equivalent to the modular tensor category of modules for the $\mathfrak{sl}_2$-root lattice vertex operator algebra. For $p\geq 2$, we conjecture that the category of $\cW(-p)$-modules is braid-reversed tensor equivalent to that of $\cW(p)$-modules, and that $\cM(-p)$ admits a tensor category of modules which is braid-reversed equivalent to a certain category of $\cM(p)$-modules constructed in \cite{CMY-sing}. These categories of $\cW(p)$- and $\cM(p)$-modules, in turn, are tensor equivalent to module (sub)categories for suitable versions of quantum $\mathfrak{sl}_2$ at the root of unity $q=e^{\pi i/p}$ \cite{GN}. Proving these conjectures for $\cW(-p)$ and $\cM(-p)$ will require a fuller understanding of $\cO_{-p}$ and its relation to $\cO_p$.

In the future, it may be possible to apply Theorem \ref{thm:main_thm} to the higher-rank chiral universal centralizer algebras of \cite[Section 7]{Ar} to generalize Theorem \ref{thm:intro_Vir_app} to higher-rank $W$-algebras. In particular, it may be possible to relate (sub)categories of modules for higher-rank affine $W$-algebras at certain positive shifted levels with (sub)categories of modules at corresponding negative levels. However, this will first require proving the existence of tensor category structure on suitable module categories for higher-rank affine $W$-algebras, as well as proving rigidity for suitable modules at positive shifted levels.

We now discuss the proof of Theorem \ref{thm:main_thm}. We construct a functor from $\cU$ to $\cV$ as follows: First, given a $U$-module $M$ which is an object of $\cU$, the vector space tensor product $M\otimes V$ is an object of the Deligne product tensor category $\cC=\cU\tens\cV$. Now since $A$ is a commutative algebra in $\cC$ (or its ind-completion if $A$ is an infinite-order extension of $U\otimes V$) by \cite{HKL}, we have an induction functor from $\cC$ to a category $\rep A$ of $A$-modules (in a generalized sense) in $\cC$; see \cite{KO,CKM1} for details. Induction takes $M\otimes V$ to the $U\otimes V$-module $A\tens_\cC(M\otimes V)$, which admits a left $A$-action induced by the vertex operator for $A$ and associativity of the tensor product on $\cC$. Finally, we obtain a $V$-module (in $\cV$) from $A\tens_\cC (M\otimes V)$ by taking $U$-invariants, which is the subspace of $U$-vacuum-like vectors in $A\tens_\cC(M\otimes V)$; equivalently this is the kernel of the Virasoro operator $L_U(-1)$ associated to $U$ acting on $A\tens_\cC(M\otimes V)$. We let $\Phi$ denote the functor from $\cU$ to $\cV$ which takes $M$ to the $U$-invariants of $A\tens_\cC(M\otimes V)$.

Using the universal property of vertex algebraic tensor products, given in terms of intertwining operators, it is straightforward to show that there is a functorial homomorphism
\begin{equation*}
 J_{M_1,M_2}: \Phi(M_1)\tens_V\Phi(M_2)\longrightarrow\Phi(M_1\tens_U M_2)
\end{equation*}
for $U$-modules $M_1$, $M_2$ in $\cU$ that is compatible with the unit and associativity isomorphisms in $\cU$ and $\cV$. That is, $\Phi$ is a lax monoidal functor from $\cU$ to $\cV$. Most of the proof of Theorem \ref{thm:main_thm} then focuses on showing that $J_{M_1,M_2}$ is actually an isomorphism when $M_1$ is an object of the subcategory $\cU_A$ of $U$-modules whose summands appear in $A$, so that $\Phi$ defines a strong monoidal functor from $\cU_A$ to $\cV_A$.

To prove that $J_{M_1,M_2}$ is injective when $M_1$ is an object of $\cU_A$, the idea is to use the same categorical argument as was used in \cite[Theorem 4.7]{McR} to prove the analogous result for fixed-point vertex operator subalgebras of compact automorphism groups. However, this proof in \cite{McR} used surjectivity of $J_{M_1,M_2}$, which is not so easy to prove directly in the coset setting. Nevertheless, the argument still works as long as $J_{M_1,M_1'}$, where $M_1'$ is the contragredient of $M_1$, is ``surjective enough''  in the sense that the image of $J_{M_1,M_1'}$ in $\Phi(M_1\tens_U M_1')$ contains the copy of $\Phi(U)\cong V$ associated to the coevaluation $i_{M_1}: U\rightarrow M_1\tens_U M_1'$. The proof of this result on $\mathrm{Im}\,J_{M_1,M_1'}$ (in Lemma \ref{lem:J_surj_enough} below) is primarily categorical in nature and is one of the most technical parts of the proof of Theorem \ref{thm:main_thm}.

Once we have proved that $J_{M_1,M_1'}$ is surjective enough, the argument of \cite[Theorem 4.7]{McR} shows that $J_{M_1,M_2}$ is injective. Then a similar categorical argument, using injectivity, shows that $J_{M_1,M_2}$ is also surjective and thus an isomorphism. This proves that $\Phi:\cU_A\rightarrow\cV_A$ is a strong monoidal functor, and using this, we show that $\cV_A$ is rigid. We could then continue to prove that $\Phi$ is a braid-reversed tensor equivalence, and in particular is suitably compatible with the (inverse) braiding isomorphisms in $\cU_A$ and $\cV_A$. However, at this point, it is more convenient to simply quote a weaker version of Theorem \ref{thm:main_thm}, which was proved in \cite{CKM2} under the \textit{a priori} assumption that $\cV_A$ is rigid, to obtain the braid-reversed equivalence.

The remainder of this paper is structured as follows. Section \ref{sec:prelim} recalls basic definitions, notations, and results in tensor categories and vertex operator algebras. Subsection \ref{subsec:cosets} in particular contains some general results on how $U\otimes V$-modules decompose as modules for $U$ and $V$, and especially on how a vertex operator algebra $A$ decomposes as a module for mutual commutant subalgebras $U$ and $V$. Although these results are more or less known, for convenience we provide proofs in Appendix \ref{app:gen_coset_proofs}. In Section \ref{sec:setting}, we state our basic assumptions on the vertex operator algebras $U$, $V$, and $A$, and we review the theory of (possibly infinite-order) vertex operator algebra extensions \cite{CKM1, CMY} as it pertains to coset-type extensions. The technical Proposition \ref{prop:IndC_intw_op_cond}, on the conditions needed to apply Theorem \ref{thm:main_thm} in the infinite-order extension case, is proved in Appendix \ref{app:intw_op_proof}. We prove Theorem \ref{thm:main_thm} in Section \ref{sec:proofs}; we generalize Theorem \ref{thm:main_thm} to the case that $A$ is a vertex operator superalgebra in Subsection \ref{subsec:VOSAs}. Finally, in Section \ref{sec:Vir_app}, we apply Theorem \ref{thm:main_thm} to Virasoro tensor categories, proving Theorems \ref{thm:intro_Vir_app} and \ref{thm:intro_W(-p)_M(-p)}.

\medskip

\noindent\textbf{Acknowledgments.} I would like to thank Shashank Kanade for discussions on the invariants functor in Subsection \ref{subsec:inv_functor}. I would like to thank Thomas Creutzig for the idea used to handle superalgebra extensions in Subsection \ref{subsec:VOSAs} and for informing me about the application of Theorem \ref{thm:main_thm} to admissible-level affine tensor categories in type $C$, which has now appeared in \cite{CGL}. I would also like to thank Tomoyuki Arakawa for informing me about his construction of the chiral universal centralizer algebras, which motivated the addition of Theorems \ref{thm:intro_Vir_app} and \ref{thm:intro_W(-p)_M(-p)} to this paper.

\section{Preliminaries}\label{sec:prelim}

In this section, we discuss basic notions and results from tensor categories, tensor categories of modules for a vertex operator algebra, and coset vertex operator algebras.

\subsection{Tensor categories}

For references on tensor categories, see for example \cite{Ka,BK,EGNO}, but note that these references use the term ``tensor category'' for somewhat variant notions. In this paper, ``tensor category'' will mean a monoidal category $\cC$ with tensor product bifunctor $\tens$ and unit object $\vac$ such that:
\begin{itemize}
 \item The category $\cC$ is $\mathbb{F}$-linear abelian where $\mathbb{F}=\Endo_\cC(\vac)$ is a field.
 
 \item The tensor product bifunctor $\tens$ is bilinear, and
 
 \item For any object $X$ in $\cC$, the functors $\bullet\tens X$ and $X\tens\bullet$ are right exact.
\end{itemize}
For the tensor categories in this paper, the field $\mathbb{F}$ will be $\CC$. We will use $l$ and $r$ to denote the left and right unit isomorphisms of a tensor category and $\cA$ to denote the natural associativity isomorphisms. If a tensor category is braided, we will use $\cR$ to denote the natural braiding isomorphisms. In this preliminary section, we will use $\vac$ to denote the unit object of a tensor category; in future sections, the unit object will be a vertex operator algebra denoted by a suitable letter and $\vac$ will denote the vacuum vector of the vertex operator algebra.

A \textit{left dual} of an object $X$ in a tensor category is an object $X^*$ together with an \textit{evaluation} morphism $e_X: X^*\tens X\rightarrow\vac$ and a \textit{coevaluation} morphism $i_X: \vac\rightarrow X\tens X^*$ such that the rigidity axioms hold: the composition
\begin{equation*}
 X\xrightarrow{l_X^{-1}} \vac\tens X\xrightarrow{i_X\tens\Id_X} (X\tens X^*)\tens X\xrightarrow{\cA_{X,X^*,X}^{-1}} X\tens(X^*\tens X)\xrightarrow{\Id_X\tens e_X} X\tens\vac\xrightarrow{r_X} X
\end{equation*}
equals $\Id_X$ and the composition
\begin{equation*}
 X^*\xrightarrow{r_{X^*}^{-1}} X^* \tens \vac\xrightarrow{\Id_{X^*}\tens i_{X}} X^*\tens (X\tens X^*)\xrightarrow{\cA_{X^*,X,X^*}} (X^*\tens X)\tens X^*\xrightarrow{e_{X}\tens\Id_{X^*}} \vac\tens X^*\xrightarrow{l_{X^*}} X^*
\end{equation*}
equals $\Id_{X^*}$. There is also a notion of right dual, but we will not need this because we will only consider tensor categories in which left duals and right duals are isomorphic. From now on, we shall call left duals simply \textit{duals}. A tensor category is \textit{rigid} if every object has a dual.

We will need to consider tensor categories that are not rigid, or at least are not known \textit{a priori} to be rigid. However, objects in these tensor categories will still satisfy a partial version of rigidity:
\begin{defi}\label{def:contragredient}
 Given an object $X$ in an $\mathbb{F}$-linear tensor category $\cC$, we say that $X'$ is a \textit{contragredient} of $X$ if there is a natural isomorphism
\begin{equation*}
 \hom_\cC(\bullet\tens X,\vac)\xrightarrow{\Gamma} \hom_\cC(\bullet, X')
\end{equation*}
of contravariant functors from $\cC$ to $\cV ec_\mathbb{F}$.
\end{defi}
 Thus if $X'$ is a contragredient of an object $X$ in $\cC$, then there is a linear isomorphism $\Gamma_Y: \hom_\cC(Y\tens X,\vac)\rightarrow\hom_\cC(Y,X')$ for any object $Y$ in $\cC$ such that for any morphism $f: Y\rightarrow\til{Y}$ in $\cC$, the diagram
\begin{equation}\label{eqn:contra_nat_iso}
 \xymatrixcolsep{4pc}
\begin{matrix}
  \xymatrix{
 \hom_\cC(Y\tens X,\vac) \ar[r]^{\Gamma_Y} & \hom_\cC(Y,X') \\
 \hom_\cC(\til{Y}\tens X,\vac) \ar[r]^{\Gamma_{\til{Y}}} \ar[u]_{G\mapsto G\circ(f\tens\Id_X)} & \hom_\cC(\til{Y},X') \ar[u]_{g\mapsto g\circ f}\\
 }
\end{matrix}
\end{equation}
commutes. If $X$ has a dual, then $X^*$ is a contragredient with $\Gamma$ contructed using the coevaluation: for a morphism $G: Y\tens X\rightarrow\vac$ in $\cC$, $\Gamma_Y(G)$ is the composition
\begin{equation*}
 Y\xrightarrow{r_Y^{-1}} Y\tens\vac\xrightarrow{\Id_Y\tens i_X} Y\tens(X\tens X^*)\xrightarrow{\cA_{Y,X,X^*}} (Y\tens X)\tens X^*\xrightarrow{G\tens\Id_{X^*}} \vac\tens X^*\xrightarrow{l_{X^*}} X^*.
\end{equation*}
Rigidity then implies $\Gamma_{X^*}(e_{X}) =\Id_{X^*}$. In general, if $X'$ is a contragredient of $X$, we can define an evaluation morphism
\begin{equation*}
 e_X =\Gamma_{X'}^{-1}(\Id_{X'}): X'\tens X\rightarrow\vac,
\end{equation*}
but we do not necessarily have a coevaluation.

The pair $(X', e_X)$ satisfies the following universal property: for any $G: Y\tens X\rightarrow\vac$ in $\cC$, there is a unique morphism $g: Y\rightarrow X'$, namely, $g=\Gamma_Y(G)$, such that the diagram
\begin{equation*}
 \xymatrixcolsep{3pc}
\xymatrix{
 Y\tens X \ar[rd]^{G} \ar[d]_{g\tens\Id_X} & \\
 X'\tens X \ar[r]_(.6){e_X} & \vac \\
}
\end{equation*}
commutes. To see this, note that for any $g: Y\rightarrow X'$, the $\til{Y}=X'$ and $f=g$ case of  \eqref{eqn:contra_nat_iso} is
\begin{equation*}
 \xymatrixcolsep{4pc}
 \xymatrix{
 \hom_\cC(Y\tens X,\vac) \ar[r]^{\Gamma_Y} & \hom_\cC(Y,X') \\
 \hom_\cC(X'\tens X,\vac) \ar[r]^{\Gamma_{X'}} \ar[u]_{H\mapsto H\circ(g\tens\Id_X)} & \hom_\cC(X',X') \ar[u]_{h\mapsto h\circ g}\\
 } .
\end{equation*}
Commutation of this diagram implies
\begin{align*}
 e_X\circ(g\tens\Id_X) & = \Gamma_{X'}^{-1}(\Id_{X'})\circ(g\tens\Id_X)\nonumber\\
 & = \Gamma_Y^{-1}(\Id_{X'}\circ g)\nonumber\\
 & = \Gamma_Y^{-1}(g),
\end{align*}
 so that $e_X\circ(g\tens\Id_X)=G$ if and only if $\Gamma_Y^{-1}(g)=G$, or $g=\Gamma_Y(G)$. 

The universal property of $(X', e_X)$ means that the contragredient is unique up to unique isomorphism. Moreover, if every object of $\cC$ has a contragredient, we get a contravariant endofunctor on $\cC$: given a morphism $f: X\rightarrow Y$ in $\cC$, $f'$ is the unique morphism, guaranteed by the universal property of contragredients, making the diagram
\begin{equation*}
 \xymatrixcolsep{4pc}
 \xymatrix{
 Y'\tens X \ar[r]^{\Id_{Y'}\tens f} \ar[d]^{f'\tens\Id_{X}} & Y'\tens Y \ar[d]^{e_Y} \\
 X'\tens X \ar[r]^{e_X} & \vac \\
 }
\end{equation*}
commute. To show that this indeed defines a functor, first note that the diagram in the case $f=\Id_X$ shows that $\Id_X'$ must be $\Id_{X'}$. Then to show that $(f\circ g)' = g'\circ f'$ for morphisms $f: Y\rightarrow Z$ and $g: X\rightarrow Y$, we need to show that the diagram
\begin{equation*}
\xymatrixcolsep{4pc}
\xymatrix{
 Z'\tens X \ar[r]^{\Id_{Z'}\tens(f\circ g)} \ar[d]^{(g'\circ f')\tens\Id_X} & Z'\tens Z \ar[d]^{e_Z} \\
 X'\tens X \ar[r]^{e_X} & \vac \\
}
 \end{equation*}
commutes. This follows from the commutativity of the diagram
\begin{equation*}
 \xymatrixcolsep{4pc}
 \xymatrix{
 Z'\tens X \ar[r]^{\Id_{Z'}\tens g} \ar[d]^{f'\tens\Id_X} & Z'\tens Y \ar[d]^{f'\tens\Id_X} \ar[rd]^{\Id_{Z'}\tens f} & \\
 Y'\tens X \ar[r]^{\Id_{Y'}\tens g} \ar[d]^{g'\tens\Id_X} & Y'\tens Y \ar[d]^{e_Y} & Z'\tens Z \ar[ld]^{e_Z} \\
 X'\tens X \ar[r]^{e_X} & \vac & \\
 }
\end{equation*}

\begin{rema}\label{rem:r-category}
 A tensor category in which every object has a contragredient is called an \textit{$r$-category} in \cite{BD}, a special case what is called a \textit{Grothendieck-Verdier category}.
\end{rema}

A ribbon category is a rigid tensor category $\cC$ with a natural isomorphism $\delta: \Id_{\cC}\rightarrow **$ satisfying certain natural conditions (see for example \cite[Section 2.2]{BK} for details). If $\cC$ is a braided tensor category, a ribbon structure $\delta$ is equivalent to a twist, which we can define more generally for braided tensor categories with contragredients. If $\cC$ is a braided tensor category with a contragredient functor, then a \textit{twist} is a natural isomorphism $\theta:\Id_\cC\rightarrow\Id_\cC$ that satisfies $\theta_\vac=\Id_\vac$, $\theta_{X'}=\theta_X'$ for any object $X$, and the \textit{balancing equation}
\begin{equation*}
 \theta_{X\tens Y} =\cR_{Y,X}\circ\cR_{X,Y}\circ(\theta_X\tens\theta_Y)
\end{equation*}
for all objects $X$, $Y$. Given such a natural isomorphism $\theta$, we can define $\delta_X: X\rightarrow X''$ for any object $X$ to be the unique morphism making the diagram
\begin{equation*}
 \xymatrixcolsep{5.5pc}
 \xymatrix{
 X\tens X' \ar[r]^{\cR_{X,X'}\circ(\theta_X\tens\Id_{X'})} \ar[d]^{\delta_X\tens\Id_{X'}} & X'\tens X \ar[d]^{e_X} \\
 X''\tens X' \ar[r]^{e_{X'}} & \vac \\
 }
\end{equation*}
commute. If $\cC$ is rigid, the natural transformation $\delta$ defined in this way is a natural isomorphism and a ribbon structure (again see for example \cite[Section 2.2]{BK}). In any ribbon category $\cC$, we can define the \textit{categorical dimension} 
\begin{equation*}
 \dim_\cC X = e_{X'}\circ(\delta_X\tens\Id_{X'})\circ i_X\in\Endo_\cC(\vac)=\mathbb{F}.
\end{equation*}
for any object $X$. If $\cC$ is braided, we can equivalently define $\dim_\cC X$ by the composition
\begin{equation*}
 \vac\xrightarrow{i_X} X\tens X^*\xrightarrow{\theta_X\tens\Id_{X^*}} X\tens X^*\xrightarrow{\cR_{X,X^*}} X^*\tens X\xrightarrow{e_X} \vac.
\end{equation*}

\subsection{Vertex operator algebras}\label{subsec:VOAs}

The tensor categories in this paper will be module categories for vertex operator algebras. We use the definition of vertex operator algebra from, for example, \cite{FLM, FHL, LL}. More generally, we will sometimes consider conformal vertex algebras as defined in \cite{HLZ1}, which may have infinite-dimensional conformal weight spaces and might not have a lower bound on conformal weights. 

We use the definition of \textit{generalized module} for a vertex operator algebra $V$ from \cite{HLZ1}. In particular, a generalized $V$-module $W$ has a grading $W=\bigoplus_{h\in\CC} W_{[h]}$ where $W_{[h]}$ is the generalized eigenspace for the Virasoro operator $L(0)$ on $W$ with generalized eigenvalue $h$. We say that $W$ is \textit{lower bounded} if $W_{[h]}=0$ for $\mathrm{Re}\,h$ sufficiently negative, and we say that $W$ is \textit{grading restricted} if for each $h\in\CC$, $W_{[h]}$ is finite dimensional and $W_{[h-n]}=0$ for $n\in\NN$ sufficiently large. The definition of grading-restricted generalized $V$-module differs from the definition of $V$-module in \cite{LL} only in that the $W_{[h]}$ are allowed to be generalized $L(0)$-eigenspaces rather than eigenspaces. But since we always assume the possibility of generalized $L(0)$-eigenspaces in this paper, we will typically use the term \textit{module} to refer to grading-restricted generalized modules. We also use the term \textit{weak module} to refer to a module for a vertex operator algebra considered as a vertex algebra, as in \cite[Definition 4.1.1]{LL}. That is, we make no grading assumptions for weak modules.

Given a  lower-bounded or grading-restricted generalized module $W=\bigoplus_{h\in\CC} W_{[h]}$ for a vertex operator algebra $V$, \cite{FHL} shows that the graded dual $W'=\bigoplus_{h\in\CC} W_{[h]}^*$ has a $V$-module structure called the contragredient of $W$. We will see later that if $W$ is an object of a braided tensor category of modules, then under mild conditions, $W'$ is a contragredient in the sense of Definition \ref{def:contragredient}.

The notion of vertex tensor category \cite{HL1, HL2} was developed by Huang-Lepowsky, and most recently by Huang-Lepowsky-Zhang \cite{HLZ1}-\cite{HLZ8}, to describe the appropriate tensor structure on module categories for a vertex operator algebra. A key feature is the existence of a tensor product bifunctor for every element of the moduli space of Riemann spheres with two positively-oriented punctures, one negatively-oriented puncture, and local coordinates at the punctures. In order to describe the braided tensor category structure derived from a vertex tensor category, it is sufficient to consider only the $P(z)$-tensor product bifunctors for $z\in\CC^\times$, where $P(z)$ denotes the Riemann sphere with positively-oriented punctures at $0$ and $z$, a negatively-oriented puncture at $\infty$, and local coordinates $w\mapsto w$, $w\mapsto w-z$, and $w\mapsto 1/w$, respectively. Moreover, all $P(z)$-tensor products are naturally isomorphic to each other via parallel transport isomorphisms. So for the remaining discussion in this subsection we may restrict to considering a single tensor product bifunctor on a given category of modules for a vertex operator algebra, which we may assume is the $P(1)$-tensor product.

The definition of tensor product of modules for a vertex operator algebra requires the notion of logarithmic intertwining operator
\begin{equation*}
 \cY(\cdot, x)\cdot: W_1\otimes W_2\rightarrow W_3\lbrace x\rbrace[\log x]
\end{equation*}
of type $\binom{W_3}{W_1\,W_2}$ among three modules for a vertex operator algebra from \cite[Definition 3.10]{HLZ2}. Since we always assume the possibility of logarithms, we will typically refer to a logarithmic intertwining operator simply as an intertwining operator. When we substitute the formal variable $x$ in an intertwining operator with $z\in\CC^\times$, using a choice of branch of $\log z$, we get a $P(z)$-intertwining map (see Definition 4.2 and Proposition 4.8 in \cite{HLZ3}) from $W_1\otimes W_2$ to the \textit{algebraic completion} $\overline{W_3}=\prod_{h\in\CC} (W_3)_{[h]}$  of $W_3$.

Now let $\cC$ be a  $\CC$-linear abelian category of modules for a simple vertex operator algebra $V$ (so that $\Endo_\cC V=\CC$). Given two $V$-modules $W_1$ and $W_2$ in $\cC$, a \textit{tensor product} of $W_1$ and $W_2$ in $\cC$ is a $V$-module $W_1\tens W_2$ in $\cC$ equipped with an intertwining operator $\cY_{W_1,W_2}$ of type $\binom{W_1\tens W_2}{W_1\,W_2}$ satisfying the following universal property: For any $V$-module $W_3$ in $\cC$ and intertwining operator $\cY$ of type $\binom{W_3}{W_1\,W_2}$, there is a unique $V$-module homomorphism $\eta: W_1\tens W_2\rightarrow W_3$ such that $\eta\circ\cY_{W_1,W_2}=\cY$. 

If a tensor product of any pair of modules in $\cC$ exists, then it is unique up to unique isomorphism and determines a bifunctor on $\cC$. For two morphisms $f_1: W_1\rightarrow\til{W}_1$ and $f_2: W_2\rightarrow\til{W}_2$ in $\cC$, the tensor product $f_1\tens f_2$ is the unique homomorphism, guaranteed by the universal property of $(W_1\tens W_2, \cY_{W_1,W_2})$, such that
\begin{equation*}
 (f_1\tens f_2)\circ\cY_{W_1,W_2}=\cY_{\til{W}_1,\til{W}_2}\circ(f_1\otimes f_2).
\end{equation*}
If $\cC$ admits vertex tensor category structure as in \cite{HL1}, we then have the following isomorphisms which give $\cC$ the structure of a braided tensor category; for more details, see \cite{HLZ8} or the exposition in \cite[Section 3.3]{CKM1}:
\begin{itemize}
 \item \textbf{Unit isomorphisms:} For $W$ in $\cC$, the natural left and right unit isomorphisms $l_W: V\tens W\rightarrow W$ and $r_W: W\tens V\rightarrow W$ are characterized by
 \begin{equation}\label{def:unit}
  l_W(\cY_{V,W}(v,x)w) =Y_W(v,x)w,\qquad r_W(\cY_{W,V}(w,x)v) = e^{xL(-1)} Y_W(v,-x)w
 \end{equation}
for $v\in V$, $w\in W$.

\item \textbf{Associativity isomorphisms:} For $W_1$, $W_2$, and $W_3$ in $\cC$, the natural associativity isomorphism
\begin{equation*}
 \cA_{W_1,W_2,W_3}: W_1\tens(W_2\tens W_3)\rightarrow(W_1\tens W_2)\tens W_3
\end{equation*}
is characterized as follows. Let $r_1, r_2\in \RR_+$ be such that $r_1>r_2>r_1-r_2$; then
\begin{align}\label{def:assoc}
\big \langle w', \overline{\cA_{W_1,W_2,W_3}}( & \cY_{W_1,W_2\tens W_3}(w_1,r_1)\cY_{W_2,W_3}(w_2,r_2)w_3)\big\rangle\nonumber\\
 &=\big\langle w', \cY_{W_1\tens W_2, W_3}(\cY_{W_1,W_2}(w_1,r_1-r_2)w_2,r_2)w_3\big\rangle
\end{align}
for $w_1\in W_1$, $w_2\in W_2$, $w_3\in W_3$, and $w'\in((W_1\tens W_2)\tens W_3)'$. Here we substitute positive real numbers into intertwining operators using the real-valued branch of logarithm, and $\cY_{W_1,W_2\tens W_3}(w_1,r_1)\cY_{W_2,W_3}(w_2,r_2)w_3$ is an element of $\overline{ W_1\tens(W_2\tens W_3)}$ defined by its action as a linear functional on $ (W_1\tens(W_2\tens W_3))'$:
\begin{align*}
 \big\langle w',\cY_{W_1,W_2\tens W_3}( & w_1,r_1)\cY_{W_2,W_3}(w_2,r_2)w_3\big\rangle\nonumber\\ &=\sum_{h\in\CC} \big\langle w',\cY_{W_1,W_2\tens W_3}(w_1,r_1)\pi_h(\cY_{W_2,W_3}(w_2,r_2)w_3)\big\rangle,
\end{align*}
where $\pi_h$ is projection onto the $h$-homogeneous space of $W_2\tens_{P(z_2)} W_3$. The vector $\cY_{W_1\tens W_2,W_3}(\cY_{W_1,W_2}(w_1,r_1-r_2)w_2,r_2)w_3\in\overline{(W_1\tens W_2)\tens W_3}$ has a similar interpretation; for these elements to be well defined, one must show that the above sum over $h\in\CC$ is absolutely convergent. Finally, $\overline{\cA_{W_1,W_2,W_3}}$ is the natural extension of the associativity isomorphism to algebraic completions. By \cite[Proposition 3.32]{CKM1}, the definition \eqref{def:assoc} of $\cA_{W_1,W_2,W_3}$ does not depend on the choice of $r_1$ and $r_2$.

\item \textbf{Braiding isomorphisms:} For $W_1$ and $W_2$ in $\cC$, the braiding isomorphism $\cR_{W_1,W_2}: W_1\tens W_2\rightarrow W_2\tens W_1$ is given by
\begin{equation}\label{def:braid}
 \cR_{W_1,W_2}(\cY_{W_1,W_2}(w_1,x)w_2) = e^{xL(-1)}\cY_{W_2,W_1}(w_2, e^{\pi i} x)w_1
\end{equation}
for $w_1\in W_1$, $w_2\in W_2$, where the notation $e^{\pi i} x$ indicates which branch of $\log(-1)$ to use for the substitution $x\mapsto -x$ in the intertwining operator. The inverse braiding isomorphism $\cR_{W_2,W_1}^{-1}: W_1\tens W_2\rightarrow W_2\tens W_1$ is characterized by
\begin{equation}\label{def:inv_braid}
 \cR_{W_2,W_1}^{-1}(\cY_{W_1,W_2}(w_1,x)w_2) = e^{xL(-1)}\cY_{W_2,W_1}(w_2, e^{-\pi i} x)w_1
\end{equation}
for $w_1\in W_1$, $w_2\in W_2$.

\end{itemize}

For sufficient conditions guaranteeing that $\cC$ in fact admits braided tensor category structure as described above, see Assumption 10.1 in \cite{HLZ6} and Assumptions 12.1 and 12.2 in \cite{HLZ8}; see also the discussion in \cite{Hu-suff-cond}. Due to work of Huang \cite{Hu-diff-eqs} and Miyamoto \cite{Mi} (see also the recent \cite{Mi_assoc}), it is expected that for a vertex operator algebra $V$, these conditions are satisfied by the category  $\cC^1_V$ of $C_1$-cofinite $V$-modules: a grading-restricted generalized $V$-module $W$ is \textit{$C_1$-cofinite} if $\dim W/C_1(W)<\infty$, where
\begin{equation*}
 C_1(W)=\bigg\lbrace v_{-1} w\,\bigg\vert\, v\in\bigoplus_{n\geq 1} V_{(n)},\,w\in W\bigg\rbrace.
\end{equation*}
In fact, the following theorem can be extracted from \cite[Section 4]{CJORY} and \cite[Section 3]{CY}, although since the results in these papers are stated differently, we discuss the proof here:
\begin{theo}\label{thm:C1V}
 If $\cC^1_V$ is closed under the contragredient modules of \cite{FHL}, then $\cC^1_V$ admits the braided tensor category structure of \cite{HLZ8} described above.
\end{theo}
\begin{proof}
The closure of $\cC_V^1$ under tensor products is proved in \cite{Mi}, and the analytic properties of intertwining operators needed for existence of the associativity isomorphisms and for proving the triangle, pentagon, and hexagon axioms are proved in \cite{Hu-diff-eqs} (see also \cite[Section 11.2]{HLZ7}). However, there is one additional condition that $\cC_V^1$ should satisfy, which is used in the proof of \cite[Theorem 11.4]{HLZ7}; see also \cite[Theorem 3.1]{Hu-suff-cond}. 

To describe this condition, take $z\in\CC^\times$ and generalized $V$-modules $W_1$, $W_2$. From \cite[Theorem 5.48]{HLZ4}, the vector space $(W_1\otimes W_2)^*$ contains a weak $V$-module $\mathrm{COMP}_{P(z)}((W_1\otimes W_2)^*)$ consisting of linear functionals that satisfy the \textit{$P(z)$-compatibility condition} of \cite{HLZ4}. Then the condition we need on $\cC_V^1$ is that if $W_1$ and $W_2$ are objects of $\cC_V^1$ and $W\subseteq\mathrm{COMP}_{P(z)}((W_1\otimes W_2)^*)$ is a singly-generated lower-bounded generalized $V$-submodule, then $W$ is an object of $C_V^1$. This condition is proved in \cite[Theorem 4.2.5]{CJORY} and  \cite[Theorem 3.6]{CY}: First observe that the inclusion $W\subseteq(W_1\otimes W_2)^*$ induces a linear map $I: W_1\otimes W_2\rightarrow W^* =\overline{W'}$ such that
\begin{equation*}
\langle I(w_1\otimes w_2), f\rangle = f(w_1\otimes w_2)
\end{equation*}
for all linear functionals $f\in W$. Because $W$ is lower bounded, $W'$ is also a lower-bounded generalized $V$-module, and then the definition of the $P(z)$-compatibility condition, in particular \cite[Equation 5.141]{HLZ4}, implies that $I$ is a $P(z)$-intertwining map (similar to \cite[Proposition 5.24]{HLZ4}). In particular, $I=\cY(\cdot, z)$ for some intertwining operator $\cY$ of type $\binom{W'}{W_1\,W_2}$, where we substitute $x\mapsto z$ in $\cY$ using some choice of branch of $\log z$.

By \cite[Key Theorem]{Mi} and \cite[Proposition 2.2]{CMY}, the image $\im\cY$ of the intertwining operator $\cY$ is an object of $\cC_V^1$, and thus $(\im\cY)'$ is also $C_1$-cofinite by hypothesis. Thus we have a $V$-homomorphism
\begin{equation*}
 \delta: W\hookrightarrow W'' \twoheadrightarrow (\im\cY)'
\end{equation*}
of $W$ to a grading-restricted generalized $V$-module, characterized by 
\begin{equation*}
\langle\delta(f), I(w_1\otimes w_2)\rangle =\langle I(w_1\otimes w_2), f\rangle =f(w_1\otimes w_2)
\end{equation*}
for $f\in W$, $w_1\in W_1$, and $w_2\in W_2$. It is clear from the above that $\delta$ is injective, so $W$ is a grading-restricted generalized $V$-module. It then follows that the map $W\rightarrow W''$ is an isomorphism, so that $\delta$ is also surjective. Consequently, $W\cong(\im\cY)'$ is an object of $\cC_V^1$, proving the desired condition on $\cC_V^1$ and thus the theorem.
\end{proof}

It is usually hard to show that a braided tensor category $\cC$ of $V$-modules is rigid, but under mild conditions, every $V$-module in $\cC$ has a contragredient in the sense of Definition \ref{def:contragredient}:
\begin{propo}\label{prop:contra_property}
 If $V$ is self-contragredient and $\cC$ is a braided tensor category of $V$-modules which is closed under the contragredient modules of \cite{FHL}, then contragredient modules are contragredients in the sense of Definition \ref{def:contragredient}.
\end{propo}
\begin{proof}
 For $V$-modules $W_1$, $W_2$, $W_3$, let $\cV^{W_3}_{W_1, W_2}$ denote the vector space of intertwining operators of type $\binom{W_3}{W_1\,W_2}$, and let $\kappa: V\rightarrow V'$ be a $V$-module isomorphism. Then we have the following sequence of isomorphisms of functors from $\cC$ to $\cV ec_\CC$:
 \begin{equation*}
  \hom_V(\bullet\tens W, V)\xrightarrow{\Gamma_1} \cV^{V}_{\bullet, W}\xrightarrow{\Gamma_2} \cV^{W'}_{\bullet, V'}\xrightarrow{\Gamma_3} \cV^{W'}_{\bullet,V} \xrightarrow{\Gamma_4} \hom_V(\bullet\tens V, W')\xrightarrow{\Gamma_5} \hom_V(\bullet, W').
 \end{equation*}
The first isomorphism comes from the universal property of the tensor product $\tens$: given $F: X\tens W\rightarrow V$, $\Gamma_1(F)=F\circ\cY_{X,W}$. The second isomorphism comes from the symmetry of spaces of intertwining operators proved in Theorem 5.5.1 and Proposition 5.5.2 of \cite{FHL} (see also \cite[Proposition 3.46]{HLZ2}): given an intertwining operator $\cY$ of type $\binom{V}{X\,W}$, we can define $\Gamma_2(\cY)$ of type $\binom{W'}{X\,V'}$ by the formula
\begin{equation*}
 \langle \Gamma_2(\cY)(b,x)v', w\rangle = \langle v', \cY(e^{xL(1)} e^{\pi i L(0)}(x^{-L(0)})^2 b, x^{-1})w\rangle
\end{equation*}
for $b\in X$, $w\in W$, and $v'\in V'$. The third isomorphism uses the isomorphism $\kappa$: given $\cY$ of type $\cV^{W'}_{X,V'}$, $\Gamma_3(\cY)=\cY\circ(\Id_{X}\otimes\kappa)$. The fourth isomorphism is again the universal property of $\tens$, and the final isomorphism uses the right unit isomorphisms: given $F: X\tens V\rightarrow W'$, we define $\Gamma_5(F)=F\circ r_{X}^{-1}$.

It is easy to check that each of these isomorphisms is natural, so they compose to yield a natural isomorphism $\Gamma: \hom_V(\bullet\tens W,V)\rightarrow\hom_V(\bullet,W')$. For the first and fourth, this amounts to the definition of the tensor product of morphisms in $\cC$. For the third, naturality amounts to the identity
\begin{equation*}
 (f\otimes \Id_{V'})\circ(\Id_{W_1}\otimes\kappa) = f\otimes \kappa =(\Id_{W_2}\otimes\kappa)\circ(f\otimes\Id_V)
\end{equation*}
for any morphism $f: W_1\rightarrow W_2$ in $\cC$, and the fifth isomorphism is natural because the right unit isomorphisms in $\cC$ are natural. Proving that the second isomorphism is natural is most interesting: we need to show that for any homomorphism $f: W_1\rightarrow W_2$ in $\cC$ and intertwining operator $\cY$ of type $\binom{V}{W_2\,W}$, 
$$\Gamma_{2; W_1}(\cY\circ(f\otimes\Id_W))=\Gamma_{2; W_2}(\cY)\circ(f\otimes\Id_{V'}).$$
In fact, for any $w_1\in W_1$, $w\in W$, and $v'\in V'$, we have
\begin{align*}
 \langle\Gamma_{2; W_1}(\cY\circ(f\otimes\Id_W))(w_1,x)v', w\rangle & = \langle v', [\cY\circ(f\otimes\Id_W)](e^{xL(1)}e^{\pi i L(0)} (x^{-L(0)})^2 w_1, x^{-1})w\rangle\nonumber\\
 & = \langle v', \cY(f(e^{xL(1)} e^{\pi i L(0)} (x^{-L(0)})^2 w_1), x^{-1})w\rangle\nonumber\\
 & =\langle v', \cY(e^{xL(1)} e^{\pi i L(0)} (x^{-L(0)})^2 f(w_1), x^{-1})w\rangle\nonumber\\
 & = \langle \Gamma_{2; W_2}(\cY)(f(w_1), x)v', w\rangle\nonumber\\
 & = \langle [\Gamma_{2;W_2}(\cY)\circ(f\otimes\Id_{V'})](w_1,x)v', w\rangle,
\end{align*}
where the third equality follows because the $V$-module homomorphism $f$ commutes with $L(0)$ and $L(1)$.
\end{proof}

\begin{rema}
 The previous proposition has appeared before in \cite[Section 5.1]{CKM2}, although fewer details for the proof were given there. Recalling Remark \ref{rem:r-category}, Proposition \ref{prop:contra_property} shows that if $V$ is self-contragredient, then a braided tensor category $\cC$ of $V$-modules that is closed under contragredients is an $r$-category in the sense of \cite{BD}. If $V$ is not self-contragredient, then $\cC$ is a Grothendieck-Verdier category in the sense of \cite{BD}, with dualizing object $V'$ (see \cite[Theorem 2.12]{ALSW}).
\end{rema}

Proposition \ref{prop:contra_property} shows that if $V$ is a self-contragredient vertex operator algebra and $\cC$ is a rigid braided tensor category of $V$-modules closed under the contragredient modules of \cite{FHL}, then duals in $\cC$ are contragredient modules. Moreover, $\cC$ has a canonical ribbon structure determined by the twist $\theta_W= e^{2\pi i L(0)}$ for objects $W$ in $\cC$. Even if such a braided tensor category $\cC$ is not rigid, it will still have a contragredient functor and a twist.

\subsection{General coset theory}\label{subsec:cosets}

Here we discuss some general results on coset subalgebras of vertex operator algebras. The general setting for this subsection is the following:
\begin{itemize}
 \item $(A,Y,\vac,\omega)$ is a positive-energy vertex operator algebra, that is, $A$ is $\NN$-graded by conformal weights and $A_{(0)}=\CC\vac$.
 
 \item $(U, Y_U=Y\vert_{U\otimes U},\vac,\omega_U)$ is a vertex operator subalgebra of $A$ with conformal vector $\omega_U\in A_{(2)}\cap U$ such that $L(1)\omega_U=0$. 
\end{itemize}
In this setting, the \textit{commutant} or \textit{coset} of $U$ in $A$ is defined by
\begin{equation*}
 C_A(U)=\lbrace v\in A\,\vert\,[Y(u, x_1), Y(v, x_2)]=0\,\,\text{for all}\,\,u\in U\rbrace.
\end{equation*}
Equivalently, (see \cite[Remark 3.2.4]{LL} and \cite[Theorem 5.2]{FZ}),
\begin{align}\label{eqn:comm_char}
 C_A(U) &=\lbrace v\in A\,\vert\, Y(u,x)v\in A[[x]]\,\,\text{for all}\,\,u\in U\rbrace\nonumber\\
 & =\lbrace v\in A\,\vert\,L_U(-1)v=0\rbrace,
\end{align}
where $L_U(-1)$ is the coefficient of $x^{-2}$ in $Y(\omega_U,x)$. 

By the Frenkel-Zhu coset theorem \cite[Theorem 5.1]{FZ}, the coset $V=C_A(U)$ is a vertex operator algebra with vertex operator $Y_V=Y\vert_{V\otimes V}$, vacuum $\vac$, and conformal vector $\omega_V=\omega-\omega_U$. Moreover, setting $Y(\omega_V,x)=\sum_{n\in\ZZ} L_V(n)\,x^{-n-2}$, we have
\begin{equation}\label{eqn:LU(n)=L(n)_on_U}
  L(n)\vert_U=L_U(n)\vert_U\hspace{2em}\text{and}\hspace{2em} L(n)\vert_V=L_V(n)\vert_V.
 \end{equation}
for $n\geq -1$ (see \cite[Theorem 3.11.12]{LL}). The $n=0$ case shows that the gradings of $U$ and $V$ are compatible with the grading of $A$:
\begin{equation*}
 U_{(n)}=A_{(n)}\cap U\hspace{2em}\text{and}\hspace{2em}V_{(n)}=A_{(n)}\cap V.
\end{equation*}
for $n\in\NN$. In particular, $U$ and $V$ are both positive-energy vertex operator algebras.

Now the following fundamental theorem on cosets is more or less known: for example a decomposition of $A$ as in the second assertion is mentioned in \cite[Section 5.2]{JL}. But for completeness, we will give a proof in Appendix \ref{app:gen_coset_proofs}:
\begin{theo}\label{thm:fund_coset}
In the setting of this subsection, assume in addition that:
\begin{itemize}
 \item The positive-energy vertex operator algebras $A$ and $U$ are simple.
  \item As a weak $U$-module, $A$ is the direct sum of irreducible grading-restricted $U$-modules.
  \item The vertex subalgebra $U$ equals its own double commutant: $U=C_A(C_A(U))$. That is, $U$ and its commutant $V$ form a dual pair in $A$: $V=C_A(U)$ and $U=C_A(V)$.
 \end{itemize}
Then:
\begin{enumerate}
 
 \item There is an injective vertex operator algebra homomorphism
 \begin{align*}
  \psi: U\otimes V & \rightarrow A \\
   u\otimes v & \mapsto u_{-1} v
 \end{align*}
 giving $A$ the structure of a $U\otimes V$-module.

 \item There is a $U\otimes V$-module isomorphism
 \begin{equation*}
  \varphi: \bigoplus_{i\in I} U_i\otimes V_i \rightarrow A
 \end{equation*}
where the $U_i$ are the distinct irreducible $U$-modules occurring in $A$, with $U=U_0$ for some $0\in I$, and the $V_i$ are grading-restricted $V$-modules given by $V_i=\hom_U(U_i,A)$. The isomorphism $\varphi$ is defined by
\begin{equation*}
 \varphi(u_i\otimes f) = f(u_i)
\end{equation*}
for $u_i\in U_i$ and $f\in V_i=\hom_U(U_i,A)$.

\item For $i\in I$, $\dim \hom_V(V,V_i) =\delta_{i,0}$.

\item The positive-energy vertex operator algebra $V$ is simple.
\end{enumerate}
\end{theo}

The $U\otimes V$-module decomposition of $A$ in the third assertion of the theorem generalizes to any $U\otimes V$-module which is semisimple as a $U$-module. First, recall that we make no grading assumptions on a weak module for a vertex operator algebra, while a generalized module has a generalized $L(0)$-eigenspace direct sum decomposition. A generalized module $X$ is grading-restricted if the generalized $L(0)$-eigenspaces $X_{[h]}$ are finite dimensional and for any $h\in\CC$, $W_{[h+n]}=0$ for $n\in\ZZ$ sufficiently negative. We will also prove the following theorem in Appendix \ref{app:gen_coset_proofs}:
\begin{theo}\label{thm:decomp_of_UV-mod}
Suppose that $U$ and $V$ are vertex operator algebras and $X$ is a weak $U\otimes V$-module which, as a weak $U$-module, is the direct sum of irreducible grading-restricted $U$-modules. Then there is a $U\otimes V$-module isomorphism
 \begin{equation*}
  \varphi_X: \bigoplus_{j\in J} M_j\otimes W_j\rightarrow X,
 \end{equation*}
where the $M_j$ are the distinct irreducible $U$-modules occurring in $X$ and $W_j=\hom_U(M_j,X)$ is a weak $V$-module. If moreover $X$ is a generalized $U\otimes V$-module, then each $W_j$ is a generalized $V$-module, and if $X$ is grading restricted, then so is each $W_j$.
\end{theo}

\section{The setting}\label{sec:setting}

In this section, we establish notation and basic properties for vertex operator algebra extensions $U\otimes V\hookrightarrow A$, as in Theorem \ref{thm:fund_coset}, from a tensor category point of view. The main tool is the vertex operator algebra extension theory developed in \cite{HKL, CKM1, CKM2}.

\subsection{Extension theory for cosets}\label{sec:mirror_equiv_setting}

We want to study coset-type extensions $U\otimes V\hookrightarrow A$ as in Theorem \ref{thm:fund_coset} in the setting that both $U$ and $V$ have categories of grading-restricted generalized modules that admit braided tensor category structure as described in Section \ref{subsec:VOAs}. However, we want to allow $A$ to be a conformal vertex algebra in general, rather than necessarily a positive-energy vertex operator algebra, so from now on, we will simply assume that $A$ is a simple conformal vertex algebra extension of $U\otimes V$ that satisfies the conclusions of Theorem \ref{thm:fund_coset}, as well as some additional conditions:
\begin{assum}\label{assum:finite_I}
 Let $U$ and $V$ be simple self-contragredient vertex operator algebras and assume in addition:
 \begin{enumerate}
  \item There is an injective conformal vertex algebra homomorphism $\iota_A: U\otimes V\rightarrow A$ where $A$ is a simple conformal vertex algebra.
   
 \item The conformal vertex algebra $A$ is semisimple as a $U\otimes V$-module, that is,
 \begin{equation*}
  A\cong\bigoplus_{i\in I} U_i\otimes V_i
 \end{equation*}
 where the $U_i$ are distinct simple $U$-modules and the $V_i$ are semisimple $V$-modules. We use $0\in I$ to denote the index such that $U_0=U$. For now, we assume the index set $I$ is finite, but later we will discuss the generalization to infinite index sets.
 
 \item The vertex subalgebras $U$ and $V$ form a dual pair in $A$ in the sense that $V_0=V$ and
 \begin{equation*}
\dim\hom_V (V,V_i) = \delta_{i,0} = \dim\hom_V(V_i,V). 
\end{equation*}
for all $i\in I$. 
 
\item The $U$-modules $U_i$ for $i\in I$ are objects of a semisimple braided ribbon tensor category $\cU$ of $U$-modules which is a locally-finite abelian category and is closed under contragredients. We take $\lbrace U_i\rbrace_{i\in\widetilde{I}}$ to be a set of equivalence class representatives of simple modules in $\cU$, where $\til{I}$ is a set containing $I$.
 
\item The $V_i$ for $i\in I$ are objects of a braided tensor category $\cV$ of $V$-modules which is closed under submodules, quotients, and contragredients. 
 \end{enumerate}

\end{assum}

\begin{rema}
We are not assuming here that $\mathcal{V}$ is rigid, but as discussed at the end of Section \ref{subsec:VOAs}, $\cV$ has a contragredient functor and a twist. Moreover, the dual of any object in $\cU$ is the contragredient module.
\end{rema}

Let $\cU_A$ be the semisimple abelian subcategory of $\cU$ with simple objects $U_i$ for $i\in I$, and similarly let $\cV_A$ be the full subcategory of $\cV$ whose objects are isomorphic to direct sums of the $V_i$ for $i\in I$. Our goal is to show that $\cU_A$ and $\cV_A$ are tensor subcategories of $\cU$ and $\cV$, respectively, and that they are braid-reversed tensor equivalent, so that in particular $\cV_A$ is rigid.

We will need the category $\cC$ of $U\otimes V$-modules whose objects are isomorphic to direct sums of modules $M\otimes W$, where $M$ is an object of $\cU$ and $W$ is an object of $\cV$. Since we assume $\cU$ is locally finite, all intertwining operator spaces $\cV_{M_1, M_2}^{M_3}\cong\hom_U(M_1\tens_U M_2,M_3)$ are finite dimensional. Then we can apply Theorem 5.2, Remark 5.3, and Theorem 5.5 of \cite{CKM2} to conclude that $\cC$ admits braided tensor category structure as described in Section \ref{subsec:VOAs}, and because $\cU$ is semisimple, the braided tensor category structure on $\cC$ is equivalent to that on the Deligne product $\cU\boxtimes\cV$. Specifically, if $M_1$, $M_2$ are modules in $\cU$ equipped with tensor product intertwining operator $\cY^U_{M_1,M_2}$ of type $\binom{M_1\tensU M_2}{M_1\,M_2}$ and $W_1$, $W_2$ are modules in $\cV$ equipped with tensor product intertwining operator $\cY^V_{W_1,W_2}$ of type $\binom{W_1\tensV W_2}{W_1\,W_2}$, then we can use the $U\otimes V$-module intertwining operator $\cY^U_{M_1,M_2}\otimes\cY^V_{W_1,W_2}$ to identify
\begin{equation*}
 (M_1\otimes W_1)\tensC(M_2\otimes W_2) = (M_1\tensU M_2)\otimes(W_1\tensV W_2).
\end{equation*}
Under this identification, the unit, associativity, and braiding isomorphisms in $\cC$ are identified with (vector space) tensor products of the corresponding isomorphisms in $\cU$ and $\cV$. Moreover, for modules $M$ in $\cU$ and $W$ in $\cV$, $M'\otimes W'$ is a contragredient of $M\otimes W$ in $\cC$. Under the identification $(M'\otimes W')\tens_\cC(M\otimes W) = (M'\tens_U M)\otimes(W'\tens_V W)$, the evaluation $e_{M\otimes W}$ is identified with $e_M\otimes e_W$.

Now since the index set $I$ in Assumption \ref{assum:finite_I} is so far assumed finite, $A$ is an object of $\cC$ and Theorem 3.2 and Remark 3.3 of \cite{HKL} identify $A$ as a commutative algebra object in $\cC$. In particular, the vertex operator $Y_A$ induces  a multiplication homomorphism
\begin{equation*}
 \mu_A: A\tensC A\rightarrow A
\end{equation*}
which satisfies associativity and commutativity, and the injection
\begin{equation*}
 \iota_A: U\otimes V\rightarrow A
\end{equation*}
 together with $\mu_A$ satisfies unitality. Since $A$ is a semisimple $U\otimes V$-module, we have a $U\otimes V$-module homomorphism $\varepsilon_A: A\rightarrow U\otimes V$ such that $\varepsilon_A\circ\iota_A=\Id_{U\otimes V}$.

 We also have a category $\repA$ whose objects are pairs $(X,\mu_X)$ with $X$ an object of $\cC$ and
 \begin{equation*}
  \mu_X: A\tensC X\rightarrow X
 \end{equation*}
an associative and unital $A$-action. The category $\repA$ is a tensor category \cite{KO, CKM1}, where for objects $X_1$, $X_2$ in $\repA$, the tensor product $X_1\tens_A X_2$ is obtained as the cokernel of a certain $\cC$-morphism into $X_1\tens_\cC X_2$. We use $I_{X_1,X_2}: X_1\tens_\cC X_2\rightarrow X_1\tens_A X_2$ to denote the corresponding cokernel morphism in $\cC$, and we set $\cY^A_{X_1,X_2}=I_{X_1,X_2}\circ\cY_{X_1,X_2}$, where $\cY_{X_1,X_2}$ is the tensor product intertwining operator of type $\binom{X_1\tens_\cC X_2}{X_1\,X_2}$ in $\cC$. For any object $(X,\mu_X)$ in $\repA$, we also set $Y_X=\mu_X\circ\cY_{A,X}$, a $U\otimes V$-module intertwining operator of type $\binom{X}{A\,X}$.

The unit object of the tensor category $\repA$ is $(A,\mu_A)$, and from \cite[Section 3.5.4]{CKM1}, the left and right unit isomorphisms in $\repA$ satisfy
\begin{equation}\label{eqn:RepA_unit}
 l^A_X(\cY^A_{A,X}(a,x)b)=Y_X(a,x)b,\hspace{2em} r^A_X(\cY^A_{X,A}(b,x)a)= e^{xL(-1)}Y_X(a,e^{-\pi i} x)b
\end{equation}
for $a\in A$, $b\in X$. By \cite[Proposition 3.62]{CKM1}, the associativity isomorphism in $\repA$ for objects $X_1$, $X_2$, $X_3$ satisfies
\begin{align}\label{eqn:RepA_assoc}
 \overline{\cA^A_{X_1,X_2,X_3}} & \big(\cY^A_{X_1,X_2\tens_A X_3}(b_1, r_1)\cY^A_{X_2,X_3}(b_2,r_2)b_3\big)=\cY^A_{X_1\tens_A X_2,X_3}(\cY^A_{X_1,X_2}(b_1, r_1-r_2)b_2, r_2)b_3
\end{align}
for $b_1\in X_1$, $b_2\in X_2$, $b_3\in X_3$, and any $r_1,r_2\in\RR_+$ such that $r_1>r_2>r_1-r_2>0$. As before, we substitute positive real numbers into intertwining operators using the real-valued branch of logarithm.

There is an induction tensor functor from $\cC$ to $\repA$, and we will need its restriction to $\cU$, which we embed into $\cC$ via $M\mapsto M\otimes V$. Specifically, we have
\begin{align*}
 \cF: \cU & \rightarrow \repA\nonumber\\
  M & \mapsto A\tensC(M\otimes V)\nonumber\\
  f & \mapsto \Id_A\tensC(f\otimes\Id_V)
\end{align*}
The induction functor $\cF$ is a tensor functor, with the $\repA$-isomorphism
\begin{equation*}
 r_A: \cF(U) = A\tensC(U\otimes V)\rightarrow A
\end{equation*}
and a natural isomorphism
\begin{equation*}
 f: \tensA\circ(\cF\times\cF)\rightarrow\cF\circ\tensU
\end{equation*}
compatible with the unit and associativity isomorphisms of $\cU$ and $\repA$. For modules $M_1$, $M_2$ in $\cU$, the isomorphism $f_{M_1,M_2}$ is characterized by the commutativity of
\begin{equation*}
 \xymatrixcolsep{4pc}
 \xymatrix{ \cF(M_1)\tensC\cF(M_2) \ar[rd]^{\widetilde{f}_{M_1,M_2}} \ar[d]_{I_{\cF(M_1),\cF(M_2)}} & \\
 \cF(M_1)\tensA\cF(M_2) \ar[r]_{f_{M_1,M_2}} & \cF(M_1\tensU M_2) \\
 },
\end{equation*}
where $\widetilde{f}_{M_1,M_2}$ is the composition
\begin{align*}
 ( A\tensC  (M_1\otimes V)) & \tensC(A\tensC(M_2\otimes V)) \xrightarrow{assoc.} (A\tensC ((M_1\otimes V)\tensC A))\tensC(M_2\otimes V)\nonumber\\
  & \xrightarrow{(\Id_A\tensC\cR^{-1}_{A,M_1\otimes V})\tensC\Id_{M_2\otimes V}} (A\tensC(A\tensC(M_1\otimes V)))\tensC(M_2\otimes V)\nonumber\\
  & \xrightarrow{assoc.} (A\tensC A)\tensC((M_1\otimes V)\tensC(M_2\otimes V))\nonumber\\ &\xrightarrow{\mu_A\tensC(\Id_{M_1\tensU M_2}\otimes l_V=r_V)} A\tensC ((M_1\tensU M_2)\otimes V).
\end{align*}
The arrows marked $assoc.$ denote suitable compositions of associativity isomorphisms and we have identified $(M_1\otimes V)\tensC(M_2\otimes V)=(M_1\tens_U M_2)\otimes(V\tensV V)$ for the last arrow.

We conclude this subsection with the following essential lemma. It is the same as \cite[Lemma 4.6]{CKM2} but in a slightly more general setting since here we are not assuming $\cV$ is rigid; see also \cite[Lemma 1.20]{KO}. Here we present a somewhat alternative proof:
\begin{lemma}\label{lem:prime_involution}
 There is an involution $i\mapsto i'$ of the index set $I$ such that $U_{i'}\cong U_i'$ and $V_{i'}\cong V_i'$ for all $i\in I$. In particular, there are isomorphisms $p_i: U_{i'}\rightarrow U_i'$ and $q_i: V_{i'}\rightarrow V_i'$ such that the diagram
 \begin{equation*}
  \xymatrixcolsep{7pc}
  \xymatrix{
  (U_{i'}\otimes V_{i'})\tens_\cC (U_i\otimes V_i) \ar[r]^(.65){\mu_A\vert_{(U_{i'}\otimes V_{i'})\tens_\cC(U_i\otimes V_i)}} \ar[d]^{(p_i\otimes q_i)\tens_\cC\Id_{U_i\otimes V_i}} & A \ar[d]^{\varepsilon_A} \\
  (U_i'\otimes V_i')\tens_\cC(U_i\otimes V_i) \ar[r]^(.6){e_{U_i\otimes V_i}=e_{U_i}\otimes e_{V_i}} & U\otimes V \\
  }
 \end{equation*}
commutes for $i\in I$.
\end{lemma}
\begin{proof}
 For any $i,j\in I$, the universal property of contragredients in $\cC$ implies there is a unique $U\otimes V$-module homomorphism $\varphi_{ij}: U_j\otimes V_j\rightarrow U_i'\otimes V_i'$ such that the diagram
 \begin{equation}\label{eqn:phi_ij_def}
  \xymatrixcolsep{7pc}
  \xymatrix{
  (U_{j}\otimes V_{j})\tens_\cC (U_i\otimes V_i) \ar[r]^(.65){\mu_A\vert_{(U_{j}\otimes V_{j})\tens_\cC(U_i\otimes V_i)}} \ar[d]^{\varphi_{ij}\tens_\cC\Id_{U_i\otimes V_i}} & A \ar[d]^{\varepsilon_A} \\
  (U_i'\otimes V_i')\tens_\cC(U_i\otimes V_i) \ar[r]^(.6){e_{U_i\otimes V_i}} & U\otimes V \\
  }
 \end{equation}
commutes. Given $i$, we first need to show that $\varphi_{ij}$ is non-zero for at least one $j\in I$.

Since $A$ is simple, it is irreducible as a (left) $A$-module (see \cite[Remark 3.9.8]{LL}). This means that for any non-zero $u_i\otimes v_i\in U_i\otimes V_i$, the span of coefficients in $Y_A(a,x)(u_i\otimes v_i)$, as $a$ ranges over $A$, is equal to $A$ (see \cite[Proposition 4.5.6]{LL}). In particular, there must be some $j\in I$ and $a_j\in U_j\otimes V_j$ such that 
\begin{equation*}
 \varepsilon_A(Y_A(a_j,x)(u_i\otimes v_i))\neq 0,
\end{equation*}
and then for some $n\in\ZZ$ we have
\begin{align*}
 0  \neq \varepsilon_A(\pi_n(Y_A(a_j,1)(u_i\otimes v_i))) &= \varepsilon_A\pi_n(\overline{\mu_A}(\cY_{U_j\otimes V_j, U_i\otimes V_i}(a_j,1)(u_i\otimes v_i)))\nonumber\\ & = (\varepsilon_A\circ\mu_A)\left(\pi_n(\cY_{U_j\otimes V_j, U_i\otimes V_i}(a_j,1)(u_i\otimes v_i))\right).
\end{align*}
Thus for some $j\in I$, we have $\varepsilon_A\circ\mu_A\vert_{(U_j\otimes V_j)\tens_\cC(U_i\otimes V_i)}\neq 0$ and so $\varphi_{ij}\neq 0$.

Now for $j\in I$, $\varphi_{ij}: U_j\otimes V_j\rightarrow U_i'\otimes V_i'$ is a homomorphism of weak $U$-modules. So because $U_j$ is irreducible, $U_i'\otimes V_i'$ contains a $U$-submodule isomorphic to $U_j$ whenever $\varphi_{ij}\neq 0$. But since $U_i'$ is also irreducible (see \cite[Proposition 5.3.2]{FHL}), this can only happen if $U_j\cong U_i'$. Since by assumption the $U_j$ are distinct, we see that $U_j\cong U_i'$, and consequently $\varphi_{ij}\neq 0$, for at most one $j\in I$.

Combining the previous two paragraphs, we see that for any $i\in I$, there is a unique $i'\in I$ such that $\varphi_{ii'}\neq 0$ and $U_{i'}\cong U_i'$. Since $$ U_{(i')'}\cong U_{i'}'\cong (U_i')'\cong U_i$$
(see \cite[Proposition 5.3.1]{FHL}), we have $(i')' =i$ so that $i\mapsto i'$ is an involution on $I$. It remains to show that $\varphi_{ii'}$ factors as $p_i\otimes q_i$ where $p_i: U_{i'}\rightarrow U_i'$ and $q_i: V_{i'}\rightarrow V_i'$ are isomorphisms. Since $U_{i'}$ and $U_i'$ are isomorphic irreducible $U$-modules, the factorization $\varphi_{ii'}=p_i\otimes q_i$ is easy:
\begin{equation*}
 \hom_{U\otimes V}(U_{i'}\otimes V_{i'}, U_i'\otimes V_i')\cong\hom_U(U_{i'}, U_i')\otimes\hom_V(V_{i'},V_i')
\end{equation*}
with $\hom_U(U_{i'}, U_i')$ one-dimensional and spanned by some isomorphism $p_i$. To show that $q_i$ is an isomorphism, we will first show that each $q_i$ is injective. 

For $i\in I$, consider $U_{i'}\otimes\ker q_i$, which is contained in $\ker \varphi_{ii'}$. Since we have seen that $\varphi_{ji'}=0$ if $j\neq i$, it follows from \eqref{eqn:phi_ij_def} that
\begin{equation*}
 \varepsilon_A\circ\mu_A\vert_{(U_{i'}\otimes\ker q_i)\tens_\cC A} =0.
\end{equation*}
But this means that for any $w\in U_{i'}\otimes\ker q_i$,
\begin{equation*}
 \varepsilon_A(Y_A(w,x)a)=0
\end{equation*}
for any $a\in A$. Since the right $A$-submodule of $A$ generated by $w$ is spanned by coefficients of $Y_A(w,x)a$ as $a$ ranges over $A$, analogously to \cite[Proposition 4.5.6]{LL}, it follows that $w$ generates a proper right submodule of $A$. Since $A$ is simple and right submodules are ideals, $w=0$ and $\ker q_i=0$.

Now to show $q_i$ is also surjective, note that $q_i$, as a $V$-module homomorphism, preserves conformal weight gradings. Thus since $q_i$ is injective,
\begin{equation*}
 \dim\,(V_{i'})_{[h]}\leq\dim\,(V_i')_{[h]} 
\end{equation*}
for any $h\in\CC$. But since $q_{i'}$ is also injective, we have
\begin{equation*}
 \dim\,(V_i')_{[h]} =\dim\,(V_i)_{[h]}^* =\dim\,(V_i)_{[h]} =\dim\,(V_{(i')'})_{[h]}\leq\dim\,(V_{i'}')_{[h]} =\dim\,(V_{i'})_{[h]}.
\end{equation*}
Thus the graded dimensions of $V_{i'}$ and $V_i'$ agree, which means any grading-preserving injection is also a surjection. Consequently, $q_i$ is a $V$-module isomorphism.
\end{proof}

\subsection{Infinite-order extensions}

In the setting of Assumption \ref{assum:finite_I}, we now consider the possibility that the index set $I$ is infinite. In this case, $A$ is no longer an object of the braided tensor category $\cC$ of $U\otimes V$-modules; instead, it is an object of the direct limit completion $\ind(\cC)$ studied in \cite{CMY}. The direct limit completion is the category of all weak $U\otimes V$-modules that are unions (equivalently, sums) of submodules that are objects of $\cC$. Since objects of $\cC$ are grading-restricted generalized $U\otimes V$-modules, objects of $\ind(\cC)$ are generalized $U\otimes V$-modules in the sense that they are the direct sums of their generalized $L(0)$-eigenspaces. But modules in $\ind(\cC)$ do not usually satisfy the grading-restriction conditions of a grading-restricted generalized $U\otimes V$-module.

Theorem 1.1 in \cite{CMY} gives conditions under which the vertex algebraic braided tensor category structure on $\cC$ extends to $\ind(\cC)$. Specifically, we need to add the following assumptions to our current setting:
\begin{assum}\label{assum:inf_order}
 If the index set $I$ of Assumption \ref{assum:finite_I} is infinite, we assume:
 \begin{enumerate}
  \item Every $U\otimes V$-module in $\cC$ is finitely generated. Since $\cU$ is semisimple and locally finite, so that every $U$-module in $\cU$ is finitely generated, it is sufficient to assume that every $V$-module in $\cV$ is finitely generated.
  
  \item For every intertwining operator $\cY$ of type $\binom{X_3}{X_1\,X_2}$ where $X_1$, $X_2$ are grading-restricted generalized $U\otimes V$-modules in $\cC$ and $X_3$ is a generalized $U\otimes V$-module in $\ind(\cC)$, the image $\im\cY\subseteq X_3$ is a module in $\cC$.
 \end{enumerate}
\end{assum}
\begin{rema}
By \cite[Corollary 2.14]{CMY}, which is an enhancement of the Key Theorem of \cite{Mi}, both assumptions above are satisfied if $\cC$ is the category of $C_1$-cofinite grading-restricted generalized $U\otimes V$-modules.
\end{rema}

Assumption \ref{assum:inf_order}(1), together with conditions in Assumption \ref{assum:finite_I}, is needed to ensure that $\ind(\cC)$ is a braided tensor category with a tensor product bifunctor extending the one on $\cC$; we will continue use $\tens_\cC$ to denote the tensor product on $\ind(\cC)$. By \cite[Theorems 6.2(1) and 6.3]{CMY}, we also have, for any generalized modules $X_1, X_2$ in $\ind(\cC)$, an intertwining operator $\cY_{X_1,X_2}$ of type $\binom{X_1\tens_\cC X_2}{X_1\,X_2}$ such that the tensor product of morphisms and the unit, associativity, and braiding isomorphisms in $\ind(\cC)$ have the exact same description as in Section \ref{subsec:VOAs}. Assumption \ref{assum:inf_order}(2) is needed to ensure that for any generalized modules $X_1, X_2$ in $\ind(\cC)$, the pair $(X_1\tens_\cC X_2,\cY_{X_1,X_2})$ satisfies the universal property of a tensor product in $\ind(\cC)$. In particular, this means by \cite[Proposition 4.26]{HLZ3} that for any generalized module $X$ in $\ind(\cC)$, the functors $X\tens_\cC\bullet$ and $\bullet\tens_\cC X$ are right exact.

Now by \cite[Theorem 5.7]{CKM2} and \cite[Theorem 7.5]{CMY}, the conformal vertex algebra $A$ is a commutative algebra in the braided tensor category $\ind(\cC)$. As the construction in \cite{CKM1} of the tensor category $\repA$ only requires $\ind(\cC)$ to be a braided tensor category with right exact tensor product, we still get this tensor category structure on $\repA$. Moreover, because the braided tensor category structure on $\ind(\cC)$ is characterized in terms of the intertwining operators $\cY_{X_1,X_2}$, just as in Section \ref{subsec:VOAs}, the unit and associativity isomorphisms in $\repA$ are still described by \eqref{eqn:RepA_unit} and \eqref{eqn:RepA_assoc}. (Although these results in \cite{CKM1} were stated assuming that $A$ is a commutative algebra in a braided tensor category of grading-restricted generalized modules for a subalgebra, the grading restriction conditions are not used in the proofs.)

This discussion shows that, as long as Assumption \ref{assum:inf_order} holds, we may now freely allow the index set $I$ for the decomposition $A\cong\bigoplus_{i\in I} U_i\otimes V_i$ to be infinite. So from now on, we will use the single notation $\cC$ to refer to the braided tensor category $\cU\tens\cV$ in case $I$ is finite, and to $\ind(\cU\tens\cV)$ in case $I$ is infinite. Of course, if $I$ is finite, we do not need the additional conditions of Assumption \ref{assum:inf_order}.

We would like to reduce Assumption \ref{assum:inf_order}(2) to conditions on the individual categories $\cU$ and $\cV$. We can do so using the following proposition; since the proof is technical and fairly long, we place it in Appendix \ref{app:intw_op_proof}:
\begin{propo}\label{prop:IndC_intw_op_cond}
Let $U$ and $V$ be vertex operator algebras, and let $\cU$ and $\cV$ be categories of grading-restricted generalized $U$- and $V$-modules, respectively. Assume moreover that:
\begin{itemize}
 \item Every $U$-module in $\cU$ is semisimple and $C_1$-cofinite.
 
 \item The category $\cV$ is closed under submodules, quotients, and finite direct sums.
 
\end{itemize}
 Then if $\cV$ satisfies the intertwining operator condition of Assumption \ref{assum:inf_order}(2), so does the category $\cC=\cU\tens\cV$ of $U\otimes V$-modules (whose objects are isomorphic to finite direct sums of modules $M\otimes W$ where $M$ is in $\cU$ and $W$ is in $\cV$).
\end{propo}

\begin{rema}
Using Theorem \ref{thm:C1V} and \cite[Corollary 2.14]{CMY}, all the conditions of Assumption \ref{assum:finite_I}, Assumption \ref{assum:inf_order}, and Proposition \ref{prop:IndC_intw_op_cond} pertaining to $\cV$ hold if $\cV$ can be taken to be the category $\cC_V^1$ of $C_1$-cofinite grading-restricted generalized $V$-modules and if $\cC^1_V$ is closed under contragredient modules. The $C_1$-cofinite module category is in fact closed under contragredients for all Virasoro and many affine vertex operator algebras \cite{CJORY, CY}.
\end{rema}

\section{The mirror equivalence theorem}\label{sec:proofs}

In this section, we prove the main result of this paper. In the notation of Section \ref{sec:mirror_equiv_setting}, we will show that the categories $\cU_A$ and $\cV_A$ are in fact tensor subcategories of $\cU$ and $\cV$, respectively, and that $\cU_A$ and $\cV_A$ are braid-reversed equivalent. This will show in particular that the semisimple $V$-modules $V_i$ appearing in the decomposition of $A$ as a weak $V$-module are simple, distinct, and rigid objects of $\cV$. 

To get the desired equivalence of categories from $\cU_A$ to $\cV_A$, recall from Section \ref{sec:mirror_equiv_setting} the induction tensor functor from $\cU$ to the tensor category $\repA$. Thus our first task is to construct a functor from $\repA$ to $V$-modules whose composition with induction restricts to a functor from $\cU_A$ to $\cV_A$.

\subsection{The invariants functor}\label{subsec:inv_functor}

In this subsection we construct the functor of $U$-invariants from $\repA$ to the category of weak $V$-modules. For a module $X$ in $\repA$, or more generally for $X$ a weak $U$-module, we define
\begin{equation*}
 X^U=\ker L_U(-1)\subseteq X,
\end{equation*}
where
\begin{equation*}
 L_U(-1)=\mathrm{Res}_x\,Y_X(\omega_U\otimes\vac_V,x).
\end{equation*}
By Corollary 4.7.6 and Proposition 4.7.7 of \cite{LL}, $X^U$ is linearly isomorphic to $\hom_U(U,X)$. For a morphism $f: X_1\rightarrow X_2$ in $\repA$, we define
\begin{equation*}
 f^U=f\vert_{X_1^U}: X_1^U\rightarrow X_2^U.
\end{equation*}
The image of $f^U$ is indeed contained in $X_2^U$ because $f$ is in particular a homomorphism of weak $U$-modules and thus commutes with $L_U(-1)$.

Let $\cF(\cU)$ be the full subcategory of $\repA$ whose objects are isomorphic to $\cF(M)$ for $M$ an object of $\cU$. Note that $\cF(\cU)$ is a monoidal subcategory of $\repA$ because $\cF(M_1)\tensA\cF(M_2)\cong\cF(M_1\tens_U M_2)$ for any modules $M_1$, $M_2$ in $\cU$.
\begin{lemma}
 The $U$-invariants functor restricts to a functor $\cG: \cF(\cU)\rightarrow\cV$.
\end{lemma}
\begin{proof}
 Since $\cU$ is semisimple and $\cF$ is an additive functor, it is sufficient to show that $\cF(U_j)^U$ is a module in $\cV$ for any $j\in\widetilde{I}$. In fact,
 \begin{align*}
  \cF(U_j)^U & = (A\tensC(U_j\otimes V))^U\nonumber\\
  & = \bigoplus_{i\in I} (U_i\tens_U U_j)^U\otimes(V_i\tensV V)\nonumber\\
  & \cong\left\lbrace\begin{array}{ccl}
                  V_{j'} & \mathrm{if} & j\in I\\
                  0 & \mathrm{if} & j\in\widetilde{I}\setminus I
                 \end{array} \right. 
 \end{align*}
because by rigidity of $\cU$ and simplicity of $U_i$ and $U_j$, $\hom_U(U, U_i\tensU U_j)$ is non-zero (and one-dimensional) precisely when $U_j\cong U_i'\cong U_{i'}$.
\end{proof}
We record the following corollary from the proof of the lemma:
\begin{corol}\label{cor:Phi_of_Ui}
 For $i\in\til{I}$,
 \begin{equation*}
  \cG(\cF(U_i')) \cong\left\lbrace\begin{array}{ccl}
                  V_{i} & \mathrm{if} & i\in I\\
                  0 & \mathrm{if} & i\in\widetilde{I}\setminus I
                 \end{array} \right. .
 \end{equation*}
\end{corol}

Now we would like to show that $\cG$ is a lax monoidal functor from $\cF(\cU)$ to $\cV$. First, there is an isomorphism $\psi: \cG(A)=\vac_U\otimes V\rightarrow V$. We also need a natural transformation
\begin{equation*}
 g: \boxtimes_V\circ(\cG\times\cG)\rightarrow\cG\circ\boxtimes_A
\end{equation*}
which is compatible with unit and associativity isomorphisms. To construct such a $g$, we consider modules $X_1$, $X_2$ in $\cF(\cU)$. Recall the (weak) $U\otimes V$-module intertwining operator $\cY^A_{X_1,X_2}=I_{X_1,X_2}\circ\cY_{X_1,X_2}$ of type $\binom{X_1\tens_A X_2}{X_1\,X_2}$. The commutator formula for intertwining operators implies
\begin{equation*}
 L_U(-1)\cY^A_{X_1,X_2}(b_1, x)b_2 = \cY^A_{X_1,X_2}(L_U(-1)b_1,x)b_2 +\cY^A_{X_1,X_2}(b_1,x)L_U(-1)b_2
\end{equation*}
for $b_1\in X_1$, $b_2\in X_2$. This means that $\cY^A_{X_1,X_2}$ restricts to a $V$-module intertwining operator of type $\binom{(X_1\tensA X_2)^U}{X_1^U\,X_2^U}$: the $L_V(-1)$-derivative property in particular holds because
\begin{align*}
 \dfrac{d}{dx}\cY^A_{X_1,X_2}(w_1,x) & =\cY^A_{X_1,X_2}(L(-1)w_1,x)\nonumber\\
 &=\cY^A_{X_1,X_2}(L_U(-1) w_1,x) +\cY^A_{X_1,X_2}(L_V(-1) w_1,x)\nonumber\\
 &=\cY^A_{X_1,X_2}(L_V(-1) w_1,x)
\end{align*}
for $w_1\in X_1^U$. Then the universal property of the tensor product of $V$-modules in $\cV$ implies there is a unique $V$-module homomorphism
\begin{equation*}
 g_{X_1,X_2}: \cG(X_1)\tensV\cG(X_2)\rightarrow\cG(X_1\tensA X_2)
\end{equation*}
such that
\begin{equation*}
 g_{X_1,X_2}\circ\cY^V_{\cG(X_1),\cG(X_2)} =\cY^A_{X_1,X_2}\vert_{\cG(X_1)\otimes\cG(X_2)}.
\end{equation*}
To show that the homomorphisms $g_{X_1,X_2}$ define a natural transformation, consider homomorphisms $f_1: X_1\rightarrow\til{X}_1$ and $f_2: X_2\rightarrow\til{X}_2$ in $\cF(\cU)$. Then using the description from \cite[Section 3.5.2]{CKM1} of the tensor product of morphisms in $\repA$, we have
\begin{align*}
 \cG(f_1\tensA f_2)\circ g_{X_1,X_2}\circ\cY^V_{\cG(X_1),\cG(X_2)} & = (f_1\tensA f_2)\circ\cY^A_{X_1,X_2}\vert_{\cG(X_1)\otimes\cG(X_2)}\nonumber\\
 & =\cY^A_{\til{X}_1,\til{X}_2}\circ(f_1\otimes f_2)\vert_{\cG(X_1)\otimes\cG(X_2)}\nonumber\\
 & = g_{\til{X}_1,\til{X}_2}\circ\cY^V_{\cG(\til{X}_1),\cG(\til{X}_2)}\circ(\cG(f_1)\otimes\cG(f_2))\nonumber\\
 & =g_{\til{X}_1,\til{X}_2}\circ(\cG(f_1)\tens_V\cG(f_2))\circ\cY^V_{\cG(X_1),\cG(X_2)}.
\end{align*}
Since $\cY^V_{\cG(X_1),\cG(X_2)}$ is a surjective intertwining operator, it follows that
\begin{equation*}
 \cG(f_1\tensA f_2)\circ g_{X_1,X_2}=g_{\til{X}_1,\til{X}_2}\circ(\cG(f_1)\tensV\cG(f_2)),
\end{equation*}
as required.
\begin{propo}
 The triple $(\cG,g,\psi)$ is a lax monoidal functor from $\cF(\cU)$ to $\cV$.
\end{propo}
\begin{proof}
 We have already observed that $\cG$ is a functor from $\cF(U)$ to $\cV$, $g$ is a natural transformation from $\tensV\circ(\cG\times\cG)$ to $\cG\circ\tensA$, and $\psi$ is an isomorphism from $\cG(A)$ to $V$. We just need to show that $g$ and $\psi$ are compatible with the unit isomorphisms in $\repA$ and $\cV$, and that $g$ is compatible with the associativity isomorphisms in $\repA$ and $\cV$.
 
 For the unit isomorphisms we need to show that the diagrams
 \begin{equation}\label{eqn:G_unit_compat}
  \xymatrixcolsep{4pc}
\begin{matrix}  \xymatrix{
  \cG(A)\tensV\cG(X) \ar[r]^{g_{A,X}} \ar[d]^{\psi\tens_V\Id_{\cG(X)}} & \cG(A\tens_A X) \ar[d]^{\cG(l^A_{X})} \\
  V\tensV\cG(X) \ar[r]^{l_{\cG(X)}} & \cG(X) \\
  }
  \end{matrix}, \hspace{2.5em}
  \xymatrixcolsep{4pc}
  \begin{matrix}\xymatrix{
  \cG(X)\tensA\cG(A) \ar[r]^{g_{X,A}} \ar[d]^{\Id_{\cG(X)}\tensV\psi} & \cG(X\tensA A) \ar[d]^{\cG(r^A_X)} \\
  \cG(X)\tensV V \ar[r]^{r_{\cG(X)}} & \cG(X) \\
  }\end{matrix}
 \end{equation}
commute for any module $X$ in $\cF(\cU)$. In fact, for $\vac_U\otimes v\in\cG(A)=\vac_U\otimes V$ and $w\in\cG(X)$,
\begin{align*}
 \cG(l^A_X) & \circ g_{A,X}\left(\cY^V_{\cG(A),\cG(X)}(\vac_U\otimes v, x)w\right) = \cG(l^A_X)\left(\cY^A_{A,X}(\vac_U\otimes v, x)w\right)\nonumber\\ & = Y_X(\vac_U\otimes v,x)w = Y_{\cG(X)}(v,x)w = l_{\cG(X)}\left(\cY^V_{V,\cG(X)}(v,x)w\right)\nonumber\\ & = l_{\cG(X)}\circ(\psi\tensV\Id_{\cG(X)})\left(\cY^V_{\cG(A),\cG(X)}(\vac_U\otimes v,x)w\right)
\end{align*}
and
\begin{align*}
 \cG(r^A_X) & \circ  g_{X,A}  \left(\cY^V_{\cG(X),\cG(A)}(w, x)(\vac_U\otimes v)\right)  = \cG(r^A_X)\left(\cY^A_{X,A}(w, x)(\vac_U\otimes v)\right)\nonumber\\ & = e^{xL(-1)}Y_X(\vac_U\otimes v,e^{-\pi i} x)w = e^{xL(-1)}Y_{\cG(X)}(v,-x)w = r_{\cG(X)}\left(\cY^V_{\cG(X),V}(w,x)v\right)\nonumber\\
 & = r_{\cG(X)}\circ(\Id_{\cG(X)}\tensV\psi)\left(\cY^V_{\cG(X),\cG(A)}(w,x)(\vac_U\otimes v)\right),
\end{align*}
where we have used \eqref{eqn:RepA_unit} and \eqref{def:unit}. Now since $\cY^V_{\cG(A),\cG(X)}$ and $\cY^V_{\cG(X),\cG(A)}$ are surjective intertwining operators, the diagrams \eqref{eqn:G_unit_compat} commute.

For the associativity isomorphisms, we need to show that the diagram
\begin{equation}\label{eqn:g_assoc_compat}
 \xymatrixcolsep{6pc}
\begin{matrix} \xymatrix{
 \cG(X_1)\tensV(\cG(X_2)\tensV\cG(X_3)) \ar[r]^{\cA_{\cG(X_1),\cG(X_2),\cG(X_3)}} \ar[d]^{\Id_{\cG(X_1)}\tensV g_{X_2,X_3}} & (\cG(X_1)\tensV\cG(X_2))\tensV\cG(X_3) \ar[d]^{g_{X_1,X_2}\tensV\Id_{\cG(X_3)}} \\
 \cG(X_1)\tensV\cG(X_2\tensA X_3) \ar[d]^{g_{X_1,X_2\tensA X_3}} & \cG(X_1\tensA X_2)\tensV \cG(X_3) \ar[d]^{g_{X_1\tensA X_2,X_3}} \\
 \cG(X_1\tensA(X_2\tensA X_3)) \ar[r]^{\cG(\cA^A_{X_1,X_2,X_3})} & \cG((X_1\tensA X_2)\tensA X_3)
 }\end{matrix}
\end{equation}
commutes for any modules $X_1$, $X_2$, $X_3$ in $\cF(\cU)$. In fact, for $w_1\in\cG(X_1)$, $w_2\in\cG(X_2)$, $w_3\in\cG(X_3)$, $w'\in\cG((X_1\tens_A X_2)\tens_A X_3)'$, and $r_1,r_2\in\RR$ such that $r_1>r_2>r_1-r_2>0$,
\begin{align*}
 & \big\langle w', \overline{g_{X_1\tensA X_2,X_3}\circ(g_{X_1,X_2}\tensV\Id_{\cG(X_3)})\circ\cA_{\cG(X_1),\cG(X_2),\cG(X_3)}}\cdot\nonumber\\
 &\hspace{10em}\cdot\big(\cY^V_{\cG(X_1),\cG(X_2)\tensV\cG(X_3)}(w_1,r_1)\cY^V_{\cG(X_2),\cG(X_3)}(w_2,r_2)w_3\big)\big\rangle\nonumber\\
 &\hspace{1em} = \big\langle w', \overline{g_{X_1\tensA X_2,X_3}\circ(g_{X_1,X_2}\tensV\Id_{\cG(X_3)})}\cdot\nonumber\\
 &\hspace{10em}\cdot\big(\cY^V_{\cG(X_1)\tens_V\cG(X_2), \cG(X_3)}\left(\cY^V_{\cG(X_1),\cG(X_2)}(w_1,r_1-r_2)w_2,r_2\right)w_3\big)\big\rangle\nonumber\\
 &\hspace{1em} = \big\langle w', \overline{g_{X_1\tensA X_2,X_3}}\big(\cY^V_{\cG(X_1\tensA X_2), \cG(X_3)}\left(\cY^A_{X_1,X_2}(w_1,r_1-r_2)w_2,r_2\right)w_3\big)\big\rangle\nonumber\\
 &\hspace{1em} = \left\langle w', \cY^A_{X_1\tensA X_2, X_3}\left(\cY^A_{X_1,X_2}(w_1,r_1-r_2)w_2,r_2\right)w_3\right\rangle\nonumber\\
 &\hspace{1em} = \big\langle w',\overline{\cG(\cA^A_{X_1,X_2,X_3})}\left(\cY^A_{X_1,X_2\tens_A X_3}(w_1,r_1)\cY^A_{X_2,X_3}(w_2,r_2)w_3\right)\big\rangle\nonumber\\
 &\hspace{1em} = \big\langle w',\overline{\cG(\cA^A_{X_1,X_2,X_3})\circ g_{X_1,X_2\tensA X_3}}\big(\cY^V_{\cG(X_1),\cG(X_2\tens_A X_3)}(w_1,r_1)\cY^A_{X_2,X_3}(w_2,r_2)w_3\big)\big\rangle\nonumber\\
 &\hspace{1em} = \big\langle w',\overline{\cG(\cA^A_{X_1,X_2,X_3})\circ g_{X_1,X_2\tensA X_3}\circ(\Id_{\cG(X_1)}\tensV g_{X_2,X_3})}\cdot\nonumber\\
 &\hspace{10em}\cdot\big(\cY^V_{\cG(X_1),\cG(X_2)\tens_V \cG(X_3)}(w_1,r_1)\cY^V_{\cG(X_2),\cG(X_3)}(w_2,r_2)w_3\big)\big\rangle
\end{align*}
using \eqref{def:assoc} and \eqref{eqn:RepA_assoc}. Here we substitute positive real numbers into intertwining operators using the real-valued branch of logarithm. Now since \cite[Corollary 7.17]{HLZ5} implies that $\cG(X_1)\tensV(\cG(X_2)\tensV\cG(X_3))$ is spanned by projections to the conformal weight spaces of vectors $\cY^V_{\cG(X_1),\cG(X_2)\tensV\cG(X_3)}(w_1,r_1)\cY^V_{\cG(X_2),\cG(X_3)}(w_2,r_2)w_3$, the diagram \eqref{eqn:g_assoc_compat} commutes.
\end{proof}

\subsection{braid-reversed tensor equivalence}

We now compose the tensor functor $\cF$ and the lax tensor functor $\cG$ to obtain the lax tensor functor $\Phi=\cG\circ\cF$, equipped with isomorphism $\varphi=\psi\circ\cG(r_A): \Phi(U)\rightarrow V$ and natural transformation
\begin{equation*}
 J: \tensV\circ(\Phi\times\Phi)\rightarrow\Phi\circ\tensU
\end{equation*}
such that
\begin{equation*}
 J_{M_1,M_2}=\cG(f_{M_1,M_2})\circ g_{\cF(M_1),\cF(M_2)}
\end{equation*}
for objects $M_1$, $M_2$ in $\cU$. The homomorphism $J_{M_1,M_2}$ is defined in terms of the commutativity of the following diagram:
\begin{equation}\label{eqn:J_def}
 \xymatrixcolsep{4pc}
 \xymatrix{
 \cF(M_1)\otimes\cF(M_2) \ar[rd]^{\cY^A_{\cF(M_1),\cF(M_2)}} \ar[r]^{\cY_{\cF(M_1),\cF(M_2)}} & \cF(M_1)\tensC\cF(M_2) \ar[d]^{I_{\cF(M_1),\cF(M_2)}} \ar[rd]^{\til{f}_{M_1,M_2}} & \\
 \cG(\cF(M_1))\otimes\cG(\cF(M_2)) \ar[d]^{\cY^V_{\cG(\cF(M_1)),\cG(\cF(M_2))}} \ar@{^{(}->}[u] & \cF(M_1)\tensA \cF(M_2) \ar[r]^{f_{M_1,M_2}} & \cF(M_1\tensU M_2) \\
 \cG(\cF(M_1))\tens_V\cG(\cF(M_2)) \ar[r]^(.525){g_{\cF(M_1),\cF(M_2)}} & \cG(\cF(M_1)\tensA\cF(M_2)) \ar[r]^(.52){\cG(f_{M_1,M_2})} \ar@{^{(}->}[u] & \cG(\cF(M_1\tensU M_2)) \ar@{^{(}->}[u] \\
 }
\end{equation}
Here arrows in the diagram labeled by intertwining operators are interpreted as the collection of linear maps corresponding to each coefficient of a monomial of the formal variable in the intertwining operator. The diagram commutes for each such fixed monomial.

Note that $J$ will not be a natural isomorphism if $I\subsetneq\til{I}$ because if $i\in\til{I}\setminus I$, then $\Phi(U_i)\tens_V\Phi(U_i')=0$ by Corollary \ref{cor:Phi_of_Ui}, while $\Phi(U_i\tens_U U_i')\neq 0$ since it contains a copy of $\Phi(U)\cong V$. However, we shall show that $J$ is an isomorphism when restricted to $\cU_A$. We start with a technical lemma:
\begin{lemma}\label{lem:J_surj_enough}
 For $j\in I$, $\im \Phi(i_{U_j})\subseteq \im J_{U_j,U_j'}$, where $i_{U_j}: U\rightarrow U_j\tens U_j'$ is the coevaluation in the rigid tensor category $\cU$.
\end{lemma}
\begin{proof}
 Both $\Phi(i_{U_j})$ and $J_{U_j,U_j'}$ are $V$-module homomorphisms into $\Phi(U_j\tens_U U_j')$. Since $\cF$ and $\cG$ are additive functors and using Corollary \ref{cor:Phi_of_Ui}, we have
 \begin{align*}
  \Phi(U_j\tens_U U_j')\cong\bigoplus_{i\in\til{I}} N^i_{j j'} \cG(\cF(U_i)) \cong\bigoplus_{i\in I} N^i_{j j'} V_{i'} = V\oplus\bigoplus_{i\in I\setminus\lbrace 0\rbrace} N^i_{j j'} V_{i'},
 \end{align*}
 where $N^k_{ij}$ is the multiplicity of $U_k$ in $U_i\tens_U U_j$. Under this identification of $\Phi(U_j\tensU U_j')$, $\im \Phi(i_{U_j})\subseteq V$ since $\Phi(U)\cong V$ and $\dim \hom_V(V,V_i)=\delta_{i,0}$ for $i\in I$. Thus we need to show that $V\subseteq\im J_{U_j,U_j'}$. Now,
\begin{equation*}
 \bigoplus_{i\in I\setminus\lbrace 0\rbrace} N^i_{j j'} V_{i'}\subseteq\ker \Phi(e_{U_j'}\circ(\delta_{U_j}\tensU\Id_{U_j'}))
\end{equation*}
since $\dim\hom_V(V_i,V)=\delta_{i,0}$ for $i\in I$. So if we can show $\Phi(e_{U_j'}\circ(\delta_{U_j}\tensU\Id_{U_j'}))\circ J_{U_j,U_j'}\neq 0$, the semisimplicity of $\Phi(U_j\tens_U U_j')$ will imply that $\im J_{U_j,U_j'}$ indeed contains the unique copy of $V$ in $\Phi(U_j\tensU U_j')$. (We use the assumption that each $V_i$ is a semisimple $V$-module here.)

In order to show $\Phi(e_{U_j'}\circ(\delta_{U_j}\tensU\Id_{U_j'}))\circ J_{U_j,U_j'}\neq 0$, it is sufficient to show that
\begin{equation*}
 \Phi(e_{U_j'}\circ(\delta_{U_j}\tensU\Id_{U_j'}))\circ\cG(f_{U_j,U_j'})\circ g_{\cF(U_j),\cF(U_j')}\left(\cY^V_{\Phi(U_j),\Phi(U_j')}(w_1,x)w_2\right) \neq 0
\end{equation*}
for some $w_1\in\Phi(U_j)$ and $w_2\in\Phi(U_j')$. By \eqref{eqn:J_def}, this is equivalent to showing
\begin{equation*}
 \cF(e_{U_j'}\circ(\delta_{U_j}\tensU\Id_{U_j'}))\circ\til{f}_{U_j,U_j'}\circ\cY_{\cF(U_j),\cF(U_j')}\vert_{\cG(\cF(U_j))\otimes\cG(\cF(U_j'))}\neq 0.
\end{equation*}
As the codomain of this map is $\cF(U)\cong A$, we can postcompose with $\varepsilon_A\circ\varphi: \cF(U)\rightarrow U\otimes V$, and then it is enough to show
\begin{equation*}
 \varepsilon_A\circ\varphi\circ\cF(e_{U_j'}\circ(\delta_{U_j}\tensU\Id_{U_j'}))\circ\til{f}_{U_j,U_j'}\circ\cY_{\cF(U_j),\cF(U_j')}\vert_{\cG(\cF(U_j)),\cG(\cF(U_j'))}\neq 0.
\end{equation*}

Let us examine $\Phi(U_j)=\cG(\cF(U_j))$ in more detail. We have
\begin{align*}
 \Phi(U_j) & = (A\tensC (U_j\otimes V))^U = \bigoplus_{i\in I} (U_i\tensU U_j)^U\otimes(V_i\tensV V) = (U_{j'}\tensU U_j)^U\otimes(V_{j'}\tensV V).
\end{align*}
Now, this is $\rho(\vac_U\otimes V_{j'})$, where $\rho$ is the $U\otimes V$-module homomorphism
\begin{align*}
 U\otimes V_{j'}\xrightarrow{i_{U_j}\otimes \Id_{V_{j'}}} (U_j\tensU U_j')\otimes V_{j'} & \xrightarrow{\cR_{U_j',U_j}^{-1}\otimes\Id_{V_{j'}}}  (U_j'\tensU U_j)\otimes V_{j'}\nonumber\\
 &\xrightarrow{(p_j^{-1}\tensU \Id_{U_j})\otimes r_{V_{j'}}^{-1}} (U_{j'}\tensU U_j)\otimes(V_{j'}\tensV V). 
\end{align*}
Similarly,
\begin{align*}
 \Phi(U_j') = (A\tensC(U_j'\otimes V))^U & = \bigoplus_{i\in I} (U_i\tensU U_j')^U\otimes(V_i\tensV V)\nonumber\\
 &= (U_j\tensU U_j')^U\otimes(V_j\tensV V) =\til{\rho}(\vac_U\otimes V_j)
\end{align*}
where $\til{\rho}= i_{U_j}\otimes r_{V_j}^{-1}$. Thus we get the commutative diagram
\begin{equation*}
 \xymatrixcolsep{4pc}
 \xymatrix{
 (\vac_U\otimes V_{j'})\otimes(\vac_U\otimes V_j) \ar[r]^(.55){\rho\otimes\til{\rho}} \ar@{^{(}->}[d] & \Phi(U_j)\otimes\Phi(U_j') \ar@{^{(}->}[d] & \\
 (U\otimes V_{j'})\otimes(U\otimes V_j) \ar[r]^(.55){\rho\otimes\til{\rho}} \ar[d]^{\cY_{U\otimes V_{j'}, U\otimes V_j}} & \cF(U_j)\otimes\cF(U_j') \ar[d]^{\cY_{\cF(U_j),\cF(U_j')}} & \\
 (U\otimes V_{j'})\tensC (U\otimes V_j) \ar[r]^(.55){\rho\tensC\til{\rho}} & \cF(U_j)\tensC\cF(U_j') \ar[r] & U\otimes V \\
 }
\end{equation*}
where the bottom right arrow is $ \varepsilon_A\circ\varphi\circ\cF(e_{U_j'}\circ(\delta_{U_j}\tensU\Id_{U_j'}))\circ\til{f}_{U_j,U_j'}$. 

To prove the lemma, then, it is enough to show that
\begin{equation*}
 \varepsilon_A\circ\varphi\circ\cF(e_{U_j'}\circ(\delta_{U_j}\tensU\Id_{U_j'}))\circ\til{f}_{U_j,U_j'}\circ(\rho\tensC\til{\rho})\circ\cY_{U\otimes V_{j'},U\otimes V_j}\vert_{(\vac_U\otimes V_{j'})\otimes(\vac_U\otimes V_j)}\neq 0.
\end{equation*}
Now, $\varepsilon_A\circ\varphi\circ\cF(e_{U_j'}\circ(\delta_{U_j}\tensU\Id_{U_j'}))\circ\til{f}_{U_j,U_j'}\circ(\rho\tensC\til{\rho})$ is a homomorphism 
\begin{equation*}
 (U\otimes V_{j'})\tensC (U\otimes V_j) = (U\tensU U)\otimes(V_{j'}\tensV V_j)\longrightarrow U\otimes V
\end{equation*}
in $\cC$. Thus this composition is given by $l_U\otimes E_j$ for some $V$-module homomorphism $E_j: V_{j'}\tensV V_j\rightarrow V$. Consequently, for $v_{j'}\in V_{j'}$ and $v_j\in V_j$, we have
\begin{align*}
 \varepsilon_A\circ\varphi\circ\cF(e_{U_j'}\circ(\delta_{U_j} & \tensU\Id_{U_j'}))\circ  \til{f}_{U_j,U_j'}\circ(\rho\tensC\til{\rho})\left(\cY_{U\otimes V_{j'},U\otimes V_j}(\vac_U\otimes v_{j'},x)(\vac_U\otimes v_j)\right)\nonumber\\
 & = l_U\left(\cY^U_{U,U}(\vac_U,x)\vac_U\right)\otimes E_j\left(\cY^V_{V_{j'},V_j}(v_{j'},x)v_j\right)\nonumber\\
 & = Y_U(\vac_U,x)\vac_U\otimes E_j\left(\cY^V_{V_{j'},V_j}(v_{j'},x)v_j\right) = \vac_U\otimes E_j\left(\cY^V_{V_{j'},V_j}(v_{j'},x)v_j\right).
\end{align*}
Since $\cY^V_{V_{j'},V_j}$ is surjective, the lemma will be proved if $E_j\neq 0$.

\allowdisplaybreaks

In order to calculate $E_j$, we need to analyze $\varepsilon_A\circ\varphi\circ\cF(e_{U_j'}\circ(\delta_{U_j}\tensU\Id_{U_j'}))\circ\til{f}_{U_j,U_j'}\circ(\rho\tensC\til{\rho})$ in detail: it is the composition
\begin{align*}
 &(U  \otimes V_{j'})  \tensC (U\otimes V_j) \xrightarrow{(i_{U_j}\otimes\Id_{V_{j'}})\tensC\Id_{U\otimes V_j}} ((U_j\tens_U U_j')\otimes V_{j'})\tensC (U\otimes V_j)\nonumber\\
 & \xrightarrow{(\cR_{U_j',U_j}^{-1}\otimes\Id_{V_{j'}})\tensC \Id_{U\otimes V_j}} ((U_j'\tensU U_j)\otimes V_{j'})\tensC (U\otimes V_j)\nonumber\\
 & \xrightarrow{( (p_j^{-1}\tensU\Id_{U_j})\otimes r_{V_{j'}}^{-1} )\tensC (i_{U_j}\otimes r_{V_j}^{-1})} ((U_{j'}\tensU U_j)\otimes(V_{j'}\tensV V))\tensC ((U_j\tensU U_j')\otimes(V_j\tensV V))\nonumber\\
 & = ((U_{j'}\otimes V_{j'})\tensC(U_j\otimes V))\tensC ((U_j\otimes V_j)\tensC (U_j'\otimes V))\nonumber\\
 &\longrightarrow (A\tensC(U_j\otimes V))\tensC(A\tensC(U_j'\otimes V))\xrightarrow{assoc.} (A\tensC ((U_j\otimes V)\tensC A))\tensC (U_j'\otimes V)\nonumber\\
 &\xrightarrow{(\Id_A\tensC \cR_{A,U_j\otimes V}^{-1})\tensC\Id_{U_j'\otimes V}} (A\tensC (A\tensC (U_j\otimes V)))\tensC(U_j'\otimes V)\nonumber\\
 & \xrightarrow{assoc.} (A\tensC A)\tensC((U_j\otimes V)\tensC (U_j'\otimes V)) \xrightarrow{\mu_A\tensC (\Id_{U_j\tensU U_j'}\otimes l_V)} A\tensC((U_j\tensU U_j')\otimes V)\nonumber\\
 & \xrightarrow{\Id_A\tensC((\delta_{U_j}\tensU\Id_{U_j'})\otimes\Id_V)} A\tensC((U_j''\tensU U_j')\otimes V)\nonumber\\ &\xrightarrow{\Id_A\tensC(e_{U_j'}\otimes\Id_V)} A\tensC (U\otimes V) \xrightarrow{r_A} A\xrightarrow{\varepsilon_A} U\otimes V.
\end{align*}
Applying naturality of the associativity and braiding isomorphisms in $\cC$ to the inclusions of $U_{j'}\otimes V_{j'}$ and $U_j\otimes V_j$ into $A$, and applying naturality of the right unit isomorphisms to $\varepsilon_A$, we get
\begin{align}\label{eqn:Ej_calc}
 & (U \otimes V_{j'})  \tensC (U\otimes V_j) \xrightarrow{(i_{U_j}\otimes\Id_{V_{j'}})\tensC\Id_{U\otimes V_j}} ((U_j\tens_U U_j')\otimes V_{j'})\tensC (U\otimes V_j)\nonumber\\
 & \xrightarrow{(\cR_{U_j',U_j}^{-1}\otimes\Id_{V_{j'}})\tensC \Id_{U\otimes V_j}} ((U_j'\tensU U_j)\otimes V_{j'})\tensC (U\otimes V_j)\nonumber\\
 & \xrightarrow{( (p_j^{-1}\tensU\Id_{U_j})\otimes r_{V_{j'}}^{-1} )\tensC (i_{U_j}\otimes r_{V_j}^{-1})} ((U_{j'}\tensU U_j)\otimes(V_{j'}\tensV V))\tensC ((U_j\tensU U_j')\otimes(V_j\tensV V))\nonumber\\
 & = ((U_{j'}\otimes V_{j'})\tensC(U_j\otimes V))\tensC ((U_j\otimes V_j)\tensC (U_j'\otimes V))\nonumber\\
& \xrightarrow{assoc.} ((U_{j'}\otimes V_{j'})\tensC((U_j\otimes V)\tensC(U_j\otimes V_j)))\tensC(U_j'\otimes V) \nonumber\\
 & \xrightarrow{(\Id_{U_{j'}\otimes V_{j'}}\tensC \cR_{U_j\otimes V_j, U_j\otimes V}^{-1})\tensC \Id_{U_j'\otimes V}} ((U_{j'}\otimes V_{j'})\tensC((U_j\otimes V_j)\tensC(U_j\otimes V)))\tensC (U_j'\otimes V)\nonumber\\
 & \xrightarrow{assoc.} ((U_{j'}\otimes V_{j'})\tensC(U_j\otimes V_j))\tensC((U_j\otimes V)\tensC(U_j'\otimes V))\nonumber\\ & \xrightarrow{(\varepsilon_A\circ\mu_A\vert_{(U_{j'}\otimes V_{j'})\tensC(U_j\otimes V_j)})\tensC(\Id_{U_j\tensU U_j'}\otimes l_V)} (U\otimes V)\tensC ((U_j\tensU U_j')\otimes V)\nonumber\\
 & \xrightarrow{\Id_{U\otimes V}\tensC((\delta_{U_j}\tensU\Id_{U_j'})\otimes\Id_V)} (U\otimes V)\tensC((U_j''\tensU U_j')\otimes V)\nonumber\\
& \xrightarrow{\Id_{U\otimes V}\tensC(e_{U_j'}\otimes\Id_V)} (U\otimes V)\tensC (U\otimes V)\xrightarrow{r_{U\otimes V}=l_{U\otimes V}} U\otimes V.
\end{align}
Now by Lemma \ref{lem:prime_involution},
\begin{align*}
 \varepsilon_A\circ\mu_A\vert_{(U_{j'}\otimes V_{j'})\tensC(U_j\otimes V_j)} & = e_{U_j\otimes V_j}\circ((p_j\otimes q_j)\tensC\Id_{U_j\otimes V_j})\nonumber\\
 &= (e_{U_j}\circ(p_j\tensU\Id_{U_j}))\otimes(e_{V_j}\circ(q_j\tensU\Id_{V_j})).
\end{align*}
Substituting this into \eqref{eqn:Ej_calc}, we get a (vector space) tensor product of a morphism in $\cU$ with a morphism in $\cV$.

The morphism in $\cU$ is the composition
\begin{align*}
 U & \tensU U\xrightarrow{i_{U_j}\tensU\Id_U} (U_j\tensU U_j')\tensU U\xrightarrow{\cR^{-1}_{U_j',U_j}\tensU\Id_U} (U_j'\tensU U_j)\tensU U\nonumber\\
 & \xrightarrow{(p_j^{-1}\tensU\Id_{U_j})\tensU i_{U_j}} (U_{j'}\tensU U_j)\tensU(U_j\tensU U_j') \xrightarrow{assoc.} (U_{j'}\tensU(U_j\tensU U_j))\tensU U_j'\nonumber\\ & \xrightarrow{(\Id_{U_{j'}}\tens\cR_{U_j,U_j}^{-1})\tensU \Id_{U_j'}} (U_{j'}\tensU(U_j\tensU U_j))\tensU U_j'\xrightarrow{assoc.} (U_{j'}\tensU U_j)\tensU (U_j\tensU U_j')\nonumber\\
 & \xrightarrow{(p_j\tensU \Id_{U_j})\tens\Id_{U_j\tensU U_j'}} (U_j'\tensU U_j)\tensU (U_j\tensU U_j')\xrightarrow{e_{U_j}\tensU (\delta_{U_j}\tensU \Id_{U_j'})} U\tensU (U_j''\tensU U_j')\nonumber\\
 & \xrightarrow{\Id_U\tensU e_{U_j'}} U\tensU U\xrightarrow{l_U=r_U} U.
\end{align*}
We use naturality of associativity to cancel $p_j$ with its inverse and to move $\cR_{U_j',U_j}^{-1}$ in this composition. We also move the final arrow using naturality of the left unit:
\begin{align*}
 &U  \tensU  U\xrightarrow{i_{U_j}\tensU i_{U_j}} (U_j\tensU U_j')\tensU (U_j\tensU U_j') \xrightarrow{\cA_{U_j\tensU U_j', U_j,U_j'}} ((U_j\tensU U_j')\tensU U_j)\tensU U_j'\nonumber\\
 & \xrightarrow{(\cR_{U_j',U_j}^{-1}\tensU\Id_{U_j})\tensU\Id_{U_j'}} ((U_j'\tensU U_j)\tensU U_j)\tensU U_j' \xrightarrow{\cA_{U_j',U_j,U_j}^{-1}\tensU\Id_{U_j'}} (U_j'\tensU(U_j\tensU U_j))\tensU U_j'\nonumber\\
 & \xrightarrow{(\Id_{U_j'}\tens\cR_{U_j,U_j}^{-1})\tens\Id_{U_j'}} (U_j'\tensU (U_j\tensU U_j))\tensU U_j'\xrightarrow{\cA_{U_j',U_j,U_j}\tensU\Id_{U_j'}} ((U_j'\tensU U_j)\tensU U_j)\tensU U_j'\nonumber\\
 & \xrightarrow{\cA_{U_j'\tensU U_j,U_j,U_j'}^{-1}} (U_j'\tensU U_j)\tensU (U_j\tensU U_j')\xrightarrow{e_{U_j}\tens\Id_{U_j\tensU U_j'}} U\tensU (U_j\tensU U_j')\xrightarrow{l_{U_j\tensU U_j'}} U_j\tensU U_j'\nonumber\\
 & \xrightarrow{\delta_{U_j}\tens\Id_{U_j'}} U_j''\tensU U_j'\xrightarrow{e_{U_j'}} U.
\end{align*}
By the hexagon axiom, the third through sixth arrows above simplify to
\begin{equation*}
(\cR_{U_j'\tensU U_j, U_j}^{-1}\tensU\Id_{U_j'})\circ( \cA_{U_j,U_j',U_j}^{-1}\tensU\Id_{U_j'}).
\end{equation*}
Then we apply naturality of the associativity and braiding isomorphisms to $e_{U_j}$ to obtain
\begin{align*}
 U & \tensU  U\xrightarrow{i_{U_j}\tensU i_{U_j}} (U_j\tensU U_j')\tensU (U_j\tensU U_j') \xrightarrow{\cA_{U_j\tensU U_j', U_j,U_j'}} ((U_j\tensU U_j')\tensU U_j)\tensU U_j'\nonumber\\
 & \xrightarrow{\cA_{U_j,U_j',U_j}^{-1}\tens_U\Id_{U_j'}} (U_j\tensU (U_j'\tensU U_j))\tensU U_j'\xrightarrow{(\Id_{U_j}\tensU e_{U_j})\tensU \Id_{U_j'}} (U_j\tensU U)\tensU U_j'\nonumber\\ & \xrightarrow{\cR_{U,U_j}^{-1}\tensU\Id_{U_j'}} (U\tensU U_j)\tensU U_j' \xrightarrow{\cA_{U,U_j,U_j'}^{-1}} U\tensU(U_j\tensU U_j') \xrightarrow{l_{U_j\tensU U_j'}} U_j\tensU U_j'\nonumber\\
 &\xrightarrow{\delta_{U_j}\tensU \Id_{U_j'}} U_j''\tensU U_j' \xrightarrow{e_{U_j'}} U.
\end{align*}
Using properties of the unit, we rewrite the fifth through seventh arrows:
\begin{align*}
 l_{U_j\tensU U_j'}\circ\cA_{U,U_j,U_j'}^{-1}\circ(\cR_{U,U_j}^{-1}\tensU\Id_{U_j'}) & = (l_{U_j}\tens\Id_{U_j'})\circ(\cR_{U,U_j}^{-1}\tensU\Id_{U_j'})\nonumber\\
 &= r_{U_j}\tens\Id_{U_j'} =(\Id_{U_j}\tens l_{U_j'})\circ\cA_{U_j,U,U_j'}^{-1}.
\end{align*}
Returning this to the composition, applying naturality of the associativity isomorphisms to $e_{U_j}$, and then using the pentagon axiom, we get
\begin{align*}
 U & \tensU  U\xrightarrow{i_{U_j}\tensU i_{U_j}} (U_j\tensU U_j')\tensU (U_j\tensU U_j') \xrightarrow{\cA^{-1}_{U_j,U_j',U_j\tensU U_j'}} U_j\tensU(U_j'\tensU(U_j\tensU U_j'))\nonumber\\
 &\xrightarrow{\Id_{U_j}\tens\cA_{U_j',U_j,U_j'}} U_j\tensU((U_j'\tensU U_j)\tensU U_j') \xrightarrow{\Id_{U_j}\tens(e_{U_j}\tensU \Id_{U_j'})} U_j\tensU (U\tensU U_j')\nonumber\\
 &\xrightarrow{\Id_{U_j}\tensU l_{U_j'}} U_j\tensU U_j'\xrightarrow{\delta_{U_j}\tensU \Id_{U_j'}} U_j''\tensU U_j' \xrightarrow{e_{U_j'}} U.
\end{align*}
We now focus on the first two arrows of the composition:
\begin{align*}
 \cA_{U_j,U_j',U_j\tensU U_j'}^{-1}\circ(i_{U_j}\tensU i_{U_j}) & = (\Id_{U_j}\tensU(\Id_{U_j'}\tensU i_{U_j}))\circ\cA_{U_j,U_j',U}^{-1}\circ(i_{U_j}\tensU\Id_U)\circ r_U^{-1}\circ r_U\nonumber\\
& = (\Id_{U_j}\tensU(\Id_{U_j'}\tensU i_{U_j}))\circ(\Id_{U_j}\tensU r_{U_j'}^{-1})\circ i_{U_j}\circ l_U.
\end{align*}
Returning this to the composition, rigidity of $U_j$ implies everything simplifies to
\begin{align*}
 U\tensU U\xrightarrow{l_U} U\xrightarrow{i_{U_j}} U_j\tensU U_j'\xrightarrow{\delta_{U_j}\tensU \Id_{U_j'}} U_j''\tensU U_j'\xrightarrow{e_{U_{j'}}} U,
\end{align*}
which is $(\dim_\cU U_j)l_U$. Because $U_j$ is a simple object of the semisimple ribbon category $\cU$, $\dim_\cU U_j\neq 0$ (see \cite[Proposition 4.8.4]{EGNO}).

On the other hand, the $\cV$ side of \eqref{eqn:Ej_calc} is the composition
\begin{align*}
 V_{j'} & \tensV  V_j\xrightarrow{r_{V_{j'}}^{-1}\tensV r_{V_j}^{-1}} (V_{j'}\tensV V)\tensV (V_j\tensV V)\xrightarrow{assoc.} (V_{j'}\tensV (V\tensV V_j))\tensV V\nonumber\\
 &\xrightarrow{ (\Id_{V_{j'}}\tens\cR_{V_j,V}^{-1})\tensV\Id_V} (V_{j'}\tensV (V_j\tensV V))\tensV V \xrightarrow{assoc.} (V_{j'}\tensV V_j)\tensV (V\tensV V)\nonumber\\
 & \xrightarrow{(q_j\tensV\Id_{V_j})\tensV (l_V=r_V)} (V_j'\tensV V_j)\tensV V \xrightarrow{e_{V_j}\tensV\Id_V} V\tensV V\xrightarrow{l_V=r_V} V.
\end{align*}
We apply naturality of the unit and associativity isomorphisms to $q_j$ and use properties of the unit isomorphisms to rewrite as
\begin{align*}
 V_{j'} & \tensV  V_j\xrightarrow{q_j\tensV\Id_{V_j}} V_j'\tensV V_j \xrightarrow{r_{V_j'}^{-1}\tensV\Id_{V_j}} (V_j'\tensV V)\tensV V_j\nonumber\\
 &\xrightarrow{r_{(V_j'\tensV V)\tensV V_j}^{-1}} ((V_j'\tensV V)\tensV V_j)\tensV V \xrightarrow{\cA_{V_j',V,V_j}^{-1}\tensV\Id_V} (V_j'\tensV(V\tensV V_j))\tensV V\nonumber\\ &\xrightarrow{(\Id_{V_j'}\tensV\cR_{V_j,V}^{-1})\tensV\Id_V} (V_j'\tensV(V_j\tensV V))\tensV V \xrightarrow{\cA_{V_j',V_j,V}\tensV\Id_V} ((V_j'\tensV V_j)\tensV V)\tensV V\nonumber\\
 &\xrightarrow{r_{(V_j'\tensV V_j)\tensV V}} (V_j'\tensV V_j)\tensV V\xrightarrow{r_{V_j'\tensV V_j}} V_j'\tens V_j\xrightarrow{e_{V_j}} V.
\end{align*}
We use naturality of the right unit to move $r_{(V_j'\tensV V)\tensV V_j}$ and cancel it against its inverse. Then again by properties of the unit, we get
\begin{align*}
 V_{j'}\tensV V_j\xrightarrow{q_j\tensV\Id_{V_j}} V_j'\tensV V_j & \xrightarrow{\Id_{V_j'}\tens l_{V_j}^{-1}} V_j'\tens(V\tensV V_j)\nonumber\\
 &\xrightarrow{\Id_{V_j'}\tensV\cR_{V_j,V}^{-1}} V_j'\tensV (V_j\tensV V)\xrightarrow{\Id_{V_j'}\tensV r_{V_j}} V_j'\tensV V_j\xrightarrow{e_{V_j}} V.
\end{align*}
The middle three arrows here collapse to the identity, so we are left with $e_{V_j}\circ(q_j\tensV\Id_{V_j})$.

Our calculations have shown that the homomorphism $E_j$ is given by
\begin{equation*}
 E_j=(\dim_\cU U_j) \, e_{V_j}\circ(q_j\tensV\Id_{V_j}).
\end{equation*}
Since this is non-zero, the lemma is proved.
\end{proof}

\begin{theo}\label{thm:J_iso}
 For $j\in I$ and $M$ a module in $\cU$, $J_{U_j,M}$ is an isomorphism.
\end{theo}
\begin{proof}
 We need to show that $J_{U_j,M}$ is both injective and surjective. The proof of injectivity is the same as the proof of \cite[Theorem 4.7]{McR} except for one difference. As in \cite{McR}, we take $k: K\rightarrow\Phi(U_j)\tensV\Phi(M)$ to be the kernel of $J_{U_j,M}$ and consider the $V$-module homomorphism $F$ which is the composition
 \begin{align*}
  \Phi  (U_j & \tens_U U_j') \tensV K \xrightarrow{\Id_{\Phi(U_j\tensU U_j')}\tensV k} \Phi(U_j\tensU U_j')\tensV(\Phi(U_j)\tensV\Phi(M))\nonumber\\
  & \xrightarrow{\cA_{\Phi(U_j\tensU U_j'),\Phi(U_j),\Phi(M)}} (\Phi(U_j\tensU U_j')\tensV\Phi(U_j))\tensV\Phi(M)\nonumber\\ &\xrightarrow{J_{U_j\tensU U_j', U_j}\tensV\Id_{\Phi(M)}} \Phi((U_j\tensU U_j')\tensU U_j)\tensV\Phi(M)\nonumber\\ &\xrightarrow{\Phi(\cA_{U_j,U_j',U_j}^{-1})\tensV\Id_{\Phi(M)}} \Phi(U_j\tensU (U_j'\tensU U_j))\tensV\Phi(M)\nonumber\\ &\xrightarrow{\Phi(\Id_{U_j}\tensU e_{U_j})\tensV\Id_{\Phi(M)}} \Phi(U_j\tensU U)\tensV\Phi(M) \xrightarrow{\Phi(r_{U_j})\tensV\Id_{\Phi(M)}} \Phi(U_j)\tensV\Phi(M).
 \end{align*}
Because $(\Phi, J, j)$ is a lax tensor functor and $\cU$ is rigid, it follows exactly as in \cite{McR} that:
\begin{itemize}
 \item The composition
 \begin{equation*}
  K\xrightarrow{l_K^{-1}} V\tensV K\xrightarrow{j^{-1}\tensV\Id_K} \Phi(U)\tensV K\xrightarrow{\Phi(i_{U_j})\tensV\Id_K} \Phi(U_j\tensU U_j')\tensV K\xrightarrow{F} \Phi(U_j)\tensV\Phi(M)
 \end{equation*}
is equal to $k$.

\item The composition
\begin{equation*}
 (\Phi(U_j)\tensV\Phi(U_j'))\tensV K\xrightarrow{J_{U_j,U_j'}\tensV\Id_K} \Phi(U_j\tensU U_j')\tensV K\xrightarrow{F} \Phi(U_j)\tensV\Phi(M)
\end{equation*}
is zero.
\end{itemize}
Now while the proof in \cite{McR} used surjectivity of $J$ to conclude that $k=0$, here we use Lemma \ref{lem:J_surj_enough}. That is, the second point above implies $F\vert_{\im (J_{U_j,U_j'}\tensV \Id_K)} =0$, so that by Lemma \ref{lem:J_surj_enough}, $F\vert_{\im (\Phi(i_{U_j})\tensV \Id_K)}=0$. Then the first point implies $k=0$, so $J_{U_j,M}$ is injective.

To prove that $J_{U_j,M}$ is surjective, we will use the injectivity we have just proven. Since we are assuming the $V_i$ are semisimple and since $\Phi(U_j\tensU U_j')$ is isomorphic to a direct sum of the $V_i$,
the injection $J_{U_j,U_j'}: \Phi(U_j)\tensV\Phi(U_j')\rightarrow\Phi(U_j\tensU U_j')$ realizes $\Phi(U_j)\tensV\Phi(U_j')$ as a submodule of a semisimple $V$-module. Hence there is a $V$-module homomorphism
\begin{equation*}
 \til{J}_{U_j,U_j'}: \Phi(U_j\tensU U_j')\rightarrow\Phi(U_j)\tensV\Phi(U_j')
\end{equation*}
such that 
$$\til{J}_{U_j,U_j'}\circ J_{U_j,U_j'} =\Id_{\Phi(U_j)\tensV\Phi(U_j')}$$
and $J_{U_j,U_j'}\circ\til{J}_{U_j,U_j'}$ is projection onto $\mathrm{Im}\,J_{U_j,U_j'}$ with respect to some $V$-module complement.

\allowdisplaybreaks

Now we will prove $J_{U_j,M}$ is surjective, by showing that the composition
\begin{align*}
 \Phi(  U_j & \tensU  M)  \xrightarrow{l_{\Phi(U_j\tensU M)}^{-1}} V\tensV\Phi(U_j\tensU M)\xrightarrow{j^{-1}\tens\Id_{\Phi(U_j\tensU M)}} \Phi(U)\tensV\Phi(U_j\tensU M)\nonumber\\
 & \xrightarrow{\Phi(i_{U_j})\tensV\Id_{\Phi(U_j\tensU M)}} \Phi(U_j\tensU U_j')\tensV\Phi(U_j\tensU M)\nonumber\\
 &\xrightarrow{\widetilde{J}_{U_j,U_j'}\tensV\Id_{\Phi(U_j\tensU M)}} (\Phi(U_j)\tensV\Phi(U_j'))\tensV\Phi(U_j\tensU M)\nonumber\\
 & \xrightarrow{\cA_{\Phi(U_j),\Phi(U_j'),\Phi(U_j\tensU M)}^{-1}} \Phi(U_j)\tensV(\Phi(U_j')\tensV\Phi(U_j\tensU M))\nonumber\\
 &\xrightarrow{\Id_{\Phi(U_j)}\tensV J_{U_j',U_j\tensU M}} \Phi(U_j)\tensV\Phi(U_j'\tensU(U_j\tensU M))\nonumber\\
 & \xrightarrow{\Id_{\Phi(U_j)}\tensV\Phi(\cA_{U_j',U_j,M})} \Phi(U_j)\tensU\Phi((U_j'\tensV U_j)\tensU M)\nonumber\\ &\xrightarrow{\Id_{\Phi(U_j)}\tensV\Phi(e_{U_j}\tensU \Id_M)} \Phi(U_j)\tensV\Phi(U\tensU M)\nonumber\\ &\xrightarrow{\Id_{\Phi(U_j)}\tensV\Phi(l_M)} \Phi(U_j)\tensV\Phi(M)\xrightarrow{J_{U_j,M}} \Phi(U_j\tensU M)
\end{align*}
is the identity on $\Phi(U_j\tens_U M)$. We shall accomplish this by reducing to the rigidity composition for $U_j$. First, we use the naturality of $J_{U_j,M}$ to obtain
\begin{align*}
  \Phi(  U_j & \tensU  M)\xrightarrow{l_{\Phi(U_j\tensU M)}^{-1}} V\tensV\Phi(U_j\tensU M)\xrightarrow{j^{-1}\tens\Id_{\Phi(U_j\tensU M)}} \Phi(U)\tensV\Phi(U_j\tensU M)\nonumber\\
 & \xrightarrow{\Phi(i_{U_j})\tensV\Id_{\Phi(U_j\tensU M)}} \Phi(U_j\tensU U_j')\tensV\Phi(U_j\tensU M)\nonumber\\
 &\xrightarrow{\widetilde{J}_{U_j,U_j'}\tensV\Id_{\Phi(U_j\tensU M)}} (\Phi(U_j)\tensV\Phi(U_j'))\tensV\Phi(U_j\tensU M)\nonumber\\
 & \xrightarrow{\cA_{\Phi(U_j),\Phi(U_j'),\Phi(U_j\tensU M)}^{-1}} \Phi(U_j)\tensV(\Phi(U_j')\tensV\Phi(U_j\tensU M))\nonumber\\
 &\xrightarrow{\Id_{\Phi(U_j)}\tensV J_{U_j',U_j\tensU M}} \Phi(U_j)\tensV\Phi(U_j'\tensU(U_j\tensU M))\nonumber\\
 & \xrightarrow{J_{U_j, U_j'\tensU(U_j\tensU M)}} \Phi(U_j\tensU(U_j'\tensU(U_j\tensU M)))\nonumber\\
 &\xrightarrow{\Phi(\Id_{U_j}\tensU\cA_{U_j',U_j,M})} \Phi(U_j\tensU((U_j'\tensU U_j)\tensU M))\nonumber\\
 & \xrightarrow{\Phi(\Id_{U_j}\tensU(e_{U_j}\tensU\Id_M))} \Phi(U_j\tensU(U\tensU M))\xrightarrow{\Phi(\Id_{U_j}\tensU l_M)} \Phi(U_j\tensU M).
\end{align*}
Next we apply compatibility of $J$ with associativity to the fifth through seventh arrows:
\begin{align*}
  \Phi(  U_j & \tensU  M)\xrightarrow{l_{\Phi(U_j\tensU M)}^{-1}} V\tensV\Phi(U_j\tensU M)\xrightarrow{j^{-1}\tens\Id_{\Phi(U_j\tensU M)}} \Phi(U)\tensV\Phi(U_j\tensU M)\nonumber\\
 & \xrightarrow{\Phi(i_{U_j})\tensV\Id_{\Phi(U_j\tensU M)}} \Phi(U_j\tensU U_j')\tensV\Phi(U_j\tensU M)\nonumber\\
 &\xrightarrow{\widetilde{J}_{U_j,U_j'}\tensV\Id_{\Phi(U_j\tensU M)}} (\Phi(U_j)\tensV\Phi(U_j'))\tensV\Phi(U_j\tensV M)\nonumber\\
 & \xrightarrow{J_{U_j,U_j'}\tensV\Id_{\Phi(U_j\tensU M)}} \Phi(U_j\tensU U_j')\tensV\Phi(U_j\tensU M)\nonumber\\
 &\xrightarrow{J_{U_j\tensU U_j',U_j\tensU M}} \Phi((U_j\tensU U_j')\tensU (U_j\tensU M))\nonumber\\
 & \xrightarrow{\Phi(\cA_{U_j,U_j',U_j\tensU M}^{-1})} \Phi(U_j\tensU(U_j'\tensU(U_j\tensU M)))\nonumber\\
 &\xrightarrow{\Phi(\Id_{U_j}\tensU\cA_{U_j',U_j,M})} \Phi(U_j\tensU((U_j'\tensU U_j)\tensU M))\nonumber\\
 & \xrightarrow{\Phi(\Id_{U_j}\tensU(e_{U_j}\tensU\Id_M))} \Phi(U_j\tensU(U\tensU M))\xrightarrow{\Phi(\Id_{U_j}\tensU l_M)} \Phi(U_j\tensU M).
\end{align*}
Now since $J_{U_j,U_j'}\circ \til{J}_{U_j,U_j'}$ is a projection onto the image of $J_{U_j,U_j'}$, and since $\im\Phi(i_{U_j})\subseteq\im J_{U_j,U_j'}$ by Lemma \ref{lem:J_surj_enough}, we can remove the fourth and fifth arrows in the composition. Then we apply the naturality of $J$ to $J_{U_j\tensU U_j',U_j\tensU M}$:
\begin{align*}
  \Phi(  U_j & \tensU  M)\xrightarrow{l_{\Phi(U_j\tensU M)}^{-1}} V\tensV\Phi(U_j\tensU M)\xrightarrow{j^{-1}\tens\Id_{\Phi(U_j\tensU M)}} \Phi(U)\tensV\Phi(U_j\tensU M)\nonumber\\ 
  & \xrightarrow{J_{U,U_j\tensU M}} \Phi(U\tensU(U_j\tensU M)) \xrightarrow{\Phi(i_{U_j}\tensU\Id_{U_j\tensU M})} \Phi((U_j\tensU U_j')\tensU(U_j\tensU M))    \nonumber\\
  & \xrightarrow{\Phi(\cA_{U_j,U_j',U_j\tensU M}^{-1})} \Phi(U_j\tensU(U_j'\tensU(U_j\tensU M)))\nonumber\\
  &\xrightarrow{\Phi(\Id_{U_j}\tensU\cA_{U_j',U_j,M})} \Phi(U_j\tensU((U_j'\tensU U_j)\tensU M))\nonumber\\
 & \xrightarrow{\Phi(\Id_{U_j}\tensU(e_{U_j}\tensU\Id_M))} \Phi(U_j\tensU(U\tensU M))\xrightarrow{\Phi(\Id_{U_j}\tensU l_M)} \Phi(U_j\tensU M).
\end{align*}
By compatibility of $J$ and $j$ with the unit isomorphisms, the first three arrows are $\Phi(l_{U_j\tensU M}^{-1})$. We also apply the triangle axiom to the last arrow and then naturality of associativity to $e_{U_j}$:
\begin{align*}
 \Phi( U_j & \tensU M)\xrightarrow{\Phi(l_{U_j\tensU M}^{-1})}\Phi(U\tensU(U_j\tensU M))\nonumber\\ &\xrightarrow{\Phi(i_{U_j}\tensU\Id_{U_j\tensU M})} \Phi((U_j\tensU U_j')\tensU(U_j\tensU M))    \nonumber\\
  & \xrightarrow{\Phi(\cA_{U_j,U_j',U_j\tensU M}^{-1})} \Phi(U_j\tensU(U_j'\tensU(U_j\tensU M)))\nonumber\\
  &\xrightarrow{\Phi(\Id_{U_j}\tensU\cA_{U_j',U_j,M})} \Phi(U_j\tensU((U_j'\tensU U_j)\tensU M))\nonumber\\
  &\xrightarrow{\Phi(\cA_{U_j,U_j'\tensU U_j, M})} \Phi((U_j\tensU(U_j'\tensU U_j))\tensU M)\nonumber\\
  &\xrightarrow{\Phi((\Id_{U_j}\tensU e_{U_j})\tensU\Id_M)} \Phi((U_j\tensU U)\tensU M)\xrightarrow{\Phi(r_{U_j}\tensU \Id_M)} \Phi(U_j\tensU M).
\end{align*}

At this point, the composition is simply $\Phi$ applied to a morphism in $\cU$, and we need to show that this morphism in $\cU$ is $\Id_{U_j\tensU M}$. In fact, the pentagon axiom applied to the third through fifth arrows yields
\begin{equation*}
 (\cA_{U_j,U_j',U_j}^{-1}\tensU\Id_M)\circ\cA_{U_j\tensU U_j', U_j,M}.
\end{equation*}
Returning this to the composition and using the naturality of associativity, we get
\begin{align*}
 U_j & \tensU M\xrightarrow{l_{U_j\tensU M}^{-1}} U\tensU(U_j\tensU M)\xrightarrow{\cA_{U,U_j,M}} (U\tensU U_j)\tensU M\nonumber\\
 &\xrightarrow{(i_{U_j}\tensU\Id_{U_j})\tensU\Id_M} ((U_j\tensU U_j')\tensU U_j)\tensU M \xrightarrow{\cA_{U_j,U_j',U_j}^{-1}\tensU\Id_M} (U_j\tensU(U_j'\tensU U_j))\tensU M\nonumber\\ &\xrightarrow{(\Id_{U_j}\tensU e_{U_j})\tensU\Id_M} (U_j\tensU U)\tensU M \xrightarrow{r_{U_j}\tensU\Id_M} U_j\tensU M.
\end{align*}
The first two arrows are $l_{U_j}^{-1}\tensU\Id_M$ by properties of the unit, and then the whole composition collapses to $\Id_{U_j\tensU M}$ by the rigidity of $U_j$. This completes the proof of the theorem.
\end{proof}

The preceding theorem brings us close to showing that the lax monoidal functor $\Phi: \cU\rightarrow\cV$ restricts to a (strong) monoidal functor on the subcategory $\cU_A$, which we will show is a tensor subcategory of $\cU$. So far, we have not analyzed compatibility of $(\Phi,J,j)$ with the braidings on $\cU$ and $\cV$, because $\Phi$ factors through the non-braided tensor category $\repA$. It is possible to show directly that $\Phi$ is a braid-reversed tensor functor on $\cU_A$, in the sense that 
\begin{equation*}
  J_{M_2,M_1}\circ\cR_{\Phi(M_2),\Phi(M_1)} =\Phi(\cR_{M_2,M_1}^{-1})\circ J_{M_1,M_2}
 \end{equation*}
for modules $M_1$, $M_2$ in $\cU_A$. However, we will not need to do this, because once we use Theorem \ref{thm:J_iso} to transfer rigidity from $\cU_A$ to $\cV_A$, results from \cite{CKM2} (proved under the assumption that modules in $\cV_A$ are rigid) show there is a braid-reversed tensor functor.

\begin{theo}\label{thm:mirror_equiv}
Under Assumption \ref{assum:finite_I}, together with Assumption \ref{assum:inf_order} if $I$ is infinite, the category $\cU_A$ of $U$-modules (respectively, the category $\cV_A$ of $V$-modules), is a semisimple braided ribbon category with distinct simple objects $U_i$ (respectively, $V_i$), $i\in I$. Moreover, there is a braid-reversed tensor equivalence $\tau: \cU_A\rightarrow\cV_A$ such that $\tau(U_i)\cong V_i'$ for all $i\in I$.
\end{theo}
\begin{proof}
 We first prove that $\cV_A$ is closed under the tensor product on $\cV$ and that it is rigid. Since every object of $\cV_A$ is isomorphic to a direct sum of $V_i$ for $i\in I$, and since $V_i\cong\Phi(U_{i'})$, it is enough to show that $\Phi(U_i)\tensV\Phi(U_j)$ is an object of $\cV_A$ for $i,j\in I$, and that each $\Phi(U_i)$ is rigid. By Theorem \ref{thm:J_iso} and Corollary \ref{cor:Phi_of_Ui},
 \begin{equation*}
  \Phi(U_i)\tensV\Phi(U_j)\cong\Phi(U_i\tensU U_j)\cong\bigoplus_{k\in\til{I}} N^k_{ij} \Phi(U_k)\cong\bigoplus_{k\in I} N^k_{ij} V_{k'},
 \end{equation*}
which is an object of $\cV_A$. Moreover, because $J_{U_i,U_{i}'}$ is an isomorphism for $i\in I$, $\Phi(U_{i}')$ is a dual of $\Phi(U_i)$, with evaluation
\begin{align*}
 \Phi(U_{i}')\tensV\Phi(U_i)\xrightarrow{J_{U_{i}',U_i}} \Phi(U_{i}'\tensU U_i)\xrightarrow{\Phi(e_{U_i})} \Phi(U)\xrightarrow{j} V
\end{align*}
and coevaluation
\begin{equation*}
 V\xrightarrow{j^{-1}} \Phi(U)\xrightarrow{\Phi(i_{U_i})} \Phi(U_i\tensU U_i')\xrightarrow{J_{U_i,U_{i}'}^{-1}} \Phi(U_i)\tensV\Phi(U_{i}').
\end{equation*}
It is easy to see that this evaluation and coevaluation satisfy the rigidity axioms from the rigidity of $U_i$, the naturality of $J$, and the compatibility of $J$ and $j$ with the unit and associativity isomorphisms of $\cU$ and $\cV$.

Now we show that the $V_i$ for $i\in I$ are distinct simple $V$-modules. Since we have assumed they are semisimple,  we just need to show $\dim\hom_V (V_i, V_j) = \delta_{i,j}$. In fact,
\begin{align*}
 \dim\hom_V (V_i, V_j) & =\dim\hom_V (V\tensV V_i, V_j)=\dim\hom_V (V, V_j\tensV V_i')\nonumber\\
 &=\sum_{k\in I} N^k_{j' i} \dim\hom_V (V, V_{k'})= N^0_{j' i} = \delta_{i,j}.
\end{align*}
The second equality here uses the rigidity of $V_i$, and the third equality uses
\begin{equation*}
 V_j\tensV V_i'\cong\Phi(U_{j'})\tensV\Phi(U_i)\cong\Phi(U_{j'}\tensU U_i)\cong\bigoplus_{k\in\til{I}} N^k_{j' i} \Phi(U_k)\cong\bigoplus_{k\in I} N^k_{j' i} V_{k'}.
\end{equation*}
Then the fourth equality uses the assumption $\dim\hom_V(V, V_i)=\delta_{i,0}$ for $i\in I$.

We have now shown that $\cV_A$ is a semisimple braided ribbon category with distinct simple objects $V_i$, $i\in I$. This means that $\cU$ and $\cV_A$ satisfy the conditions needed to apply \cite[Theorem 5.10(2)]{CKM2}, and we conclude that $\cU_A$ is closed under the tensor product on $\cU$ and that there is a braid-reversed tensor equivalence $\tau: \cU_A\rightarrow\cV_A$ such that $\tau(U_i)\cong V_i'$ for $i\in I$.
\end{proof}

\begin{rema}
 Recall that the twists on the braided ribbon categories $\cU_A$ and $\cV_A$ are given by $e^{2\pi iL_U(0)}$ and $e^{2\pi i L_V(0)}$, respectively. Since the vertex operator algebra $A$ is $\ZZ$-graded by $L(0)=L_U(0)+L_V(0)$-eigenvalues, the braid-reversed equivalence $\tau$ reverses twists. That is, if $\theta_{U_i}$ and $\theta_{V_i}$ are the scalars by which the twists act on the simple modules $U_i$ and $V_i$, respectively, then $\theta_{V_i'}=\theta_{V_i} =\theta_{U_i}^{-1}$.
\end{rema}

\subsection{Superalgebra extensions}\label{subsec:VOSAs}

Here we adapt Theorem \ref{thm:mirror_equiv} to coset extensions $U\otimes V\subseteq A$ where now $A$ is a conformal vertex superalgebra, using a trick due to Thomas Creutzig. There are two classes of conformal vertex superalgebra that we want to consider, called correct statistics and wrong statistics superalgebras in \cite{CKL}:
\begin{itemize}
 \item A conformal vertex superalgebra $A=A^\even\oplus A^\odd$ of correct statistics is $\frac{1}{2}\ZZ$-graded by $L(0)$-eigenvalues in such a way that $A^{\bar{j}}=\bigoplus_{n\in\frac{j}{2}+\ZZ} A_{(n)}$ for $j=0,1$.
 
 \item A vertex operator superalgebra of wrong statistics is $\ZZ$-graded by $L(0)$-eigenvalues.
\end{itemize}

We retain the assumptions on the vertex operator algebras $U$ and $V$ and on the tensor categories $\cU$ and $\cV$ from Section \ref{sec:mirror_equiv_setting}, but we replace the assumptions on $A$ with the following:
\begin{itemize}
 \item There is a conformal inclusion $U\otimes V\subseteq A^\even$ where $A=A^\even\oplus A^\odd$ is a simple conformal vertex superalgebra that is semisimple as a $U\otimes V$-module. That is, for $j\in\ZZ/2\ZZ$,
 \begin{equation*}
  A^j=\bigoplus_{i\in I^j} U_i\otimes V_i
 \end{equation*}
where the $U_i$ are distinct simple $U$-modules and the $V_i$ are semisimple $V$-modules.

\item Using $0\in I^\even$ to denote the index such that $U_0=U$, we have $V_0=V$ and
\begin{equation*}
 \dim\hom_{V}(V,V_i)=\delta_{i,0}
\end{equation*}
for $i\in I^\even\cup I^\odd$.
\end{itemize}
If the index set $I=I^\even\cup I^\odd$ is infinite, we also impose Assumption \ref{assum:inf_order}.

Now fix a simple vertex operator superalgebra $S=S^\even\oplus S^\odd$ with the same statistics as $A$ such that $S^\even$ is a simple self-contragredient vertex operator algebra and $S^\odd$ is a simple $C_1$-cofinite $S^\even$-module contained in some vertex algebraic braided tensor category of $S^\even$-modules. For example, if $A$ has correct statistics, we can take $S$ to be the vertex operator superalgebra of one free fermion, that is, $S^\even$ is the simple Virasoro vertex operator algebra $L(\frac{1}{2},0)$ and $S^\odd$ is its simple module $L(\frac{1}{2},\frac{1}{2})$. If $A$ has incorrect statistics, we can take $S$ to be the vertex operator superalgebra of one pair of symplectic fermions. Because $S^\even$ and $S^\odd$ are objects of some braided tensor category of $S^\even$-modules, \cite[Corollary 4.8]{McR} (see also \cite[Theorem 4.2]{CM} and \cite[Theorem 3.1]{CKLR}) shows that there is a semisimple braided ribbon category $\cS$ of $S^\even$-modules whose simple objects are $S^\even$ and $S^\odd$, and that $S^\odd\tens S^\odd\cong S^\even$.

So now we can replace the conformal vertex superalgebra $A$ with the conformal vertex algebra
\begin{equation*}
 \widetilde{A}=(S\otimes A)^\even =(S^\even\otimes A^\even)\oplus(S^\odd\otimes A^\odd)=\bigoplus_{i\in I^\even} (S^\even\otimes U_i)\otimes V_i\oplus\bigoplus_{i\in I^\odd} (S^\odd\otimes U_i)\otimes V_i.
\end{equation*}
We would like to show that $\widetilde{A}$ satisfies the assumptions in Section \ref{sec:mirror_equiv_setting}, with $U$ replaced by $S^\even\otimes U$ and $\cU$ replaced by $\cS\tens\cU$ (the semisimple category of $S^\even\otimes U$-modules whose simple objects are $S^j\otimes U_i$ for $i\in \widetilde{I}$ and $j\in\ZZ/2\ZZ$).

 First, $S^\even\otimes U$ is simple and self-contragredient because $S^\even$ and $U$ both are. We also have a conformal inclusion $(S^\even\otimes U)\otimes V\subseteq\widetilde{A}$ such that $\widetilde{A}$ is semisimple as a $(S^\even\otimes U)\otimes V$-module, but we need to show that $\widetilde{A}$ is simple. We first claim that $A^\even$ is a simple conformal vertex algebra and $A^\odd$ is a simple $A^\even$-module: Indeed, given any non-zero $b\in A^j$ for $j\in\ZZ/2\ZZ$, 
 \begin{equation*}
 A=\mathrm{span}\lbrace a_n b\,\vert\,a\in A^\even, n\in\ZZ\rbrace +\mathrm{span}\lbrace a_n b\,\vert\,a\in A^\odd, n\in\ZZ\rbrace
 \end{equation*}
 since $A$ is simple (see \cite[Proposition 4.5.6]{LL}). Since $a_n b\in A^{i+j}$ for $a\in A^i$, we get
 \begin{equation*}
 A^j=\mathrm{span}\lbrace a_n b\,\vert\, a\in A^0, n\in\ZZ\rbrace
 \end{equation*}
 for any non-zero $b\in A^j$, and thus $A^j$ is simple. Then
 \begin{equation*}
  \widetilde{A}=(S^\even\otimes A^\even)\oplus(S^\odd\otimes A^\odd)
 \end{equation*}
is a semisimple $S^\even\otimes A^\even$-module. 

The two simple $S^\even\otimes A^\even$-modules occurring in the decomposition of $\widetilde{A}$ are non-isomorphic since $S^\even$ and $S^\odd$ are distinct $S^\even$-modules. Thus any non-zero ideal $J\subseteq\widetilde{A}$ must contain either $S^\even\otimes A^\even$ or $S^\odd\otimes A^\odd$. If $J$ contains $S^\even\otimes A^\even$, then $J$ also contains $S^\odd\otimes A^\odd$ because 
\begin{equation*}
 s_1\otimes a_1=Y_{\widetilde{A}}(\vac,x)(s_1\otimes a_1)
\end{equation*}
for any $s_1\in S^\odd$, $a_1\in A^\odd$. If $J$ contains $S^\odd\otimes A^\odd$, then
\begin{equation*}
 Y_{\widetilde{A}}\vert_{(S^\odd\otimes A^\odd)\otimes(S^\odd\otimes A^\odd)}= Y_S\vert_{S^\odd\otimes S^\odd}\otimes Y_A\vert_{A^\odd\otimes A^\odd}
\end{equation*}
is an intertwining operator of type $\binom{S^\even\otimes A^\even}{S^\odd\otimes A^\odd\,S^\odd\otimes A^\odd}$ which is non-zero since simplicity of $S$ and $A$ implies $Y_S\vert_{S^\odd\otimes S^\odd}$ and $ Y_A\vert_{A^\odd\otimes A^\odd}$ are non-zero. But then $Y_{\widetilde{A}}\vert_{(S^\odd\otimes A^\odd)\otimes(S^\odd\otimes A^\odd)}$ is surjective since $S^\even\otimes A^\even$ is simple, so $J$ contains $S^\even\otimes A^\even$ as well. This shows that $\til{A}$ is simple.

Now it is clear that $S^\even\otimes U$ and $V$ form a dual pair inside $\widetilde{A}$, in the sense of Assumption \ref{assum:finite_I}. Next, for $i\in\widetilde{I}$ and $j\in\ZZ/2\ZZ$, the $S^j\otimes U_i$ are objects of the semisimple braided ribbon tensor category $\cS\tens\cU$ of $S^\even\otimes U$-modules. This category is locally-finite abelian and closed under contragredients since these properties hold for both $\cS$ and $\cU$. Finally, the $V_i$ are still objects of the braided tensor category $\cV$ of $V$-modules which is closed under submodules, quotients, and contragredients.

If the index set $I$ is infinite, we need to consider the conditions in Assumption \ref{assum:inf_order}. Every $S^\even\otimes U\otimes V$-module in $\cS\tens\cC=\cS\tens\cU\tens\cV$ is finitely generated since this holds for both $\cS$ and $\cC$. Now suppose $\cY$ is an intertwining operator of type $\binom{X_3}{X_1\,X_2}$ where $X_1$, $X_2$ are grading-restricted $S^\even\otimes U\otimes V$-modules in $\cS\tens\cC$ and $X_3$ is a generalized module in $\ind(\cS\tens\cC)$. Then $\im\cY$ will be a module in $\cS\tens\cC$ by Proposition \ref{prop:IndC_intw_op_cond} (with $\cU$ and $\cV$ replaced by $\cS$ and $\cC$, respectively), once we verify the conditions of that proposition. In fact, every $S^\even$-module in $\cS$ is semisimple and $C_1$-cofinite by assumption, and $\cC$ is closed under submodules, quotients and finite direct sums because $\cV$ is and because $\cU$ is semisimple (see \cite[Theorem 5.2]{CKM2}). Thus because we assume $\cC$ satisfies the intertwining operator condition of Assumption \ref{assum:inf_order}, Proposition \ref{prop:IndC_intw_op_cond} shows that $\cS\tens\cC$ does as well.

Theorem \ref{thm:mirror_equiv} now shows that the category $\cV_A$ of $V$-modules whose objects are finite direct sums of the $V_i$ is a semisimple braided ribbon category which is braid-reversed equivalent to a category of $(\cS\tens\cU)_{\til{A}}$ of $S^\even\otimes U$-modules. Now, the category $\cS$ of $S^\even$-modules is braided tensor equivalent to a cocycle twist of $\rep\ZZ/2\ZZ$ by \cite[Corollary 4.8]{McR}, and this cocycle twist is nothing but the braided tensor category of finite-dimensional vector superspaces. Using this equivalence of $\cS$ with superspaces, we can identify $(\cS\tens\cU)_{\til{A}}$ with a braided tensor category structure on the semisimple category $\cU_A$ of $U$-modules whose simple objects are the $U_i$, $i\in I$. The only difference between this braided tensor category structure on $\cU_A$ and that inherited from $\cU$ is a sign change on some of the braidings. That is, $\cU_A$ and $\cV_A$ are super-braid-reversed equivalent in the sense of the following definition:
\begin{defi}
 Let $\cC=\cC^\even\oplus\cC^\odd$ and $\mathcal{D}=\mathcal{D}^\even\oplus\mathcal{D}^\odd$ be $\ZZ/2\ZZ$-graded braided tensor categories. A \textit{super-braid-reversed equivalence} $(\tau,J,j):\cC\rightarrow\mathcal{D}$ is a tensor equivalence such that for $i\in\ZZ/2\ZZ$, $\tau(X)$ is an object of $\mathcal{D}^i$ if $X$ is an object of $\cC^i$, and such that the diagram
 \begin{equation*}
  \xymatrixcolsep{7pc}
  \xymatrix{
  \tau(X_1)\tens_{\mathcal{D}}\tau(X_2) \ar[r]^{\cR_{\tau(X_1),\tau(X_2)}} \ar[d]^{J_{X_1,X_2}} & \tau(X_2)\tens_{\mathcal{D}}\tau(X_1) \ar[d]^{J_{X_2,X_1}}\\
  \tau(X_1\tens_{\cC} X_2) \ar[r]^{(-1)^{i_1 i_2}\tau(\cR_{X_2,X_1}^{-1})} & \tau(X_2\tens_{\cC} X_1) \\
  }
 \end{equation*}
commutes for objects $X_1$ in $\cC^{i_1}$ and $X_2$ in $\cC^{i_2}$.
\end{defi}

The $\ZZ/2\ZZ$-gradings of $\cU_A$ and $\cV_A$ are defined so that $U_i$ and $V_i$ have degree $j\in\ZZ/2\ZZ$ if $i\in I^j$. In particular, if $i_1\in I^{j_1}$ and $i_2\in I^{j_2}$, then $U_{i_1}\tens_\cU U_{i_2}$ and $V_{i_1}\tens_\cV V_{i_2}$ have degree $j_1+j_2$. We summarize the discussion of this section in the following theorem:
\begin{theo}\label{thm:superalg_mirror_equiv}
 Let $U$ and $V$ be simple self-contragredient vertex operator algebras such that:
 \begin{itemize}
  \item There is an injective conformal vertex algebra homomorphism $\iota_A: U\otimes V\rightarrow A^\even$ with $A=A^\even\oplus A^\odd$ a simple conformal vertex superalgebra of either correct or wrong statistics.
  \item For $j\in\ZZ/2\ZZ$,
  \begin{equation*}
   A^j\cong\bigoplus_{i\in I^j} U_i\otimes V_i
  \end{equation*}
as a $U\otimes V$-module, where the $U_i$ are distinct simple $U$-modules and the $V_i$ are semisimple $V$-modules.

\item Using $0\in I^\even$ to denote the index such that $U_0=U$, we have $V_0=V$ and
\begin{equation*}
 \dim \hom_V(V,V_i)=\delta_{i,0}
\end{equation*}
for $i\in I^\even\cup I^\odd$.

\item The $U$-modules $U_i$ for $i\in I$ are objects of a locally-finite semisimple braided ribbon category $\cU$ of $U$-modules that is closed under contragredients.

\item The $V$-modules $V_i$ for $i\in I$ are objects of a braided tensor category $\cV$ of grading-restricted generalized $V$-modules which is closed under submodules, quotients, and contragredients.
 \end{itemize}
 If $I$ is infinite, assume in addition that:
 \begin{itemize}
  \item Every $V$-module in $\cV$ is finitely generated.
  \item For every intertwining operator $\cY$ of type $\binom{X_3}{M_1\otimes W_1\,M_2\otimes W_2}$, where $M_1$ and $M_2$ are $U$-modules in $\cU$, $W_1$ and $W_2$ are $V$-modules in $\cV$, and $X_3$ is a generalized $U\otimes V$-module in $\ind(\cU\tens\cV)$, the image $\im\cY\subseteq X_3$ is an object of $\cU\tens\cV$.
 \end{itemize}
 Let $\cU_A$ (respectively, $\cV_A$) be the category of $U$-modules (respectively, $V$-modules) whose objects are finite direct sums of the $U_i$ (respectively, of the $V_i$) for $i\in I$. Then:
 \begin{enumerate}
  \item The category $\cU_A$ of $U$-modules is a tensor subcategory of $\cU$.
  
  \item The category $\cV_A$ of $V$-modules is a semisimple ribbon subcategory of $\cV$ with distinct simple objects $\lbrace V_i\rbrace_{i\in I}$. In particular, $\cV_A$ is rigid.
  
  \item There is a super-braid-reversed tensor equivalence $\tau: \cU_A\rightarrow\cV_A$ such that $\tau(U_i)\cong V_i'$ for $i\in I$, where $V_i'$ is the contragredient of $V_i$.
 \end{enumerate}
\end{theo}

\begin{rema}
 In the setting of the preceding theorem, the relation between the twists on $\cU_A$ and $\cV_A$ depends on the statistics of $A$: If $A$ has correct statistics, then for $j\in\ZZ/2\ZZ$ and $i\in I^j$, $\tau(\theta_{U_i})=(-1)^j\theta_{\tau(U_i)}^{-1}$, while if $A$ has wrong statistics, then $\tau(\theta_{U_i})=\theta_{\tau(U_i)}^{-1}$. 
\end{rema}

\section{Application to Virasoro tensor categories}\label{sec:Vir_app}

Let $V_{c(t)}$ be the universal Virasoro vertex operator algebra of central charge $c(t)=13-6t-6t^{-1}$, $t\in\CC^\times$, and let $\cO_t:=\cC^1_{V_{c(t)}}$ be the category of $C_1$-cofinite grading-restricted generalized $V_{c(t)}$-modules. It is proved in \cite{CJORY} that $\cO_t$ admits the vertex algebraic braided tensor category structure of \cite{HLZ1}-\cite{HLZ8}, as described in Section \ref{subsec:VOAs}. In this section, we will use Theorem \ref{thm:mirror_equiv} to show that $\cO_{-p}$ for $p\in\ZZ_{\geq 2}$ contains a rigid semisimple tensor subcategory that is braided equivalent to an abelian $3$-cocycle twist of the category $\rep\mathfrak{sl}_2$ of finite-dimensional $\mathfrak{sl}_2$-modules (the case $p=1$ was addressed in \cite{MY2}). We will then show that $V_{c(-p)}$ is the $PSL_2(\CC)$-fixed point subalgebra of a simple conformal vertex algebra $\cW(-p)$, analogous to the realization of $V_{c(p)}$ for $p\in\ZZ_{\geq 2}$ as the $PSL_2(\CC)$-fixed point subalgebra of the non-rational but $C_2$-cofinite triplet $W$-algebra $\cW(p)$ \cite{AdM, ALM}.

To apply Theorem \ref{thm:mirror_equiv} to $\cO_{-p}$, we want to realize $V_{c(p)}\otimes V_{c(-p)}$ as a vertex operator subalgebra of a simple conformal vertex algebra. The correct conformal vertex algebra to use is the \textit{chiral universal centralizer} $\mathbf{I}^k_{SL_2}$ of $SL_2$ at level $k=-2+\frac{1}{p}$, constructed in \cite[Section 7]{Ar} by a two-step process of quantized Drinfeld-Sokolov reduction: One starts with the conformal vertex algebra $\mathcal{D}_{SL_2,k}^{ch}$ of chiral differential operators on $SL_2$ at level $k$, which is an extension of a tensor product of universal affine vertex operator algebras $V^k(\mathfrak{sl}_2)\otimes V^{k^*}(\mathfrak{sl}_2)$ with $k^*=-2-\frac{1}{p}$. Treating $\mathcal{D}_{SL_2,k}^{ch}$ as an $\NN$-gradable weak $V^k(\mathfrak{sl}_2)$-module, one applies the quantized Drinfeld-Sokolov reduction functor $H_{DS,f}^0$ with respect to the negative root vector $f\in\mathfrak{sl}_2$ to obtain the \textit{equivariant affine $W$-algebra} $\textbf{W}^k_{SL_2,f}:=H_{DS,f}^0(\mathcal{D}^{ch}_{SL_2,k})$ (see \cite[Section 6]{Ar}). One then treats $\mathbf{W}^k_{SL_2,f}$ as an $\NN$-gradable weak $V^{k^*}(\mathfrak{sl}_2)$-module and applies $H_{DS,f}^0$ again to obtain $\mathbf{I}_{SL_2}^k:=H_{DS,f}^0(\mathbf{W}_{SL_2,f}^k)$.

Since the quantized Drinfeld-Sokolov reductions of $V^k(\mathfrak{sl}_2)$ and $V^{k^*}(\mathfrak{sl}_2)$ are the Virasoro vertex operator algebras $V_{c(p)}$ and $V_{c(-p)}$, respectively, $\mathbf{I}_{SL_2}^k$ will be an extension of $V_{c(p)}\otimes V_{c(-p)}$. We need to determine its decomposition as a $V_{c(p)}\otimes V_{c(-p)}$-module. To do so, we first recall some results on Verma and generalized Verma modules for the affine Lie algebra $\widehat{\mathfrak{sl}}_2$. We take a standard basis $\lbrace e,f,h\rbrace$ for $\mathfrak{sl}_2$, and we recall the affine Lie algebra
\begin{equation*}
\widehat{\mathfrak{sl}}_2 =\mathfrak{sl}_2\otimes\CC[t,t^{-1}]\oplus\CC\mathbf{k}.
\end{equation*}
For $a\in\mathfrak{sl}_2$ and $n\in\ZZ$, we use $a(n)$ to denote the action of $a\otimes t^n$ on an $\widehat{\mathfrak{sl}}_2$-module. Such a module $M$ has level $k$ if $\mathbf{k}$ acts on $M$ by the scalar $k\in\CC$. For $k,r\in\CC$, let $M^k_r$ denote the Verma $\widehat{\mathfrak{sl}}_2$-module of level $k$ with a highest-weight vector on which $h(0)$ acts by $r-1$. For $r\in\ZZ_+$, the Verma module $M^k_r$ has the generalized Verma module quotient
\begin{equation*}
V^k_r = \ind^{\widehat{\mathfrak{sl}}_2}_{(\widehat{\mathfrak{sl}}_2)_{\geq 0}} V(r-1),
\end{equation*}
where $(\widehat{\mathfrak{sl}}_2)_{\geq 0}=\mathfrak{sl}_2\otimes\CC[t]\oplus\CC\mathbf{k}$, $V(r-1)$ is the $r$-dimensional simple $\mathfrak{sl}_2$-module, and $a(n)$ acts by zero on $V(r-1)$ for $a\in\mathfrak{sl}_2$, $n>0$. The Verma module $M^k_r$ (and also $V^k_r$ when $r\in\ZZ_+$) has a unique irreducible quotient $L^k_r$ when $k\neq -2$.

By the proof of \cite[Proposition 2.1]{Le} (or \cite[Proposition 3.5]{MY-intw-op}) there is an exact sequence
\begin{equation*}
0\longrightarrow M^k_{-r} \longrightarrow M^k_r\longrightarrow V^k_r\longrightarrow 0
\end{equation*}
for $r\in\ZZ_+$ and any $k$. But for $k=-2\pm\frac{1}{p}$, $p\in\ZZ_+$, there is also an exact sequence
\begin{equation*}
0\longrightarrow M^k_{-r} \longrightarrow M^k_r\longrightarrow L^k_r\longrightarrow 0
\end{equation*}
for $r\in\ZZ_+$. For $k=-2+\frac{1}{p}$, this follows from the second case of \cite[Theorem A(1)]{Ma}, while for $k=-2-\frac{1}{p}$, this follows from the second case of \cite[Theorem A(2)]{Ma}. Thus for $k=-2\pm\frac{1}{p}$, $L^k_r$ is a simple quotient of $V^k_r$ with the same character of $V^k_r$. That is, all generalized Verma modules $V^k_r$ are simple when $k=-2\pm\frac{1}{p}$.

For $k=-2\pm\frac{1}{p}$, $p\in\ZZ_+$, let $V^k(\mathfrak{sl}_2)$ be the level-$k$ universal affine vertex operator algebra (which equals $V^k_1$ as an $\widehat{\mathfrak{sl}}_2$-module and thus is also a simple vertex operator algebra). Also let $KL^k(\mathfrak{sl}_2)$ be the Kazhdan-Lusztig category of finitely-generated grading-restricted generalized $V^k(\mathfrak{sl}_2)$-modules. Because any module $W$ in $KL^k(\mathfrak{sl}_2)$ has finite-dimensional conformal weight spaces, a lower bound on conformal weights, and finitely many generators, $W$ has a finite filtration
\begin{equation*}
0 = W_0\subseteq W_1\subseteq W_2\subseteq\cdots\subseteq W_n= W
\end{equation*}
such that each $W_i/W_{i-1}$ is a quotient of a generalized Verma module (in particular, $W_i/W_{i-1}$ for $i\geq 1$ is the submodule of $W/W_{i-1}$ generated by a finite-dimensional irreducible $\mathfrak{sl}_2$-submodule contained in a minimal conformal weight space of $W/W_{i-1}$). Thus because generalized Verma modules are simple for $k=-2\pm\frac{1}{p}$, every module in $KL^k(\mathfrak{sl}_2)$ has finite length. Moreover, all short exact sequences of the form
\begin{equation*}
0\longrightarrow V^k_{r_1}\longrightarrow W\longrightarrow V^k_{r_2}\longrightarrow 0
\end{equation*}
split (by essentially the same argument as in \cite[Lemma 5.1.1]{CJORY}). Thus $KL^k(\mathfrak{sl}_2)$ is semisimple for $k=-2\pm\frac{1}{p}$, and then so is its direct limit completion $\ind(KL^k(\mathfrak{sl}_2))$. Note that, at least for the levels we are considering, $\ind(KL^k(\mathfrak{sl}_2))$ is the category that is called $KL_k$ in \cite[Section 3]{Ar}; thus a \textit{vertex algebra object in $KL_k$} in the sense of \cite{Ar} is a vertex algebra $V$ containing a vertex subalgebra isomorphic to $V^k(\mathfrak{sl}_2)$ such that $V$ is an object of $\ind(KL^k(\mathfrak{sl}_2))$ as a $V^k(\mathfrak{sl}_2)$-module. 

Now we return to the equivariant affine $W$-algebra $\mathbf{W}^k_{SL_2,f}$ of $SL_2$ at level $k=-2+\frac{1}{p}$, $p\in\ZZ_+$. By \cite[Equation 55]{Ar}, $\mathbf{W}^k_{SL_2,f}$ contains a vertex subalgebra isomorphic to $V_{c(p)}\otimes V^{k^*}(\mathfrak{sl}_2)$ where $k^*=-2-\frac{1}{p}$; moreover, $\mathbf{W}^k_{SL_2,f}$ is an object of $\ind(KL^{k^*}(\mathfrak{sl}_2))$ as a $V^{k^*}(\mathfrak{sl}_2)$-module by \cite[Proposition 6.6]{Ar}. Since $\ind(KL^{k^*}(\mathfrak{sl}_2))$ is semisimple, the decomposition of $\mathbf{W}_{SL_2,f}^k$ as a $V_{c(p)}\otimes V^{k^*}(\mathfrak{sl}_2)$-module is completely determined by the decomposition of its top level
\begin{equation*}
T(\mathbf{W}_{SL_2,f}^k) :=\lbrace v\in\mathbf{W}_{SL_2,f}^k\,\,\vert\,\,a(n)v=0\,\,\text{for all}\,\,a\in\mathfrak{sl}_2, n>0\rbrace
\end{equation*}
as a $V_{c(p)}\times\mathfrak{sl}_2$-module. This decomposition of $T(\mathbf{W}_{SL_2,f}^k)$ is given in \cite[Proposition 6.5]{Ar}, and it follows that
\begin{equation}\label{eqn:equiv_aff_W_decomp}
\mathbf{W}^k_{SL_2,f}\cong\bigoplus_{r\in\ZZ_+} H^0_{DS,f}(V^k_r)\otimes V^{k^*}_r
\end{equation}
as a $V_{c(p)}\otimes V^{k^*}(\mathfrak{sl}_2)$-module. Here the Drinfeld-Sokolov reduction $H^0_{DS,f}(V^k_r)$ is a $V_{c(p)}$-module which is simple (by the Main Theorem of \cite{Ar-KRW-conj} since $V^k_r$ is a simple non-integrable $\widehat{\mathfrak{sl}}_2$-module) and has lowest conformal weight $\frac{r^2-1}{4}p-\frac{r-1}{2}$ (see \cite[Equation 59]{Ar}).

Next we recall from \cite{CJORY} the simple objects in the categories $\cO_{\pm p}$ of $C_1$-cofinite grading-restricted generalized $V_{c(\pm p)}$-modules. These categories contain in particular the simple $V_{c(\pm p)}$-modules $\mathcal{L}_{r,s}^{\pm p}$ for $r,s\in\ZZ_+$ whose lowest conformal weights are
\begin{equation*}
h_{r,s}(\pm p) = \pm\frac{r^2-1}{4}p-\frac{rs-1}{2}\pm\frac{s^2-1}{4p}.
\end{equation*}
Thus $\mathcal{L}_{r,1}^{p}\cong H_{DS,f}^0(V^k_r)$ in the notation of the previous paragraph. Further, when we apply Drinfeld-Sokolov reduction to the equivariant affine $W$-algebra $\mathbf{W}_{SL_2,f}^k$ to obtain the chiral universal centralizer $\mathbf{I}_{SL_2}^k$ of \cite[Section 7]{Ar}, the decomposition \eqref{eqn:equiv_aff_W_decomp} implies that
\begin{equation}\label{eqn:chiral_univ_cent_decomp}
\mathbf{I}_{SL_2}^k\cong\bigoplus_{r\in\ZZ_+} \mathcal{L}_{r,1}^p\otimes\mathcal{L}_{r,1}^{-p}
\end{equation}
as a $V_{c(p)}\otimes V_{c(-p)}$-module. Since the lowest conformal weight of $\mathcal{L}_{r,1}^p\otimes\mathcal{L}_{r,1}^{-p}$ is $-r+1$, $\mathbf{I}^k_{SL_2}$ is a $\mathbb{Z}$-graded conformal vertex algebra. We can see that $\mathbf{I}_{SL_2}^k$ is simple from \cite[Corollary 9.3]{AM} (see also \cite[Theorem 2.2]{Ar}), which states that a vertex algebra is simple if its Li filtration is separated and its associated scheme is a smooth reduced symplectic variety. For $\mathbf{I}_{SL_2}^k$, the first condition follows from \cite[Theorem 6.1(ii)]{Ar} and the second from \cite[Equation 64]{Ar}.

In the notation of Assumption \ref{assum:finite_I}, we now take $U=V_{c(p)}$ and $V=V_{c(-p)}$ for $p\in\ZZ_{\geq 2}$, and we take $A=\mathbf{I}_{SL_2}^k$ with $k=-2+\frac{1}{p}$. Using \eqref{eqn:chiral_univ_cent_decomp}, conditions (1), (2), and (3) of Assumption \ref{assum:finite_I} hold for this choice of $U$, $V$, and $A$. Next, we take the category $\cU$ of Assumption \ref{assum:finite_I}(4) to be the semisimple subcategory $\cO_p^L\subseteq\cO_p$ with simple objects $\cL_{r,1}^p$, $r\in\ZZ_+$. The category $\cO_p^L$ is closed under contragredients since simple $V_{c(p)}$-modules are self-contragredient, and it is a braided ribbon tensor subcategory of $\cO_p$ by  \cite[Theorem 4.3]{MY}. We take the category $\cV$ of Assumption \ref{assum:finite_I}(5) to be $\cO_{-p}$, which is a braided tensor category by \cite[Theorem 4.2.6]{CJORY}. Since every $V_{c(-p)}$-module in $\cO_{-p}$ is finitely generated, Assumption \ref{assum:inf_order}(1) also holds in our setting. Finally, Assumption \ref{assum:inf_order}(2) holds via Proposition \ref{prop:IndC_intw_op_cond}, since $\cO_{-p}$ satisfies the intertwining operator condition of Assumption \ref{assum:inf_order}(2) by \cite[Corollary 2.12]{CMY}. We can now apply Theorem \ref{thm:mirror_equiv}:
\begin{theo}\label{thm:main_thm_app}
For $p\in\ZZ_{\geq 2}$, let $\cO_{-p}^L$ be the semisimple category of $V_{c(-p)}$-modules with simple objects $\cL_{r,1}^{-p}$, $r\in\ZZ_+$. Then $\cO_{-p}^L$ is a rigid tensor subcategory of $\cO_{-p}$, and there is a braid-reversed tensor equivalence $\tau: \cO_p^L\rightarrow\cO_{-p}^L$ such that $\tau(\cL_{r,1}^p)\cong\cL_{r,1}^{-p}$ for all $r\in\ZZ_+$.
\end{theo}

From the $\mathfrak{sl}_2$-type fusion rules for $\cO_p^L$ \cite[Theorem 4.3]{MY}, we then get:
\begin{corol}\label{cor:O_-p_L_fus_rules}
The following tensor product formula holds in $\cO_{-p}$, $p\in\ZZ_{\geq 2}$: for $r,r'\in\ZZ_+$,
\begin{equation*}
 \cL_{r,1}^{-p}\tens\cL_{r',1}^{-p}\cong\bigoplus_{\substack{k=\vert r-r'\vert+1\\ k+r+r'\equiv 1\,(\mathrm{mod}\,2)}}^{r+r'-1} \cL_{k,1}^{-p}.
\end{equation*}

\end{corol}

\begin{rema}
It is shown in the proof of \cite[Theorem 4.3]{MY} that for $p\in\ZZ_{\geq 2}$, $\cO_p^L$ is braided tensor equivalent to a ($p$-dependent) twist of the category $\rep\mathfrak{sl}_2$ of finite-dimensional $\mathfrak{sl}_2$-modules by an abelian $3$-cocycle on $\ZZ/2\ZZ$. Thus the same holds for $\cO_{-p}^L$ (with the inverse abelian $3$-cocycle for $\cO_{-p}^L$ as for $\cO_p^L$).
\end{rema}

\begin{rema}
The $p=1$ versions of Theorem \ref{thm:main_thm_app} and Corollary \ref{cor:O_-p_L_fus_rules} have already been proved directly in \cite{MY2}. The advantage of the direct proof is that one can then use the braid-reversed equivalence to construct the chiral universal centralizer $\mathbf{I}_{SL_2}^{-1}$ by the method of gluing vertex algebras from \cite{CKM2}.
\end{rema}

We can now construct simple conformal vertex algebra extensions of $V_{c(-p)}$, $p\geq 2$, using the method of \cite[Section 7]{MY2} (which covers the $p=1$ case). First, the braid-reversed tensor equivalence $\tau: \cO_p^L\rightarrow\cO_{-p}^L$ of Theorem \ref{thm:main_thm_app} extends to a braid-reversed tensor equivalence $T:\ind(\cO_p^L)\rightarrow\ind(\cO_{-p}^L)$ (these direct limit completions are semisimple tensor subcategories of the braided tensor categories $\ind(\cO_{\pm p})$). Next, we recall the triplet $W$-algebra $\cW(p)$, which is a simple vertex operator algebra extension of $V_{c(p)}$ with automorphism group $PSL_2(\CC)$ such that $V_{c(p)}$ is the $PSL_2(\CC)$-fixed point subalgebra \cite{AdM, ALM}. As a $PSL_2(\CC)\times V_{c(p)}$-module,
\begin{equation*}
\cW(p)\cong\bigoplus_{n=0}^\infty V(2n)\otimes\cL_{2n+1,1}^p,
\end{equation*}
where $V(2n)$ is the $(2n+1)$-dimensional irreducible $PSL_2(\CC)$-module.
\begin{theo}\label{thm:W(-p)}
For $p\in\ZZ_{\geq 2}$, there is a unique simple conformal vertex algebra $\cW(-p)$ of central charge $c(-p)=13+6p+6p^{-1}$ such that
\begin{equation}\label{eqn:W(-p)_decomp}
\cW(-p)\cong\bigoplus_{n=0}^\infty V(2n)\otimes\cL_{2n+1,1}^{-p}
\end{equation}
as a $V_{1c(-p)}$-module. Moreover, $PSL_2(\CC)$ acts by conformal vertex algebra automorphisms on $\cW(-p)$ with fixed-point subalgebra $V_{c(-p)}$ such that \eqref{eqn:W(-p)_decomp} gives the decomposition of $\cW(-p)$ as a $PSL_2(\CC)$-module. 
\end{theo}
\begin{proof}
The proof is essentially the same as that for the $p=1$ version in \cite[Theorem 7.1]{MY2}. Namely, we take $\cW(-p)=T(\cW(p))$, which is a simple commutative algebra in the braided tensor category $\ind(\cO_{-p}^L)$, and thus also a simple conformal vertex algebra extension of $V_{c(-p)}$ with the decomposition \eqref{eqn:W(-p)_decomp} as a $V_{c(-p)}$-module (see \cite[Theorem 3.2]{HKL}, \cite[Theorem 5.7]{CKM2}, and \cite[Theorem 7.5]{CMY}). Then for $g\in\mathrm{Aut}(\cW(p))=PSL_2(\CC)$, $T(g)$ is a commutative algebra automorphism of $\cW(-p)$ and thus also a conformal vertex algebra automorphism. So $PSL_2(\CC)$ acts on $\cW(-p)$, and the same proof as in \cite[Theorem 7.1]{MY2} shows that \eqref{eqn:W(-p)_decomp} gives the decomposition of $\cW(-p)$ as a $PSL_2(\CC)$-module.

To prove the uniqueness, first note that $\cW(-p)$ is a simple conformal vertex algebra in the semisimple symmetric tensor subcategory of $\ind(\cO_{-p}^L)$ with simple objects $\cL_{2n+1,1}^{-p}$, $n\in\NN$. This category is equivalent to the direct limit completion of $\rep PSL_2(\CC)$, which also embeds into the braided tensor category $\ind(\cO_1)$ of modules for the Virasoro algebra at central charge $1$ (see \cite[Example 4.12]{McR}). Thus the simple conformal vertex algebra structure on $\cW(-p)$ is unique if and only if the same holds for any simple vertex operator algebra extension of $V_1$ with decomposition $\bigoplus_{n=0}^\infty V(2n)\otimes\cL_{2n+1,1}^1$ as a $V_1$-module. Indeed, any such vertex operator algebra extension of $V_1$ is isomorphic to the $\mathfrak{sl}_2$-root lattice vertex operator algebra (see for example \cite[Theorem B.1]{MY2}), so $\cW(-p)$ is also unique up to isomorphism.
\end{proof}

Since $PSL_2(\CC)$ acts by automorphisms on $\cW(-p)$ with fixed-point subalgebra $V_{c(-p)}$, the $U(1)$-fixed-point subalgebra is another conformal vertex algebra extension of $V_{c(-p)}$. This conformal vertex algebra can also be constructed by applying the braid-reversed tensor equivalence $T$ to the singlet $W$-algebra $\cM(p)$ \cite{Ad}, which is the $U(1)$-fixed-point subalgebra of $\cW(p)$. That is, defining $\cM(-p)=T(\cM(p))$, we get:
\begin{propo}\label{prop:M(-p)}
For $p\in\ZZ_{\geq 2}$, there is a simple conformal vertex algebra $\cM(-p)$ of central charge $c(-p)$ such that as a $V_{c(-p)}$-module,
\begin{equation*}
\cM(-p)\cong\bigoplus_{n=0}^\infty \cL_{2n+1,1}^{-p}.
\end{equation*}
\end{propo}

Once the detailed tensor structure of $\cO_{-p}$ is better understood, it will be possible to study the representations of $\cW(-p)$ and $\cM(-p)$. We conjecture that the braid-reversed tensor equivalence $\tau:\cO_p^L\rightarrow\cO_{-p}^L$ of Theorem \ref{thm:main_thm_app} extends to a braid-reversed tensor equivalence from the subcategory $\cO_p^0\subseteq\cO_p$ defined in \cite[Section 5.1]{MY} to a suitable subcategory of $\cO_{-p}$. If so, then by \cite[Corollary A.4]{MY2}, $\cW(-p)$ would admit a tensor category of modules that would be braid-reversed tensor equivalent to the category of grading-restricted generalized $\cW(p)$-modules, and $\cM(-p)$ would admit a tensor category of modules that would be braid-reversed tensor equivalent to the subcategory $\cC_{\cM(p)}^0$ of $\cM(p)$-modules defined in \cite{CMY-sing}. By \cite{GN}, the category of grading-restricted generalized $\cW(p)$-modules is a non-semisimple modular tensor category equivalent to the module category for a quasi-Hopf modification of the small quantum group $u_q(\mathfrak{sl}_2)$, $q=e^{\pi i/p}$, and the braided tensor category $\cC_{\cM(p)}^0$ of $\cM(p)$-modules is equivalent to a module subcategory for the unrolled quantum group of $\mathfrak{sl}_2$ at $q=e^{\pi i/p}$.

\begin{rema}
In contrast with the singlet and triplet algebras $\cM(p)$ and $\cW(p)$ for $p\geq 2$, the algebras $\cM(-p)$ and $\cW(-p)$ do not seem to have obvious free field realizations. For example, while $\cM(p)$ is realized inside a rank-one Heisenberg Fock module with a modified Virasoro structure (that is, a Feigin-Fuchs module for the Virasoro algebra), the same cannot hold $\cM(-p)$, since Feigin-Fuchs modules for the Virasoro algebra at central charge $13+6p+6p^{-1}$ have finite length.
\end{rema}

\appendix

\section{Proofs for general coset theorems}\label{app:gen_coset_proofs}

In this appendix, we give proofs of Theorems \ref{thm:fund_coset} and \ref{thm:decomp_of_UV-mod}. Throughout, we use formal calculus as developed in \cite{FLM, LL}. Especially, we need the formal delta function $\delta(x)=\sum_{n\in\ZZ} x^n$ and its properties discussed in \cite[Chapter 2]{LL}. We also use the binomial expansion convention $(x+y)^h=\sum_{i\geq 0} \binom{h}{i} x^{h-i} y^i$ for $h\in\CC$, so that $(x+y)^h\neq (y+x)^h$ unless $h\in\NN$.

\subsection{Proof of Theorem \ref{thm:decomp_of_UV-mod}}

Since the second assertion of Theorem \ref{thm:fund_coset} is essentially a special case of Theorem \ref{thm:decomp_of_UV-mod}, we prove Theorem \ref{thm:decomp_of_UV-mod} first. Thus suppose $U$ and $V$ are vertex operator algebras and $X$ is a weak $U\otimes V$-module that is semisimple as a weak $U$-module. Let $\lbrace M_j\rbrace_{j\in J}$ denote the distinct irreducible grading-restricted $U$-modules that occur in $X$, so that as a weak $U$-module,
 \begin{equation*}
  X=\bigoplus_{j\in J} X_j
 \end{equation*}
where $X_j$ is the sum of all $U$-submodules of $X$ which are isomorphic to $M_j$. Then for all $j\in J$, we have the weak $U$-module homomorphism
\begin{align*}
 \varphi^{(j)}_X: M_j\otimes\hom_U(M_j,X) & \rightarrow X_j
\end{align*}
defined by $\varphi^{(j)}_X(m_j\otimes f)=f(m_j)$ for $m_j\in M_j$ and $f\in\hom_U(M_j,X)$. We claim that each $\varphi^{(j)}_X$ is actually an isomorphism, so that 
$$\varphi_X=\bigoplus_{j\in J} \varphi^{(j)}_X:\,\, \bigoplus_{j\in J} M_j\otimes\hom_U(M_j,X)\longrightarrow\bigoplus_{j\in J} X_j\quad( = X)$$ 
is an isomorphism of weak $U$-modules.

To prove the claim, note that since each $X_j$ is the sum of $U$-submodules $\til{M}$ isomorphic to $M_j$, and since each such $\til{M}$ is the image of some $U$-module homomorphism $f: M_j\rightarrow X$, each $\varphi^{(j)}_X$ is surjective. To show that $\varphi^{(j)}_X$ is injective, suppose 
 \begin{equation*}
  \varphi^{(j)}_X\left(\sum m^{(k)}_j\otimes f_k\right) = \sum f_k(m^{(k)}_j) =0
 \end{equation*}
for linearly independent $m^{(k)}_j\in M_j$ and $f_k\in\hom_U(M_j,X)$. Since each $f_k$ is a $U$-module homomorphism and thus preserves $L_U(0)$-eigenspaces, we may take each $m_j^{(k)}$ to be homogeneous of the same $L_U(0)$-weight $h$. Since $M_j$ is an irreducible $U$-module, $(M_j)_{[h]}$ is an irreducible module for the algebra of degree-preserving operators in the subalgebra of $\Endo_\CC(M_j)$ generated by the $u_n$ for $u\in U$, $n\in\ZZ$. Then the Jacobson Density Theorem implies there exist degree-preserving operators $a_k$ such that
\begin{equation*}
 a_k\cdot m^{(l)}_j =\delta_{k,l} m^{(l)}_j.
\end{equation*}
Then for each $k$,
\begin{equation*}
 f_k(m^{(k)}_j) = \sum_l f_l(\delta_{k,l} m^{(l)}_j) =\sum_l f_l(a_k\cdot m^{(l)}_j) = a_k\cdot\sum_l f_l(m^{(l)}_j) =0.
\end{equation*}
Thus $\ker f_k\neq 0$ for each $k$. Since $M_j$ is irreducible, each $f_k=0$, and so $\varphi^{(j)}_X$ is injective.

Now we need to show that each $W_j=\hom_U(M_j,X)$ is a weak $V$-module and that $\varphi_X$ is a homomorphism of weak $U\otimes V$-modules. First we define the vertex operator
\begin{equation*}
 Y_{W_j}: V\otimes W_j\rightarrow W_j[[x,x^{-1}]]
\end{equation*}
by
\begin{equation*}
 [Y_{W_j}(v,x)f](m_j) = Y_X(\vac_U\otimes v,x)f(m_j)
\end{equation*}
for $v\in V$, $f\in W_j$, and $m_j\in M_j$. To see that indeed $Y_{W_j}(v,x)f\in W_j[[x,x^{-1}]]$, we need to check that each coefficient of $Y_{W_j}(v,x)f$ is a $U$-module homomorphism from $M_j$ to $X$:
\begin{align*}
 [Y_{W_j}(v,x)f](Y_{M_j}(u,y)m_j) & = Y_X(\vac_U\otimes v,x)f(Y_{M_j}(u,y)m_j)\nonumber\\
 & = Y_X(\vac_U\otimes v,x)Y_X(u\otimes\vac_V,y)f(m_j)\nonumber\\
 & = Y_X(u\otimes\vac_V,y)Y_X(\vac_U\otimes v,x)f(m_j)\nonumber\\
 & = Y_X(u\otimes\vac_V,y)[Y_{W_j}(v,x)f](m_j)
\end{align*}
for any $u\in U$, $m_j\in M_j$. For the third equality, the relation
\begin{equation}\label{eqn:U_V_commute}
 [Y_X(\vac_U\otimes v,x), Y_X(u\otimes\vac_V,y)]=0
\end{equation}
for $u\in U$, $v\in V$ follows from the commutator formula for the vertex operator $Y_X$ acting on the weak $U\otimes V$-module $X$.

Before showing that $(W_j,Y_{W_j})$ is a weak $V$-module, we show that for any $m_j\in M_j$,
 \begin{align*}
  \varphi_{m_j}=\varphi^{(j)}_X(m_j\otimes\cdot): W_j & \rightarrow X,
 \end{align*}
 intertwines the vertex operator actions of $V$ on $W_j$ and $X$:
 \begin{align}\label{eqn:psi_ui_V_hom}
 \varphi_{m_j}\left(Y_{W_j}(v,x)f\right) & = \left[Y_{W_j}(v,x)f\right](m_j)\nonumber\\
 &= Y_X(\vac_U\otimes v,x)f(m_j) = Y_X(\vac_U\otimes v,x)\varphi_{m_j}(f)
\end{align}
for any $v\in V$, $f\in W_j$. Using this relation, we show that the vertex operator $Y_{W_j}$ is lower truncated, that is, $Y_{W_j}(v,x)f\in W_j((x))$ for any $v\in V$, $f\in W_j$: First, for any $m_j\in M_j$,
\begin{equation*}
 \varphi_{m_j}\left(Y_{W_j}(v,x)f\right) = Y_X(\vac_U\otimes v,x)\varphi_{m_j}(f)\in X((x))
\end{equation*}
by the lower truncation of $Y_X$, that is, $\varphi_{m_j}(v_n f)=0$ for $n\in\ZZ$ sufficiently positive. As $\varphi^{(j)}_X$ is injective, $\varphi_{m_j}$ is also injective if $m_j\neq 0$. So $v_n f=0$ for $n\in\ZZ$ sufficiently positive, that is, $Y_{W_j}(v,x)f\in W_j((x))$.

\allowdisplaybreaks

The remaining properties of a  weak $V$-module that we need to prove for $W_j$ are the vacuum property, the Jacobi identity, and the $L_V(-1)$-derivative property. The vacuum property follows from the vacuum property for $Y_X$:
\begin{equation*}
 \left[Y_{W_j}(\vac,x)f\right](m_j) =Y_X(\vac,x)f(m_j)=f(m_j) =\left[\Id_{W_j}(f)\right](m_j)
\end{equation*}
 for $f\in W_j$, $m_j\in M_j$. To prove the Jacobi identity for $W_j$, we use the Jacobi identity for $X$:
\begin{align*}
& x_0^{-1}\delta\left(\frac{x_1-x_2}{x_0}\right) \left[ Y_{W_j}(v_1,x_1)Y_{W_j}(v_2,x_2)f\right](m_j)\nonumber\\
&\hspace{8em}- x_0^{-1}\delta\left(\frac{-x_2+x_1}{x_0}\right) \left[Y_{W_j}(v_2,x_2)Y_{W_j}(v_1,x_1)f\right](m_j)\nonumber\\
&\hspace{2em} = x_0^{-1}\delta\left(\frac{x_1-x_2}{x_0}\right) Y_X(\vac_U\otimes v_1,x_1)\left[Y_{W_j}(v_2,x_2)f\right](m_j)\nonumber\\
&\hspace{8em}-x_0^{-1}\delta\left(\frac{-x_2+x_1}{x_0}\right) Y_X(\vac_U\otimes v_2,x_2)\left[Y_{W_j}(v_2,x_2)f\right](m_j)\nonumber\\
&\hspace{2em} =x_0^{-1}\delta\left(\frac{x_1-x_2}{x_0}\right) Y_X(\vac_U\otimes v_1,x_1)Y_X(\vac_U\otimes v_2,x_2)f(m_j)\nonumber\\
&\hspace{8em} -x_0^{-1}\delta\left(\frac{-x_2+x_1}{x_0}\right) Y_X(\vac_U\otimes v_2,x_2)Y_X(\vac_U\otimes v_1,x_1)f(m_j)\nonumber\\
&\hspace{2em} = x_2^{-1}\delta\left(\frac{x_1-x_0}{x_2}\right) Y_X(Y_{U\otimes V}(\vac_U\otimes v_1,x_0)(\vac_U\otimes v_2), x_2)f(m_j)\nonumber\\
&\hspace{2em} = x_2^{-1}\delta\left(\frac{x_1-x_0}{x_2}\right) Y_X(\vac_U\otimes Y_V(v_1,x_0)v_2, x_2)f(m_j)\nonumber\\
& \hspace{2em} = x_2^{-1}\delta\left(\frac{x_1-x_0}{x_2}\right) \left[Y_{W_j}(Y_V(v_1,x_0)v_2,x_2)f\right](m_j)
\end{align*}
for any $v_1, v_2\in V$, $f\in W_j$, $m_j\in M_j$. For the $L_V(-1)$-derivative property,
\begin{align*}
 \left[ \frac{d}{dx} Y_{W_j}(v,x)f\right](m_j) & = \frac{d}{dx} Y_X(\vac_U\otimes v,x)f(m_j) = Y_X(L(-1)(\vac_U\otimes v),x)f(m_j)\nonumber\\
 &= Y(\vac_U\otimes L_V(-1)v,x)f(m_j) =\left[Y_{W_j}(L_V(-1)v,x)f\right](m_j)
\end{align*}
for $v\in V$, $f\in W_j$, and $m_j\in M_j$, where the third equality uses $L(-1)=L_U(-1)\otimes\Id_V+\Id_U\otimes L_V(-1)$.

Now that each $W_j$ is a weak $V$-module, we verify that $\varphi^{(j)}_X$ is a homomorphism of weak $U\otimes V$-modules:
\begin{align*}
 \varphi^{(j)}_X  \left(Y_{M_j\otimes W_j}(u\otimes v, x)(m_j\otimes f)\right) & =\varphi^{(j)}_X\left(Y_{M_j}(u,x)m_j\otimes Y_{W_j}(v,x)f\right)\nonumber\\
 &\hspace{-7em} =[Y_{W_j}(v,x)f](Y_{M_j}(u,x)m_j) =Y_{X}(\vac_U\otimes v,x)f\left(Y_{M_j}(u,x)m_j\right)\nonumber\\
 &\hspace{-7em} = Y_X(\vac_U\otimes v,x)Y_X(u\otimes\vac_V,x)f(m_j) = Y_X(u\otimes v, x) \varphi^{(j)}_X(m_j\otimes f).
\end{align*}
for $u\in U$, $v\in V$, $m_j\in M_j$, $f\in W_j$. The last equality comes from the Jacobi identity for $X$:
\begin{align}\label{eqn:Jacobi_calc}
 Y_X(u\otimes v,x) & = Y_X((\vac_U\otimes v)_{-1}(u\otimes\vac_V), x)\nonumber\\
 & =\mathrm{Res}_y\,y^{-1} Y_X(Y_{U\otimes V}(\vac_U\otimes v,y)(u\otimes\vac_V),x)\nonumber\\
 & =\mathrm{Res}_y\,\mathrm{Res}_z\,y^{-1} z^{-1}\delta\left(\frac{x+y}{z}\right) Y_X(Y_{U\otimes V}(\vac_U\otimes v,y)(u\otimes\vac_V),x)\nonumber\\
 & =\mathrm{Res}_y\,\mathrm{Res}_z\, y^{-2}\delta\left(\frac{z-x}{y}\right) Y_X(\vac_U\otimes v, z)Y_X(u\otimes\vac_V,x)\nonumber\\
 & \hspace{4em} -\mathrm{Res}_y\,\mathrm{Res}_z\, y^{-2}\delta\left(\frac{-x+z}{y}\right) Y_X(u\otimes \vac_V, x)Y_X(\vac_U\otimes v,z)\nonumber\\
 & =\mathrm{Res}_z\,\left((z-x)^{-1}+(x-z)^{-1}\right) Y_X(\vac_U\otimes v, z)Y_X(u\otimes\vac_V,x)\nonumber\\
 & =\mathrm{Res}_z\,z^{-1}\delta\left(\frac{x}{z}\right) Y_X(\vac_U\otimes v, x)Y_X(u\otimes\vac_V,x)\nonumber\\
 & = Y_X(\vac_U\otimes v, x)Y_X(u\otimes\vac_V,x).
\end{align}
for $u\in U$ and $v\in V$, where we have used \eqref{eqn:U_V_commute} for the fifth equality and delta function substitution properties for the sixth.

We have now finished the proof that there is a weak $U\otimes V$-module isomorphism
\begin{equation*}
 \varphi_X: \bigoplus_{j\in J} M_j\otimes W_j\longrightarrow X
\end{equation*}
with $W_j=\hom_U(M_j,X)$. To complete the proof of Theorem \ref{thm:decomp_of_UV-mod}, we need to show that each $W_j$ is a (grading-restricted) generalized $V$-module if $X$ is a (grading-restricted) generalized $U\otimes V$-module. First suppose that $X$ is a generalized $U\otimes V$-module, that is, $X$ is the direct sum of its generalized $L(0)$-eigenspaces. Since $X$ is also a semisimple weak $U$-module, it is also the direct sum of its $L_U(0)$-eigenspaces. Each $L_U(0)$-eigenspace $X_{[h_U]}$ with $L_U(0)$-eigenvalue $h_U\in\CC$ is $L(0)$-invariant since $L(0)=L_U(0)+L_V(0)$ commutes with $L_U(0)$. Thus each $X_{[h_U]}$ is the direct sum of generalized eigenspaces for 
$$L_V(0)\vert_{X_{[h_U]}}=L(0)\vert_{X_{[h_U]}}-h_U$$
(see for example \cite[Remark 2.13]{HLZ1}), showing that $X$ is the direct sum of generalized $L_V(0)$-eigenspaces. Now since each $W_j$ embeds into $X$ as a weak $V$-submodule which is in particular an $L_V(0)$-invariant subspace of $X$, it follows that each $W_j$ is the direct sum of its generalized $L_V(0)$-eigenspaces, that is, $W_j$ is a generalized $V$-module.

Finally, we need to show that if the generalized $L(0)$-eigenvalue grading of $X$ satisfies the grading-restriction conditions, then so does the generalized $L_V(0)$-eigenvalue grading of $W_j$ for each $j\in J$. First, because $M_j$ is irreducible, the eigenvalues of $L_U(0)$ on $M_j$ are contained in $h_j+\NN$ for some $h_j\in\CC$. Take any non-zero $m_j\in (M_j)_{[h_j]}$ and recall the $V$-module injection $\varphi_{m_j}=\varphi_X^{(j)}(m_j\otimes\cdot): W_j\rightarrow X$. Now suppose that $h\in\CC$ is a generalized eigenvalue of $L_V(0)$ on $W_j$ and $f\in (W_j)_{[h]}$, so that $(L_V(0)-h)^N\cdot f=0$ for some $N\in\ZZ_+$. We calculate
\begin{align*}
 (L(0)-h_j-h)^N\cdot\varphi_{m_j}(f) & = \sum_{n=0}^N \binom{N}{n} (L_U(0)-h_j)^{N-n} (L_V(0)-h)^n\cdot f(m_j)\nonumber\\
 & = \sum_{n=0}^N \binom{N}{n} (L_V(0)-h)^n\cdot f\left((L_U(0)-h_j)^{N-n}\cdot m_j\right)\nonumber\\
 & =(L_V(0)-h)^N\cdot f(m_j)\nonumber\\
 & =\left[(L_V(0)-h)^N\cdot f\right](m_j) = 0,
\end{align*}
showing that $\varphi_{m_j}$ maps $(W_j)_{[h]}$ injectively into $X_{[h_j+h]}$. Since $X_{[h_j+h]}$ is finite dimensional, so is $(W_j)_{[h]}$. Moreover, since $X_{[h_j+h+n]}=0$ for $n\in\ZZ$ sufficiently negative, so is $(W_j)_{[h+n]}$. This completes the proof of Theorem \ref{thm:decomp_of_UV-mod}.

\subsection{Proof of Theorem \ref{thm:fund_coset}}
We begin by recalling the setting of Theorem \ref{thm:fund_coset}:
\begin{itemize}
 \item $(A, Y,\vac,\omega)$ is a simple positive-energy vertex operator algebra.
 
 \item $(U, Y_U =Y\vert_{U\otimes U},\vac, \omega_U)$ is a simple vertex operator subalgebra of $A$ with conformal vector $\omega_U\in A_{(2)}\cap U$ such that $L(1)\omega_U=0$.
 
 \item As a weak $U$-module, $A$ is the direct sum of irreducible grading-restricted $U$-modules.
 
 \item The vertex subalgebra $U$ equals its own double commutant, that is, if $V=C_A(U)$, then $U=C_A(V)$.
\end{itemize}
The coset $V=C_A(U)$ is a vertex operator algebra with conformal vector $\omega_V=\omega-\omega_U$, and both $U$ and $V$ are positive-energy vertex operator algebras.

Our first goal is to prove that the linear map $\psi: U\otimes V\rightarrow A$ defined by $\psi(u\otimes v)=u_{-1} v$ for $u\in U$, $v\in V$ is a vertex operator algebra homomorphism. We begin with a lemma whose proof amounts to the same calculation as \eqref{eqn:Jacobi_calc}:
\begin{lemma}\label{lem:UV_hom_into_A}
 For $u\in U$ and $v\in V$, $Y(u_{-1} v, x) = Y(u,x)Y(v,x)$.
\end{lemma}

The product $Y(u,x)Y(v,x)$ in this lemma is well defined only because the vertex operators for $u$ and $v$ commute. Now we can prove that $\psi$ is a vertex operator homomorphism: it preserves vacuum and conformal vectors because  $\vac_{-1}\vac=\vac$ and
\begin{equation*}
 (\omega_U)_{-1}\vac +\vac_{-1}\omega_V =\omega_U+\omega_V=\omega.
\end{equation*}
Then for $u_1,u_2\in U$ and $v_1,v_2\in V$, we have
\begin{align}\label{eqn:Y_psi_psi}
 Y(\psi(u_1\otimes v_1),x)\psi(u_2\otimes v_2) & = Y((u_1)_{-1} v_1,x)(u_2)_{-1} v_2\nonumber\\
 & = Y(u_1,x)Y(v_1,x)(u_2)_{-1} v_2\nonumber\\
 & = Y(u_1,x)(u_2)_{-1}Y(v_1,x)v_2
\end{align}
using Lemma \ref{lem:UV_hom_into_A}. On the other hand,
\begin{align}\label{eqn:psi_Y}
 \psi (Y_{U\otimes V}(u_1\otimes v_1,x) & (u_2\otimes v_2)) = \psi\left( Y_U(u_1,x)u_2\otimes Y_V(v_1,x)v_2\right)\nonumber\\
 & = \left(Y(u_1,x)u_2\right)_{-1} Y(v_1,x)v_2\nonumber\\
 & = \mathrm{Res}_y\,y^{-1}Y(Y(u_1,x)u_2,y)Y(v_1,x)v_2\nonumber\\
 & = \mathrm{Res}_y\,\mathrm{Res}_z\,y^{-1} z^{-1}\delta\left(\frac{y+x}{z}\right)Y(Y(u_1,x)u_2,y)Y(v_1,x)v_2\nonumber\\
 & = \mathrm{Res}_y\,\mathrm{Res}_z\,y^{-1}\left( x^{-1}\delta\left(\frac{z-y}{x}\right)Y(u_1,z)Y(u_2,y)\right.\nonumber\\
 &\hspace{10em}\left.-x^{-1}\delta\left(\frac{y-z}{-x}\right)Y(u_2,y)Y(u_1,z)\right)Y(v_1,x)v_2\nonumber\\
 & =\mathrm{Res}_y\,\mathrm{Res}_z\,\left(y^{-1} z^{-1}\delta\left(\frac{x+y}{z}\right)Y(u_1,z)Y(u_2,y)Y(v_1,x)v_2\right.\nonumber\\ &\hspace{10em}\left.+y^{-2}\delta\left(\frac{-x+z}{y}\right)Y(u_2,y)Y(u_1,z)Y(v_1,x)v_2\right)\nonumber\\
 & =\mathrm{Res}_y\,y^{-1}Y(u_1,x+y)Y(u_2,y)Y(v_1,x)v_2\nonumber\\
 &\hspace{5em}+\mathrm{Res}_z\,(-x+z)^{-1}Y(u_2,-x+z)Y(u_1,z)Y(v_1,x)v_2.
\end{align}
It is important to observe that, because vertex operators for vectors in $U$ commute with vertex operators for vectors in $V$, all expressions and delta function substitutions here are well defined. Now, by the second equality in \eqref{eqn:comm_char}, the second series in \eqref{eqn:psi_Y} involves only non-negative powers of $z$, and hence $\mathrm{Res}_z$ vanishes. For the same reason,
\begin{align*}
 \mathrm{Res}_y\,y^{-1} Y(u_1,x+y)Y(u_2,y)Y(v_1,x)v_2 &=\lim_{y\to 0} Y(u_1,x+y)Y(u_2,y)Y(v_1,x)v_2\nonumber\\
 &=Y(u_1,x)(u_2)_{-1}Y(v_1,x)v_2,
\end{align*}
which agrees with \eqref{eqn:Y_psi_psi}. This shows that $\psi$ is a homomorphism of vertex algebras.

We have not yet shown that $\psi$ is injective, but it at least gives $A$ the structure of a (grading-restricted) $U\otimes V$-module with vertex operator 
\begin{align*}
 Y_A: (U\otimes V)\otimes A & \rightarrow A((x))\nonumber\\
 (u\otimes v)\otimes a & \mapsto Y(\psi(u\otimes v),x)a =Y(u_{-1} v,x)a =Y(u,x)Y(v,x)a
\end{align*}
(recall Lemma \ref{lem:UV_hom_into_A}). Thus Theorem \ref{thm:decomp_of_UV-mod} gives a decomposition of $A$ as a $U\otimes V$-module: Let $\lbrace U_i\rbrace_{i\in I}$ denote the set of distinct irreducible (grading-restricted) $U$-modules that occur in $A$. Since $U$ is simple and occurs in $A$, we fix an index $i=0$ such that $U_0=U$. For each $i\in I$, $V_i=\hom_U(U_i,A)$ is a grading-restricted generalized $V$-module by Theorem \ref{thm:decomp_of_UV-mod}, and there is a $U\otimes V$-module isomorphism
\begin{equation*}
 \varphi: \bigoplus_{i\in I} U_i\otimes V_i \longrightarrow A
\end{equation*}
defined by
\begin{equation*}
 \varphi(u_i\otimes f) =f(u_i)
\end{equation*}
for $u_i\in U_i$ and $f\in V_i$. In fact, each $V_i$ is an ordinary $V$-module in the sense that $L_V(0)$ acts semisimply on $V_i$; this is because $V_i$ occurs as a $V$-submodule of $A$, because $L_V(0)=L(0)-L_U(0)$ on $A$, and because the commuting operators $L(0)$ and $L_U(0)$ both act semisimply on $A$. This proves the second assertion of Theorem \ref{thm:fund_coset}.

We would like to show that the vertex operator algebra homomorphism $\psi: U\otimes V\rightarrow A$ is injective next. But first, we need to show $V_0\cong V$ as $V$-modules:
\begin{propo}
 The image in $A$ of the injection $\varphi_\vac=\varphi(\vac\otimes\cdot): V_0\rightarrow A$ is $V$. In particular, $\varphi_\vac$ determines a $V$-module isomorphism from $V_0$ to $V$.
\end{propo}
\begin{proof}
 By definition, $\varphi_\vac(f)=f(\vac)$ for any $f\in V_0=\hom_U(U,A)$. Then for any $u\in U$,
 \begin{equation*}
  Y(u,x)f(\vac)=f(Y(u,x)\vac)\in A[[x]],
 \end{equation*}
showing that $f(\vac)\in C_U(A)=V$ by the first equality of \eqref{eqn:comm_char}.

Conversely, if $v\in V$, we define $f_v: U\rightarrow A$ by $f_v(u)=v_{-1} u$ for $u\in U$. The linear map $f_v$ is a $U$-module homomorphism because
\begin{equation*}
 f_v(Y_U(u_1,x)u_2) = v_{-1}Y(u_1,x)u_2=Y(u_1,x)v_{-1} u_2=Y(u_1,x)f_v(u_2)
\end{equation*}
for any $u_1,u_2\in U$. Then
\begin{equation*}
 \varphi_\vac(f_v)=f_v(\vac)=v_{-1}\vac=v,
\end{equation*}
showing that $v\in\im \varphi_\vac$.

We have shown that $\im\varphi_\vac=V$, so $\varphi_\vac$ defines a linear isomorphism from $V_0$ to $V$. Then \eqref{eqn:psi_ui_V_hom} in the case $X=A$, $m_j=\vac$, $W_j=V_0$ shows $\varphi_\vac$ is also a $V$-module homomorphism.
\end{proof}

Now we show $\psi$ is injective, completing the proof of the first assertion of Theorem \ref{thm:fund_coset}:
\begin{propo}\label{prop:psi_diag}
 The diagram of $U\otimes V$-module homomorphisms
 \begin{equation*}
  \xymatrixcolsep{4pc}
  \xymatrix{
  U_0\otimes V_0 \ar[r]^{\varphi} \ar[d]_{\Id_U\otimes\varphi_\vac} & A \\
  U\otimes V \ar[ru]_{\psi} & \\
  }
 \end{equation*}
commutes. In particular, the vertex operator algebra homomorphism $\psi$ is injective.
\end{propo}
\begin{proof}
 For any $u\in U$, $f\in V_0$, we have
 \begin{equation*}
  \psi(u\otimes\varphi_\vac(f)) =\psi(u\otimes f(\vac)) =u_{-1}f(\vac)=f(u_{-1}\vac)=f(u)=\varphi(u\otimes f),
 \end{equation*}
showing that $\psi\circ(\Id_U\otimes\varphi_1\vac)=\varphi\vert_{U_0\otimes V_0}$. Then injectivity of $\psi$ follows from injectivity of $\varphi$ and the fact that $\varphi_\vac: V_0\rightarrow V$ is an isomorphism.
\end{proof}

Since $V_0\cong V$ as $V$-modules, $\dim\hom_V(V,V_0)\geq 1$. We will get equality, as in the third assertion of Theorem \ref{thm:fund_coset}, if the fourth assertion, that $V$ is simple, holds. Indeed, because $U\otimes V\cong U_0\otimes V_0$ occurs as a $U\otimes V$-module direct summand of $A$, and because $A$ is simple, an argument given in the proof of \cite[Theorem 1.2]{McR-fin-ind-ss} shows that $U\otimes V$ is simple. Then it is easy to see $V$ is simple as well (see \cite[Proposition 4.7.2]{FHL}).

It remains to prove the $i\neq 0$ case of the third assertion of Theorem \ref{thm:fund_coset}. Thus we finally eliminate the possibility of non-zero homomorphisms from $V$ to $V_i$ for $i\neq 0$ using the assumption that $U=C_A(V)$:
\begin{propo}
 For $i\neq 0$, $\hom_V(V,V_i)=0$.
\end{propo}
\begin{proof}
 Take a $V$-module homomorphism $F: V\rightarrow V_i=\hom_U(U_i,A)$, denoted by $v\mapsto F[v]$. Then for any $u_i\in U_i$, $F[\vac](u_i)\in A$ satisfies
 \begin{equation*}
  Y(v,x)F[\vac](u_i) = Y_A(\vac\otimes v,x)F[\vac](u_i)
 = \left[Y_{V_i}(v,x)F[\vac]\right](u_i)=F\left[Y_V(v,x)\vac\right](u_i)\in A[[x]]
 \end{equation*}
for all $v\in V$. As in \eqref{eqn:comm_char}, this means $F[\vac](u_i)\in C_A(V)=U$ for any $u_i\in U_i$, that is, $F[\vac]$ defines a $U$-module homomorphism from $U_i$ to $U$. Since $i\neq 0$, $U_i$ and $U$ are non-isomorphic irreducible $U$-modules, meaning $F[\vac]=0$. But then
\begin{equation*}
 F[v]=F[v_{-1}\vac]=v_{-1}F[\vac]=0
\end{equation*}
for any $v\in V$, showing $F=0$.
\end{proof}

\section{Proof of Proposition \ref{prop:IndC_intw_op_cond}}\label{app:intw_op_proof}

Let $U$ and $V$ be vertex operator algebras, and let $\cU$ and $\cV$ be categories of grading-restricted generalized $U$- and $V$-modules, respectively, such that:
\begin{itemize}
 \item Every $U$-module in $\cU$ is semisimple and $C_1$-cofinite.
 
 \item The category $\cV$ is closed under submodules, quotients, and finite direct sums.
\end{itemize}
We need to show that if $\cV$ satisfies the intertwining operator condition of Assumption \ref{assum:inf_order}(2), then so does the category $\cC=\cU\tens\cV$.

Suppose $\cY$ is an intertwining operator of type $\binom{X_3}{X_1\,X_2}$ where $X_1$, $X_2$ are objects of $\cC$ and $X_3$ is an object of $\ind(\cC)$; we need to show that $\im\cY$ is an object of $\cC$. Since $\cU$ is semisimple and since modules in $\cU$ are $C_1$-cofinite, every module in $\cU$ is a finite direct sum of irreducible $U$-modules in $\cU$. Therefore $\im\cY$ will be a finite sum of images of intertwining operators of type $\binom{X_3}{M_1\otimes W_1\,M_2\otimes W_2}$ where $M_1$, $M_2$ are irreducible objects in $\cU$ and $W_1$, $W_2$ are objects in $\cV$. Since a sum of submodules is a quotient of a direct sum, and because $\cC$ is closed under quotients and finite direct sums, it is enough to consider the case $X_1=M_1\otimes W_1$ and $X_2=M_2\otimes W_2$.

Now what can we say about $X_3$? Since $X_3$ is the union of its $\cC$-submodules, it is the union of semisimple $U$-submodules in $\cU$. That is, $X_3$ is a generalized $U\otimes V$-modules which, as a weak $U$-module is the direct sum of irreducible grading-restricted $U$-modules. Thus by Theorem \ref{thm:decomp_of_UV-mod},
\begin{equation*}
 X_3\cong\bigoplus_{j\in J} M_3^{j}\otimes W_3^{j}
\end{equation*}
as a $U\otimes V$-module, where $\lbrace M_3^j\rbrace_{j\in J}$ are the distinct irreducible $U$-modules from $\cU$ occurring in $X_3$ and each $W_3^{j}=\hom_U(M_3^{j},X_3)$ is a generalized $V$-module.

We will need each $W_3^{j}$ to be an object of $\ind(\cV)$, that is, any $f\in W_3^{j}=\hom_U(M_3^{j},X_3)$ should be contained in a $V$-submodule that is an object of $\cV$. To show this, note that any non-zero $m_j\in M_3^{j}$ determines a $V$-module embedding of $W_3^{j}$ into $X_3$ under which $f$ maps to $f(m_j)$. So we need $f(m_j)$ to be contained in a $V$-submodule of $X_3$ which is an object of $\cV$. Since $X_3$ is an object of $\ind(\cC)$, the $U\otimes V$-submodule generated by $f(m_j)$ is an object of $\cC$; as $m_j\in M_3^{j}$, this submodule must be isomorphic to $M_3^{j}\otimes \widetilde{W}_3^{j}$ for some $V$-module $\widetilde{W}_3^{j}$ in $\cV$. We may assume $m_j$ is $L_U(0)$-homogeneous, so that $f(m_j)$ is contained in a $V$-submodule of $X_3$ that is isomorphic to $(M_3^{j})_{[\mathrm{wt}\,m_j]}\otimes\widetilde{W}_3^{j}$, which is an object of $\cV$ since $(M_3^{j})_{[\mathrm{wt}\,m_j]}$ is finite dimensional and $\cV$ is closed under finite direct sums. This shows $W_3^{j}$ is an object of $\ind(\cV)$.  

Before returning to the intertwining operator $\cY$, we need the following lemma which is a minor generalization of \cite[Proposition 13.18]{DL} and \cite[Proposition 2.11]{ADL}:
\begin{lemma}\label{lem:UV_intw_to_U_intw}
 Suppose $U$ and $V$ are vertex operator algebras, $M_1$, $M_2$, and $M_3$ are generalized $U$-modules, $W_1$, $W_2$, and $W_3$ are generalized $V$-modules, and $\cY$ is a (logarithmic) intertwining operator of type $\binom{M_3\otimes W_3}{M_1\otimes W_1\,M_2\otimes W_2}$. For $h\in\CC$, let $P_{h}: M_3\otimes W_3\rightarrow M_3\otimes(W_3)_{[h]}$ denote the projection onto the generalized eigenspace for $L_V(0)$ with generalized eigenvalue $h$. Then for any $w_1\in W_1$, $w_2\in W_2$, and $h\in\CC$,
 \begin{equation*}
  \cY_{w_1,w_2;h}=x^{-L_V(0)} P_{h}\cY(\cdot\otimes x^{L_V(0)} w_1,x)(\cdot\otimes x^{L_V(0)} w_2)
 \end{equation*}
is a generalized $U$-module intertwining operator of type $\binom{M_3\otimes (W_3)_{[h]}}{M_1\,M_2}$.
\end{lemma}
\begin{proof}
Note first that
\begin{equation*}
 x^{-L_V(0)} P_{h}=x^{-h} e^{-L_V^{nil}(0)\log x} P_{h},
\end{equation*}
where $L_V^{nil}(0)$ is the locally nilpotent part of $L_V(0)$ on $W_3$. Also, for homogeneous $w_1$ and $w_2$,
\begin{equation*}
 x^{L_V(0)} w_i = x^{\mathrm{wt}\,w_i} e^{L_V^{nil}(0)\log x} w_i=\sum_{k\geq 0}\frac{1}{k!} L_V^{nil}(0)^k w_i\,x^{\mathrm{wt}\,w_i}(\log x)^k.
\end{equation*}
These show that $\cY_{w_1,w_2;h}(m_1,x)m_2$ for any $m_1\in M_1$, $m_2\in M_2$ is a series in $x$ and $\log x$ such that the coefficient of any power of $x$ is a (finite) polynomial in $\log x$, with coefficients in $M_3\otimes(W_3)_{[h]}$. Further, $\cY_{w_1,w_2;h}$ is lower truncated because each
\begin{equation*}
 \cY(m_1\otimes L_V^{nil}(0)^{k_1} w_1,x)(m_2\otimes L_V^{nil}(0)^{k_2} w_2)
\end{equation*}
is lower truncated for the finitely many $k_1$, $k_2$ such that $L_V^{nil}(0)^{k_i}w_i\neq 0$.

The Jacobi identity for $\cY_{w_1,w_2;h}$ follows from the relations
\begin{align*}
 Y_{M_3\otimes(W_3)_{[h]}}(u, \,& x_1)  \cY_{w_1,w_2;h}(m_1,x_2)m_2\nonumber\\
 &=(Y_{M_3}(u,x_1)\otimes\Id_{(W_3)_{[h]}})x_2^{-L_V(0)}P_{h}\cY(m_1\otimes x_2^{L_V(0)}w_1,x_2)(m_2\otimes x_2^{L_V(0)}w_2)\nonumber\\
 & = x_2^{-L_V(0)} P_{h} Y_{M_3\otimes W_3}(u\otimes\vac,x_1)\cY(m_1\otimes x_2^{L_V(0)}w_1,x_2)(m_2\otimes x_2^{L_V(0)}w_2),
\end{align*}
\begin{align*}
 \cY_{w_1,w_2;h}(m_1,\, & x_2) Y_{M_2}(u,x_1)m_2\nonumber\\
 & =x_2^{-L_V(0)} P_{h}\cY(m_1\otimes x_2^{L_V(0)}w_1,x_2)(Y_{M_2}(u,x_1)m_2\otimes x_2^{L_V(0)}w_2)\nonumber\\
 & =x_2^{-L_V(0)} P_{h}\cY(m_1\otimes x_2^{L_V(0)}w_1,x_2)Y_{M_2\otimes W_2}(u\otimes \vac,x_1)(m_2\otimes x_2^{L_V(0)}w_2),
\end{align*}
and
\begin{align*}
 \cY_{w_1,w_2;h}( & Y_{M_1}(u,x_0)m_1,x_2)m_2\nonumber\\
 & =x_2^{-L_V(0)}P_{h}\cY(Y_{M_1}(u,x_0)m_1\otimes x_2^{L_V(0)}w_1,x_2)(m_2\otimes x_2^{L_V(0)}w_2)\nonumber\\
 & =x_2^{-L_V(0)}P_{h}\cY(Y_{M_1\otimes W_1}(u\otimes\vac,x_0)(m_1\otimes x_2^{L_V(0)}w_1),x_2)(m_2\otimes x_2^{L_V(0)}w_2)
\end{align*}
for $u\in U$, $m_1\in M_1$, and $m_2\in M_2$, together with the Jacobi identity for the $U\otimes V$-module intertwining operator $\cY$.

For the $L_U(-1)$-derivative property, we calculate
\begin{align*}
 \frac{d}{dx} & \cY_{w_1,w_2;h}(m_1,x)m_2 \nonumber\\
 & = x^{-L_V(0)}P_{h}\left.\frac{d}{dx}\right\vert_{(m_1\otimes x^{L_V(0)} w_1)\otimes(m_2\otimes x^{L_V(0)}w_2)}\cY(\cdot,x)\cdot \nonumber\\
 &\qquad -L_V(0) x^{-L_V(0)-1} P_{h}\cY(m_1\otimes x^{L_V(0)}w_1,x)(m_2\otimes x^{L_V(0)}w_2)\nonumber\\
 & \qquad +x^{-L_V(0)}P_{h}\cY(L_V(0)(m_1\otimes x^{L_V(0)-1}w_1),x)(m_2\otimes x^{L_V(0)}w_2)\nonumber\\
 & \qquad +x^{-L_V(0)}P_{h}\cY(m_1\otimes x^{L_V(0)}w_1,x_2)L_V(0)(m_2\otimes x^{L_V(0)-1}w_2)\nonumber\\
 & = x^{-L_V(0)}P_{h}\cY(L(-1)(m_1\otimes x^{L_V(0)}w_1),x)(m_2\otimes x^{L_V(0)}w_2)\nonumber\\
 & \qquad -x^{-L_V(0)-1}P_{h}\left(L_V(0)\cY(m_1\otimes x^{L_V(0)}w_1,x)(m_2\otimes x^{L_V(0)}w_2)\right.\nonumber\\
 & \hspace{10em} -\cY(L_V(0)(m_1\otimes x^{L_V(0)}w_1),x)(m_2\otimes x^{L_V(0)}w_2)\nonumber\\
 &\hspace{10em} \left.-\cY(m_1\otimes x^{L_V(0)}w_1,x)L_V(0)(m_2\otimes x^{L_V(0)}w_2)\right)\nonumber\\
 & =x^{-L_V(0)}P_{h}\cY(L_U(-1)m_1\otimes x^{L_V(0)}w_1,x)(m_2\otimes x^{L_V(0)}w_2)\nonumber\\
 & \qquad -x^{-L_V(0)-1}P_{h}\left(L_V(0)\cY(m_1\otimes x^{L_V(0)}w_1,x)(m_2\otimes x^{L_V(0)}w_2)\right.\nonumber\\
 & \hspace{10em} -\cY((xL_V(-1)+L_V(0))(m_1\otimes x^{L_V(0)}w_1),x)(m_2\otimes x^{L_V(0)}w_2)\nonumber\\
 &\hspace{10em} \left.-\cY(m_1\otimes x^{L_V(0)}w_1,x)L_V(0)(m_2\otimes x^{L_V(0)}w_2)\right)\nonumber\\
 &  =\cY_{w_1,w_2;h}(L_U(-1)m_1,x)m_2,
\end{align*}
where we have used the $L(-1)$-derivative property for $\cY$ and the commutator formula for $L_V(0)=(\vac\otimes\omega_V)_1$.
\end{proof}

We now apply the lemma to our intertwining operator $\cY$ of type $\binom{X_3}{M_1\otimes W_1\,M_2\otimes W_2}$, where $X_3\cong\bigoplus_{j\in J} M_3^{j}\otimes W_3^{j}$: For any $w_1\in W_1$, $w_2\in W_2$, and $h\in\CC$, we get a $U$-module intertwining operator
\begin{equation*}
 \cY_{w_1,w_2;h}: M_1\otimes M_2\rightarrow\bigg(\bigoplus_{j\in J} M_3^{j}\otimes (W_3^{j})_{[h]}\bigg)[\log x]\lbrace x\rbrace
\end{equation*}
For any $j\in J$, we compose $\cY_{w_1,w_2;h}$ with the $U$-module projection $P_{j,h}$ onto $M_3^{j}\otimes(W_3^{j})_{[h]}$. Then, for any $w_3'\in(W^{j}_3)_{[h]}^*$, we extend $w_3'$ to the $U$-homomorphism
\begin{equation*}
 \Id_{M_3^{j}}\otimes w_3': M_3^{j}\otimes(W_3^{j})_{[h]}\rightarrow M_3^{j}.
\end{equation*}
The result is an intertwining operator 
\begin{equation*}
(\Id_{M_3^{j}}\otimes w_3')\circ P_{j,h}\circ\cY_{w_1,w_2;h}: M_1\otimes M_2\rightarrow M_3^{j}[\log x]\lbrace x\rbrace
\end{equation*} 
for any $j\in\widetilde{J}$, $h\in\CC$, $w_1\in W_1$, $w_2\in W_2$, and $w_3'\in (W_3^{j})_{[h]}^*$.

\begin{lemma}\label{lem:C1-cofin_bound}
 The space of intertwining operators of type $\binom{M_3^j}{M_1\,M_2}$ is non-zero for at most finitely many $j\in J$.
\end{lemma}
\begin{proof}
 Suppose $\lbrace j_k\rbrace_{k=1}^K\subseteq J$ is such that each space of intertwining operators of type $\binom{M_3^{j_k}}{M_1\,M_2}$ is non-zero. Then there is a non-zero intertwining operator $\cY$ of type $\binom{M_3}{M_1\,M_2}$, where $M_3=\bigoplus_{k=1}^K M_3^{j_k}$. Since the simple $U$-modules $M_3^{j_k}$ are non-isomorphic, $\cY$ is a surjective intertwining operator. Thus because $M_1$ and $M_2$ are $C_1$-cofinite by assumption, \cite[Lemma 3]{Mi} shows that there is a constant $K(M_1,M_2)$ such that $\dim(M_3/C_1(M_3))\leq K(M_1,M_2)$ (in fact, the proof of \cite[Lemma 3]{Mi} shows that we may take $K(M_1,M_2)=\dim(M_1/C_1(M_1)) \dim(M_2/C_1(M_2))<\infty$).
 
 Now, $M_3$ has an $\NN$-grading $M_3=\bigoplus_{n\in\NN} M_3(n)$ such that $M_3(0)$ is the direct sum of the lowest conformal weight spaces of the $M_3^{j_k}$ and such that $C_1(M_3)\subseteq\bigoplus_{n\geq 1} M_3(n)$. This forces $K\leq K(M_1,M_2)<\infty$, proving the lemma.
\end{proof}

Now since the $M_3^{j}$ are non-isomorphic, $\im\cY$ is the direct sum of its projections to the $M_3^{j}\otimes W_3^{j}$. If the image of such a projection $P_j$ is non-zero for some $j\in J$, then for some $m_1\in M_1$, $m_2\in M_2$, $w_1\in W_1$, and $w_2\in W_2$, we have
\begin{align*}
 0 & \neq x^{-L_V(0)} P_j\circ\cY(m_1\otimes w_1,x)(m_2\otimes w_2)\nonumber\\
 & =P_j\circ x^{-L_V(0)}\cY(m_1\otimes w_1,x)(m_2\otimes w_2)\nonumber\\
 & =P_j\circ\sum_{h\in\CC} \cY_{x^{-L_V(0)}w_1,x^{-L_V(0)}w_2;h}(m_1,x)m_2\nonumber\\
 & =\sum_{h\in\CC} P_{j,h}\circ\cY_{x^{-L_V(0)}w_1,x^{-L_V(0)}w_2;h}(m_1,x)m_2.
\end{align*}
Since each summand on the right side is a (finite) series in $x$ and $\log x$ with coefficients $P_{j,h}\circ\cY_{\widetilde{w}_1,\widetilde{w}_2;h}$ for certain $\widetilde{w}_1\in W_1$, $\widetilde{w}_2\in W_2$, it follows that $P_j\circ\cY\neq 0$ only if $P_{j,h}\circ\cY_{\til{w}_1,\til{w}_2;h}\neq 0$. Then since $P_{j,h}\circ\cY_{\til{w}_1,\til{w}_2;h}\neq 0$ implies
\begin{equation*}
 (\Id_{M_3^{j}}\otimes w_3')\circ P_{j,h}\circ\cY_{\til{w}_1,\til{w}_2;h}\neq 0
\end{equation*}
for some $w_3'\in (W_3^j)_{[h]}^*$, Lemma \ref{lem:C1-cofin_bound} shows that $P_j\circ\cY\neq 0$ for finitely many $j\in J$, that is, $\im\cY$ is contained in a finite direct sum of $U\otimes V$-submodules $M_3^{j}\otimes W_3^{j}$. Since $\cC$ is closed under finite direct sums, we are reduced to showing that if $\cY$ is an intertwining operator of type $\binom{M_3\otimes W_3}{M_1\otimes W_1\,M_2\otimes W_2}$ where $M_1$, $M_2$, and $M_3$ are irreducible $U$-modules in $\cU$, $W_1$ and $W_2$ are $V$-modules in $\cV$, and $W_3$ is a generalized $V$-module in $\ind(\cV)$, then $\im\cY$ is an object of $\cC$.

We will be done if we can generalize \cite[Theorem 2.10]{ADL} to show that an intertwining operator $\cY$ of type $\binom{M_3\otimes W_3}{M_1\otimes W_1\,M_2\otimes W_2}$, as in the preceding paragraph, is a finite sum
\begin{equation*}
 \cY=\sum \cY^U_i\otimes\cY^V_i
\end{equation*}
where the $\cY^U_i$ are $U$-module intertwining operators of type $\binom{M_3}{M_1\,M_2}$ and the $\cY^V_i$ are $V$-module intertwining operators of type $\binom{W_3}{W_1\,W_2}$. For, we will then have 
\begin{equation*}
 \im\cY\subseteq\sum \im\cY^U_i\otimes\im\cY^V_i\subseteq M_3\otimes\sum\im\cY^V_i.
\end{equation*}
Since by assumption $\cV$ is closed under finite direct sums and quotients and each $\im\cY^V_i$ is an object of $\cV$, we find that $M_3\otimes\sum\im\cY^V_i$ is an object of $\cC$. But then $\cC$ is closed under submodules (because $\cU$ and $\cV$ are), so $\im\cY$ is an object of $\cC$.  So we are reduced to proving \cite[Theorem 2.10]{ADL} in the setting that $W_3$ is a generalized $V$-module in $\ind(\cV)$, rather than a grading-restricted $V$-module.

From Lemma \ref{lem:UV_intw_to_U_intw}, we have the $U$-module intertwining operator
\begin{equation*}
 \cY_{w_1,w_2;h}=x^{-L_V(0)} P_h\cY(\cdot\otimes x^{L_V(0)} w_1,x)(\cdot\otimes x^{L_V(0)} w_2)
\end{equation*}
of type $\binom{M_3\otimes (W_3)_{[h]}}{M_1\,M_2}$ for any $w_1\in W_1$, $w_2\in W_2$, and $h\in\CC$. Since $M_3\otimes(W_3)_{[h]}$ is an $\NN$-gradable weak $U$-module and $M_1$, $M_2$ are $C_1$-cofinite by assumption, \cite[Key Theorem]{Mi} implies that $\im\cY_{w_1,w_2;h}$ is $C_1$-cofinite. This means that 
\begin{equation*}
 \im\cY_{w_1,w_2;h}\subseteq M_3\otimes W^{3}_{w_1,w_2;h}
\end{equation*}
for some finite-dimensional subspace $W^{3}_{w_1,w_2;h}\subseteq (W_3)_{[h]}$. Since $W^{3}_{w_1,w_2;h}$ is finite dimensional, $U$-module intertwining operator spaces satisfy
\begin{equation*}
 \cV_{M_1,M_2}^{M_3\otimes W^{3}_{w_1,w_2;h}}\cong\cV_{M_1,M_2}^{M_3}\otimes W^{3}_{w_1,w_2;h},
\end{equation*}
so
\begin{equation*}
 \cY_{w_1,w_2;h} =\sum \cY^U_i\otimes F^{(i)}_h(w_1,w_2)
\end{equation*}
for
\begin{equation*}
 F^{(i)}_h(w_1,w_2)\in W^{3}_{w_1,w_2;h}\subseteq(W_3)_{[h]},
\end{equation*}
where $\lbrace\cY^U_i\rbrace$ is a basis for the space of intertwining operators $\cV_{M_1,M_2}^{M_3}$. Because $M_1$, $M_2$, and $M_3$ are irreducible and $C_1$-cofinite, $\cV_{M_1,M_2}^{M_3}$ is finite dimensional by \cite[Corollary 3.17]{Li_fin_prop}.

From the definitions, we can now write
\begin{equation*}
 \cY(m_1\otimes w_1,x)(m_2\otimes w_2) =\sum_{h\in\CC}\sum_i \cY^U_i(m_1,x)m_2\otimes x^{L_V(0)} F^{(i)}_h(x^{-L_V(0)} w_1,x^{-L_V(0)}w_2)
\end{equation*}
for $m_1\in M_1$, $m_2\in M_2$, $w_1\in W_1$, and $w_2\in W_2$. We will be done if we can show that
\begin{equation*}
 \cY^V_i(w_1,x)w_2=\sum_{h\in\CC} x^{L_V(0)} F^{(i)}_h(x^{-L_V(0)} w_1,x^{-L_V(0)}w_2)
\end{equation*}
defines a $V$-module intertwining operator of type $\binom{W_3}{W_1\,W_2}$ for each $i$. Since $\cY$ is an intertwining operator, Lemma \ref{lem:UV_intw_to_U_intw} (with the roles of $U$ and $V$ reversed) implies that for $m_1\in M_1$, $m_2\in M_2$, and homogeneous $m_3'\in M_3'$,
\begin{align*}
 \cY_{m_1,m_2;m_3'} & = (m_3'\otimes\Id_{W_3})\circ x^{-L_U(0)} P_{\mathrm{wt}\,m_3'}\cY(x^{L_U(0)}m_1\otimes\cdot,x)(x^{L_U(0)}m_2\otimes\cdot)\nonumber\\
 & =\sum_i \langle m_3', x^{-L_U(0)}\cY^U_i(x^{L_U(0)} m_1,x)x^{L_U(0)}m_2\rangle \cY^V_i\nonumber\\
 & =\sum_i \langle m_3',\cY^U_i(m_1,1)m_2\rangle\cY^V_i
\end{align*}
is a $V$-module intertwining operator of type $\binom{W_3}{W_1\,W_2}$. So to show that each $\cY_i^V$ is an intertwining operator, it is sufficient to invert this relation by showing each $\cY_i^V$ is a (finite) linear combination of intertwining operators $\cY_{m_1,m_2; m_3'}$.

Because $M_1$, $M_2$, and $M_3$ are irreducible, \cite[Proposition 2.10]{Li_intw_ops} shows that there is an injective linear map
\begin{equation*}
 \pi: \cV_{M_1,M_2}^{M_3}\hookrightarrow\hom_{A(U)}(A(M_1)\otimes_{A(U)} T(M_2),T(M_3)),
\end{equation*}
where $A(U)$ is the Zhu algebra of $U$, $A(M_1)$ is the $A(U)$-bimodule associated to $M_1$, and $T(M_2)$, $T(M_3)$ are the lowest conformal weight spaces of $M_2$, $M_3$. Specifically,
\begin{equation*}
 \pi(\cY): (m_1+O(M_1))\otimes_{A(U)} m_2 \mapsto o_\cY(m_1)m_2
\end{equation*}
for $m_1\in M_1$, $m_2\in T(M_2)$, where $o_\cY(m_1)$ denotes the component of $\cY(m_1,x)$ that maps $T(M_2)$ to $T(M_3)$. That is, $o_\cY(m_1)m_2=P_{T(M_3)}\cY(m_1,1)m_2$, where $P_{T(M_3)}$ is the projection onto $T(M_3)$ with respect to the conformal weight gradation of $M_3$.

Now if $S$ is any subspace that generates $A(M_1)$ as an $A(U)$-bimodule, we can continue to embed $\cV_{M_1,M_2}^{M_3}$ using injective linear maps as follows:
\begin{align*}
 \cV_{M_1,M_2}^{M_3} & \hookrightarrow\hom_{A(U)}(A(M_1)\otimes_{A(U)} T(M_2),T(M_3))\nonumber\\
 &\hookrightarrow\hom_\CC(S\otimes T(M_2),T(M_3))\hookrightarrow(T(M_3)^*\otimes S\otimes T(M_2))^*.
\end{align*}
Denoting this composition of injections by $\widetilde{\pi}$, we have
\begin{equation*}
 \langle \widetilde{\pi}(\cY), m_3'\otimes(m_1+O(M_1))\otimes m_2\rangle =\langle m_3',o_\cY(m_1)m_2\rangle =\langle m_3',\cY(m_1,1)m_2\rangle
\end{equation*}
for $m_3'\in T(M_3)^*=T(M_3')$, $m_1+O(M_1)\in S$, and $m_2\in T(M_2)$. Since the $\cY^U_i$ are linearly independent, $\widetilde{\pi}$ maps them to linearly independent functionals on $T(M_3')\otimes S\otimes T(M_2)$ which we can complete to a basis of $(T(M_3')\otimes S\otimes T(M_2))^*$. Now by \cite[Proposition 3.16]{Li_fin_prop}, we can take $S$ to be finite dimensional since $M_1$ is $C_1$-cofinite. Thus we have corresponding dual basis elements 
\begin{equation*}
 m_i\in (T(M_3')\otimes S\otimes T(M_2))^{**} =T(M_3')\otimes S\otimes T(M_2)
\end{equation*}
such that
\begin{equation*}
 \langle\widetilde{\pi}(\cY^U_j), m_i\rangle =\delta_{i,j}.
\end{equation*}
Thus for any fixed $i$, if we set $m_i=\sum_k m_{3,k}'\otimes(m_{1,k}+O(M_1))\otimes m_{2,k}$, then we get
\begin{equation*}
 \sum_k \cY_{m_{1,k},m_{2,k};m_{3,k}'} =\sum_k\sum_j \langle m_{3,k}',\cY^U_j(m_{1,k},1)m_{2,k}\rangle \cY^V_j =\sum_j \langle\widetilde{\pi}(\cY^U_j),m_i\rangle\cY^V_j=\cY^V_i.
\end{equation*}
This shows that each $\cY^V_i$ is an intertwining operator, completing the proof of Proposition \ref{prop:IndC_intw_op_cond}.

%


\begin{thebibliography}{EGNO}

\bibitem[ADL]{ADL}
T. Abe, C. Dong and H. Li, Fusion rules for the vertex operator algebras $M(1)^+$ and $V_L^+$, \textit{Comm. Math. Phys.} \textbf{253} (2005), no. 1, 171--219.

\bibitem[Ad]{Ad}
D. Adamovi\'{c}, Classification of irreducible modules of certain subalgebras of free boson vertex algebra, \textit{J. Algebra} \textbf{270} (2003), no. 1, 115--132.

\bibitem[ALM]{ALM}
D. Adamovi\'{c}, X. Lin and A. Milas, $ADE$ subalgebras of the triplet vertex algebra $\cW(p)$: $A$-series, \textit{Commun. Contemp. Math.} \textbf{15} (2013), no. 6, 1350028, 30 pp.

\bibitem[AdM]{AdM}
D. Adamovi\'{c} and A. Milas, On the triplet vertex algebra $\cW(p)$, {\em Adv. Math.} \textbf{217} (2008), no. 6, 2664--2699.

\bibitem[ALSW]{ALSW}
R. Allen, S. Lentner, C. Schweigert, and S. Wood, Duality structures for module categories of vertex operator algebras and the Feigin Fuchs boson, arXiv:2107.05718.

\bibitem[Ar1]{Ar-KRW-conj}
T. Arakawa, Representation theory of superconformal algebras and the Kac-Roan-Wakimoto conjecture, \textit{Duke Math. J.} \textbf{130} (2005), no. 3, 435--478. 

 \bibitem[Ar2]{Ar}
 T. Arakawa, Chiral algebras of class $\mathcal{S}$ and Moore-Tachikawa symplectic varieties, arXiv:1811.01577.
 
 \bibitem[ArM]{AM}
T. Arakawa and A. Moreau, Arc spaces and chiral symplectic cores, \textit{Publ. Res. Inst. Math. Sci.} \textbf{57} (2021), no. 3-4, 795--829.

\bibitem[BK]{BK} B. Bakalov and A. Kirillov, Jr., \textit{Lectures on Tensor Categories and Modular Functors}, University Lecture Series, \textbf{21}, American Mathematical Society, Providence, RI, 2001. x+221 pp.

\bibitem[BD]{BD}
M. Boyarchenko and V. Drinfeld, A duality formalism in the spirit of Grothendieck and Verdier, \textit{Quantum Topol.} \textbf{4} (2013), no. 4, 447--489.

\bibitem[CM]{CM}
S. Carnahan and M. Miyamoto, Regularity of fixed-point vertex operator algebras, arXiv:1603.05645.
 
 \bibitem[Cr]{Cr1}
 T. Creutzig, Fusion categories for affine vertex algebras at admissible levels, \textit{Selecta Math. (N.S.)} \textbf{25} (2019), no. 2, Paper No. 27, 21 pp.
 
 \bibitem[CGL]{CGL}
 T. Creutzig, N. Genra and A. Linshaw, Category $\mathcal{O}$ for vertex algebras of $\mathfrak{osp}_{1\vert 2n}$, arXiv:2203.08188.
 
 \bibitem[CHY]{CHY}
 T. Creutzig, Y.-Z. Huang and J. Yang, Braided tensor categories of admissible modules for affine Lie algebras, \textit{Comm. Math. Phys.} \textbf{362} (2018), no. 3, 827--854.
 
\bibitem[CJORY]{CJORY}
T. Creutzig, C. Jiang, F. Orosz Hunziker, D. Ridout and J. Yang, Tensor categories arising from the Virasoro algebra, \textit{Adv. Math.} \textbf{380} (2021), 107601, 35 pp.

\bibitem[CKaL]{CKL}
T. Creutzig, S. Kanade and A. Linshaw, Simple current extensions beyond semi-simplicity, \textit{Commun. Contemp. Math.} \textbf{22} (2020), no. 1, 1950001, 49 pp.

\bibitem[CKoL]{CKoL}
T. Creutzig, V. Kovalchuk and A. Linshaw, Generalized parafermions of orthogonal type, \textit{J. Algebra} \textbf{593} (2022), 178--192. 

\bibitem[CKLR]{CKLR}
T. Creutzig, S. Kanade, A. Linshaw and D. Ridout, Schur-Weyl duality for Heisenberg cosets, \textit{Transform. Groups} \textbf{24} (2019), no. 2, 301--354.

\bibitem[CKM1]{CKM1}
 T. Creutzig, S. Kanade and R. McRae, Tensor categories for vertex operator superalgebra extensions, to appear in \textit{Mem. Amer. Math. Soc.}, arXiv:1705.05017.
 
 \bibitem[CKM2]{CKM2}
 T. Creutzig, S. Kanade and R. McRae, Gluing vertex algebras, \textit{Adv. Math.} \textbf{396} (2022), Paper No. 108174, 72 pp.

 \bibitem[CMY1]{CMY}
 T. Creutzig, R. McRae and J. Yang, Direct limit completions of vertex tensor categories, \textit{Commun. Contemp. Math.} \textbf{24} (2022), no. 2, Paper No. 2150033, 60 pp. 
 
 \bibitem[CMY2]{CMY-sing}
T. Creutzig, R. McRae and J. Yang, On ribbon categories for singlet vertex algebras, \textit{Comm. Math. Phys.} \textbf{387} (2021),  no. 2, 865--925.
 
 \bibitem[CY]{CY}
T. Creutzig and J. Yang, Tensor categories of affine Lie algebras beyond admissible levels, {\em Math. Ann.} \textbf{380} (2021), no. 3-4, 1991--2040. 
 
\bibitem[DL]{DL}
C. Dong and J. Lepowsky, \textit{Generalized Vertex Algebras and Relative Vertex Operators},  Progress in Mathematics, \textbf{112}, Birkh\"{a}user Boston, Inc., Boston, MA, 1993. x+202 pp.

 
 \bibitem[EGNO]{EGNO}
 P. Etingof, S. Gelaki, D. Nikshych and V. Ostrik, \textit{Tensor Categories},  Mathematical Surveys and Monographs, \textbf{205}, American Mathematical Society, Providence, RI, 2015, xvi+343 pp.
 
 \bibitem[FHL]{FHL} I. Frenkel, Y.-Z. Huang and J. Lepowsky,
On axiomatic approaches to vertex operator algebras
and modules, \textit{Mem. Amer. Math. Soc.} \textbf{104} (1993), no. 494, viii+64 pp.

\bibitem[FLM]{FLM} I. Frenkel, J. Lepowsky and A. Meurman, \textit{Vertex Operator Algebras and the Monster}, Pure and Applied Mathematics, \textbf{134} Academic Press, Inc., Boston, MA, 1988. liv+508 pp.

\bibitem[FZ]{FZ}
I. Frenkel and Y. Zhu, Vertex operator algebras associated to representations of affine and Virasoro algebras, \textit{Duke Math. J.} \textbf{66} (1992), no. 1, 123--168.

\bibitem[GN]{GN}
T. Gannon and C. Negron, Quantum $SL(2)$ and logarithmic vertex operator algebras at $(p, 1)$-central charge, to appear in \textit{J. Eur. Math. Soc. (JEMS)}, arXiv:2104.12821.

\bibitem[Hu1]{Hu-diff-eqs}
Y.-Z. Huang, Differential equations and intertwining operators, \textit{Commun. Contemp. Math.} \textbf{7} (2005), no. 3, 375--400.

\bibitem[Hu2]{Hu-rig-mod}
Y.-Z. Huang,  Rigidity and modularity of vertex tensor categories, \textit{Commun. Contemp. Math.} \textbf{10} (2008), suppl. 1, 871--911.


\bibitem[Hu3]{Hu-suff-cond}
Y.-Z. Huang, On the applicability of logarithmic tensor category theory, arXiv:1702.00133.
 
 \bibitem[HKL]{HKL}
 Y.-Z. Huang, A. Kirillov, Jr. and J. Lepowsky, Braided tensor categories and extensions of vertex
operator algebras, \textit{Comm. Math. Phys.} \textbf{337} (2015), 1143--1159.

\bibitem[HL1]{HL1} Y.-Z. Huang and J. Lepowsky,
Tensor products of modules for a vertex operator algebra and vertex tensor categories, \textit{Lie Theory and Geometry}, 349--383, Progress in Math., \textbf{123},
Birkh\"{a}user, Boston, MA, 1994.

\bibitem[HL2]{HL2}
Y.-Z. Huang and J. Lepowsky,
Tensor categories and the mathematics of rational and logarithmic conformal field theory,
\emph{J. Phys. A} {\bf 46} (2013), 494009, 21 pp.

\bibitem[HLZ1]{HLZ1}
Y.-Z. Huang, J. Lepowsky and L. Zhang, Logarithmic tensor category theory for generalized modules for a conformal vertex algebra, I: Introduction and strongly graded algebras and their generalized modules, \textit{Conformal Field Theories and Tensor Categories}, 169--248, Math. Lect. Peking Univ., Springer, Heidelberg, 2014.
	
\bibitem[HLZ2]{HLZ2}
Y.-Z. Huang, J. Lepowsky and L. Zhang, Logarithmic tensor category theory for 
generalized modules for a conformal vertex algebra, II: Logarithmic formal 
calculus and properties of logarithmic intertwining operators, arXiv:1012.4196.
	
\bibitem[HLZ3]{HLZ3}
Y.-Z. Huang, J. Lepowsky and L. Zhang, Logarithmic tensor category theory for 
generalized modules for a conformal vertex algebra, III: Intertwining maps and 
tensor product bifunctors, arXiv:1012.4197.
	
\bibitem[HLZ4]{HLZ4}
Y.-Z. Huang, J. Lepowsky and L. Zhang, Logarithmic tensor category theory for 
generalized modules for a conformal vertex algebra, IV: Constructions of tensor 
product bifunctors and the compatibility conditions, arXiv:1012.4198.
	
\bibitem[HLZ5]{HLZ5}
Y.-Z. Huang, J. Lepowsky and L. Zhang, Logarithmic tensor category theory for 
generalized modules for a conformal vertex algebra, V: Convergence condition 
for intertwining maps and the corresponding compatibility condition, 
arXiv:1012.4199.
	
\bibitem[HLZ6]{HLZ6}
Y.-Z. Huang, J. Lepowsky and L. Zhang, Logarithmic tensor category theory for 
generalized modules for a conformal vertex algebra, VI: Expansion condition, 
associativity of logarithmic intertwining operators, and the associativity 
isomorphisms, arXiv:1012.4202.
	
\bibitem[HLZ7]{HLZ7}
Y.-Z. Huang, J. Lepowsky and L. Zhang, Logarithmic tensor category theory for 
generalized modules for a conformal vertex algebra, VII: Convergence and 
extension properties and applications to expansion for intertwining maps, 
arXiv:1110.1929.
	
\bibitem[HLZ8]{HLZ8}
Y.-Z. Huang, J. Lepowsky and L. Zhang, Logarithmic tensor category theory for 
generalized modules for a conformal vertex algebra, VIII: Braided tensor 
category structure on categories of generalized modules for a conformal vertex 
algebra, arXiv:1110.1931.

\bibitem[JL]{JL}
C. Jiang and Z. Lin, Tensor decomposition, parafermions, level-rank duality, and reciprocity law
for vertex operator algebras, \textit{Trans. Amer. Math. Soc.} \textbf{375} (2022), no. 12, 8325--8352.

\bibitem[Ka]{Ka}
C. Kassel, \textit{Quantum Groups}, Graduate Texts in Mathematics, \textbf{155}, Springer-Verlag, New York, 1995, xii+531 pp.

\bibitem[KO]{KO}
A. Kirillov, Jr. and V. Ostrik, On a $q$-analogue of the McKay correspondence and the $ADE$ classification of $\mathfrak{sl}_2$   conformal field theories, \textit{Adv. Math.}  \textbf{171}  (2002), 183--227. 

\bibitem[Le]{Le}
J. Lepowsky, Generalized Verma modules, the Cartan-Helgason theorem, and the Harish-Chandra homomorphism, \textit{J. Algebra} \textbf{49} (1977), no. 2, 470--495.

\bibitem[LL]{LL}
J. Lepowsky and H. Li, \textit{Introduction to Vertex
Operator Algebras and Their Representations}, Progress in Mathematics, \textbf{227}, Birkh\"{a}user Boston, Inc., Boston, MA, 2004. xiv+318 pp.

\bibitem[Li1]{Li_fin_prop}
H. Li, Some finiteness properties of regular vertex operator algebras, \textit{J. Algebra} \textbf{212} (1999), no. 2, 495--514.

\bibitem[Li2]{Li_intw_ops}
H. Li, Determining fusion rules by $A(V)$-modules and bimodules, \textit{J. Algebra} \textbf{212} (1999), no. 2, 515--556.

\bibitem[Lin]{Lin}
X. Lin, Mirror extensions of rational vertex operator algebras, \textit{Trans. Amer. Math. Soc.} \textbf{369}
(2017), no. 6, 3821--3840.

\bibitem[Ma]{Ma}
F. Malikov, Verma modules over Kac-Moody algebras of rank $2$, (Russian) \textit{Algebra i Analiz} \textbf{2} (1990), no. 2, 65--84; translation in \textit{Leningrad Math. J.} \textbf{2} (1991), no. 2, 269--286.

\bibitem[McR1]{McR}
R. McRae, On the tensor structure of modules for compact orbifold vertex operator algebras, \textit{Math. Z.} \textbf{296} (2020), no. 1-2, 409--452.

\bibitem[McR2]{McR-fin-ind-ss}
R. McRae, On semisimplicity of module categories for finite non-zero index vertex operator subalgebras, \textit{Lett. Math. Phys.} \textbf{112} (2022), no. 2, Paper No. 25, 28 pp.

\bibitem[MY1]{MY-intw-op} 
R. McRae and J. Yang, Vertex algebraic intertwining operators among generalized Verma modules for $\widehat{\mathfrak{sl}(2,\mathbb{C})}$, \textit{Trans. Amer. Math. Soc.} \textbf{370} (2018), no. 4, 2351--2390.

\bibitem[MY2]{MY}
R. McRae and J. Yang, Structure of Virasoro tensor categories at central charge $13-6p-6p^{-1}$ for integers $p>1$, arXiv:2011.02170.

\bibitem[MY3]{MY2}
R. McRae and J. Yang, An $\mathfrak{sl}_2$-type tensor category for the Virasoro algebra at central charge $25$ and applications, \textit{Math. Z.} \textbf{303} (2023), no. 2, Paper No. 32, 40 pp.

\bibitem[Mi1]{Mi}
M. Miyamoto, $C_1$-cofiniteness and fusion products of vertex operator algebras, \textit{Conformal Field Theories and Tensor Categories}, 271--279, Math. Lect. Peking Univ., Springer, Heidelberg, 2014.

\bibitem[Mi2]{Mi_assoc}
M. Miyamoto, Associativity of fusion products of $C_1$-cofinite $\NN$-gradable modules of vertex operator algebra, arXiv:2105.01851.

\end{thebibliography}
\end{document}